\documentclass{amsart}

\usepackage{epsfig,color}

\usepackage[latin1]{inputenc}
\usepackage[T1]{fontenc}
\usepackage{color}
\usepackage{pst-plot,pstricks}
\usepackage[brazil,english]{babel}
\usepackage{indentfirst, amsfonts, amsmath, amsthm, amssymb, amscd}
\usepackage{bbm}

\usepackage{constants} 
\usepackage[colorlinks]{hyperref}
\usepackage[colorinlistoftodos, textwidth=4cm, shadow]{todonotes}

\hypersetup{ colorlinks=true, pdftitle={Central limit theorem for the  modulus of continuity  of averages of observables in transversal families of piecewise expanding unimodal maps}, pdfsubject={linear response, dynamical systems, unimodal maps, expanding maps, ergodic theory, central limit theorem}, pdfauthor={Amanda de Lima and Daniel Smania},pdfkeywords={ linear response, dynamical systems, unimodal maps, expanding maps, ergodic theory, central limit theorem}}

\newtheorem{theorem}{Theorem}[section]
\newtheorem{lemma}[theorem]{Lemma}

\newtheorem{cor}[theorem]{Corollary}
\newtheorem{proposition}[theorem]{Proposition}

\theoremstyle{definition}
\newtheorem{remark}[theorem]{Remark}
\newtheorem{definition}[theorem]{Definition}

\newcommand{\R}{\mathbb{R}}
\newcommand{\N}{\mathbb{N}}
\newcommand{\PF}{\mathcal{L}}
\newcommand{\p}{\mathcal{P}}

\newcommand{\nivelh}{N(t,h)-\lfloor\mathcal{K}\log N(t,h) \rfloor}

\newcommand{\norm}[1]{\left|{#1}\right|}

\newcommand{\Nivel}{N-\lfloor\mathcal{K} \log N\rfloor}
\newcommand{\fc}{\mathbbm{1}}
\newcommand{\Ncerto}{N_3}

\begin{document}
\author{Amanda de Lima and Daniel Smania}

\address{
Departamento de Matem\'atica,
ICMC-USP, Caixa Postal 668,  S\~ao Carlos-SP,
CEP 13560-970
S\~ao Carlos-SP, Brazil}
\email{smania@icmc.usp.br\\
amandal@icmc.usp.br}
\urladdr{www.icmc.usp.br/$\sim$smania/}
\date{\today}
\title{Central limit theorem for the  modulus of continuity  of averages of observables on transversal families of piecewise expanding unimodal maps}
\begin{abstract}  
Consider a $C^2$  family of   mixing  $C^4$ piecewise expanding unimodal maps $t \in [a,b] \mapsto f_t$, with a critical point $c$,  that is transversal to the topological classes of such maps.  Given a Lipchitz  observable $\phi$  consider the function 
$$\mathcal{R_{\phi}}(t)= \int \phi \ d\mu_t,$$
where $\mu_t$ is the unique absolutely continuous invariant probability of $f_t$. Suppose that $\sigma_t > 0$   for every $t \in [a,b]$, where 
$$\sigma_t^2= \sigma_t^2(\phi) =\lim_{n\to \infty} \int
\left(\frac{\sum_{j=0}^{n-1}\left(\phi \circ f_t^j-\int \phi d\mu_t\right)}{\sqrt{n}}\right)^2 \, d\mu_t.$$ 

We  show that
$$
m \left\{t \in [a,b]\colon t+h \in [a,b] \text { and } \frac{1}{\Psi(t) \sqrt{-\log |h|}}\left(\frac{\mathcal{R_{\phi}}(t+h)-\mathcal{R_{\phi}}(t)}{h}\right) \le y\right\}$$
converges to 
$$\frac{1}{\sqrt{2\pi}} \int_{-\infty}^{y} e^{-\frac{s^2}{2}} ds, $$
 where $\Psi(t)$ is a dynamically defined function and $m$ is the  Lebesgue measure  on $[a,b]$, normalized in such way that $m([a,b])=1$. As a consequence we show that $\mathcal{R_{\phi}}$ is not a Lipchitz function on any subset of $[a,b]$ with positive Lebesgue measure. 
\end{abstract}

\thanks{A.L. was partially supported by FAPESP 2010/17419-6 and D.S. was partially supported by CNPq 305537/2012-1}
\maketitle
\markright{CLT for the  modulus of continuity  of averages of observables in transversal families}
\setcounter{tocdepth}{1}
\tableofcontents


\section{Introduction and statement of the main  results}
Let $f_t$ be a smooth family of (piecewise) smooth maps on a manifold $M$,  and let us suppose that for each $f_t$ there is a physical (or SBR) probability $\mu_t$ on $M$.  Given an observable $\phi: M \to \R$, we can ask if the function 
$$
\begin{array}{cccl}
\mathcal{R_{\phi}} \; : & \! [0,1] & \! \longrightarrow & \! \R\\
& \! t & \! \longmapsto & \! \int \phi \ d\mu_t
\end{array}
$$
is differentiable and if we can find an explicit formula for its derivative.  The study of this question  is the so called  {\it linear response problem}.

D. Ruelle showed that $\mathcal{R_{\phi}}$ is differentiable and also gave the formula for $\mathcal{R_{\phi}}'$, in the case of smooth uniformly hyperbolic dynamical systems (see Ruelle in \cite{ruelle1} and \cite{ruelle2}, and Baladi and Smania in \cite{bs3} for more details).

In the setting of smooth families of piecewise expanding unimodal maps,   Baladi and Smania (see \cite{bs1}) proved that if we have a $C^2$ family of piecewise expanding unimodal maps of class $C^3$, then $\mathcal{R_{\phi}}$   is differentiable in $t_0$, with $\phi\in C^{1+ Lip}$,   {\it provided} that the family  $f_t$ is tangent to the topological class of $f_{t_0}$ at $t=t_0$. It turns out that the family $s\mapsto f_s$ is tangent to the topological class of $f_t$ at the parameter $t$  if and only if 
$$
J(f_t,v_t)= \sum_{k=0}^{M_t-1}\frac{v_t(f_t^k(c))}{Df_t^k(f_t(c))}=0,
$$
where $v_t =\partial_sf_s|_{s=t}$ and $M_t$ is either the period of the critical point $c$ if $c$ is periodic, or $\infty$, otherwise (see \cite{bs2}). Now, let us consider a $C^2$ family of piecewise expanding unimodal maps of class $C^4$ that is {\it transversal} to the topological classes of piecewise unimodal maps, that is 
\begin{equation}\label{eq_funcional}
J(f_t,v_t)= \sum_{k=0}^{M_t-1}\frac{v_t(f_t^k(c))}{Df_t^k(f_t(c))}\neq 0
\end{equation}
for every $t$.  

Baladi and Smania, \cite{bs1} and \cite{bs1b}, proved that $\mathcal{R_{\phi}}$ is not differentiable, for most of the parameters $t$, even if $\phi$ is quite regular. One can ask what is the regularity of the function $\mathcal{R_{\phi}}$ in this case. We know from Keller   \cite{KL3} (see also  Mazzolena \cite{mazzo}) that $\mathcal{R_{\phi}}$ has modulus of continuity $|h|(\log (1/|h|) + 1)$. 

We will show the  Central Limit Theorem for the modulus of continuity of the function $\mathcal{R_{\phi}}$ where $\phi$ is a lipschitzian observable. Let  
$$\sigma_t^2= \sigma_t^2(\phi) =\lim_{n\to \infty} \int
\left(\frac{\sum_{j=0}^{n-1}\left(\phi \circ f_t^j-\int \phi d\mu_t\right)}{\sqrt{n}}\right)^2 \, d\mu_t \neq 0.$$

Let  $t \mapsto f_t$ be a $C^2$ family of $C^4$ piecewise expanding unimodal maps. Note that each $f_t$ has a unique absolutely continuous invariant probability $\mu_t = \rho_t m$, where its density $\rho_t$ has bounded variation. Let 
\begin{equation}\label{defLt}
L_t = \int \log |Df_t| \ d\mu_t > 0, \  \ell_t= \frac{1}{\sqrt{L_t}}.
\end{equation}

Indeed $\rho_t$ is continuous except on the forward orbit $f^j_t(c)$  of the critical point  (see Baladi \cite{baladi}). Let $S_t $ be the jump of $\rho_t$ at the critical value, that is
\begin{equation}\label{defSt}
S_t = \lim_{x\rightarrow f_t(c)^-} \rho_t(x)- \lim_{x\rightarrow f_t(c)^+} \rho_t(x) =  \lim_{x\rightarrow f_t(c)^-} \rho_t(x) >0.
\end{equation}

\begin{theorem}\label{main} Let 
$$t \in [a,b] \mapsto f_t,$$ be a transversal $C^2$ family of  mixing  $C^4$ piecewise expanding unimodal maps
$$f_t \colon [0,1]\rightarrow [0,1].$$
 If  $\phi$ is a lipschitzian observable satisfying  $\sigma_t \neq 0$ for every $t\in [a,b]$, then for every $ y \in \mathbb{R}$
\begin{equation} \label{cltp}
\lim_{h\to 0} m \left\{t \in [a,b]\colon  t+h \in [a,b] \text{ and } \frac{1}{\Psi(t)\sqrt{-\log |h|}}\left(\frac{\mathcal{R_{\phi}}(t+h)-\mathcal{R_{\phi}}(t)}{h}\right) \le y\right\} 
\end{equation}
converges to 
$$\frac{1}{\sqrt{2\pi}} \int_{-\infty}^{y} e^{-\frac{s^2}{2}} ds,$$
where
$$
\Psi(t)=\sigma_t S_t J_t \ell_t.
$$
and $m$ is the Lebesgue measure normalized in such way that $m([a,b])=1$. 
\end{theorem}

\begin{cor} \label{cor_main} Under the same assumptions above, the function $\mathcal{R_{\phi}}$ is not a lipschitzian function on any subset of $[a,b]$ with positive Lebesgue measure. 
\end{cor}

The proof of Corollary \ref{cor_main}  will be given  in the last section as a consequence of a stronger result (Corollary \ref{cor_main2}).


\section{Families of piecewise expanding unimodal maps} \label{def_fam}
We begin this section by setting the one-parameter family of piecewise expanding unimodal maps.
\begin{definition}\label{def_peum}\index{piecewise expanding unimodal map}
A piecewise expanding $C^r$ unimodal map $f:[0,1] \to [0,1]$ is a continuous map with a critical point $c \in (0,1)$, $f(0)=f(1)=0$ and  such that $f|_{[0,c]}$ and $f|_{[c,1]}$ are $C^r$ and
$$
\norm{\frac{1}{Df}}_{\infty}< 1.
$$
\end{definition}

We say that $f$ is mixing if $f$ is topologically mixing on the interval $[f^2(c), f(c)]$. For instance,  if 
$$\inf_x |Df(x)| > \sqrt{2}$$
then  $f$ is not renormalisable. In particular $f$ is topologically mixing  on $[f^2(c),f(c)]$.

We can see the set of all $C^r$ piecewise expanding unimodal maps that share the same critical point $c \in (0,1)$ as a convex subset of the affine subspace $\{ f \in B^r\colon \ f(0)=f(1)\}$ of the Banach space $B^r$ of all continuous functions $f\colon [0,1]\rightarrow \mathbb{R}$ that are $C^r$ on the intervals $[0,c]$ and $[c,1]$, with the norm
$$|f|_r = |f|_\infty + |f|_{[0,c]}|_{C^r} + |f|_{[c,1]}|_{C^r}.$$

Let $f_t:[0,1] \to [0,1]$, $t \in [a,b]$ be a one-parameter family of piecewise expanding $C^4$ unimodal maps. We assume 
\begin{enumerate}
\item For all $t \in [a,b]$ the critical point of $f_t$ is $c$.
\item The maps $f_t$ are uniformly expanding, that is, there exist constants $1< \lambda \le \Lambda < \infty$ such that for all $t \in [a,b]$,
$$
\norm{\frac{1}{Df_t}}_{\infty}< \frac{1}{\lambda} \mbox{\;\;\; and \;\;\;} \norm{Df_t}_{\infty} < \Lambda.
$$
\item The map
$$t \in [a,b] \mapsto f_t \in B^4$$
is of class $C^2$.
\end{enumerate}

Each $f_t$ admits a unique absolutely continuous invariant probability measure $\mu_t$ and its density $\rho_t$ has bounded variation (see \cite{lasotayorke}).
By Keller  (see \cite{KL3}),
\begin{equation}\label{densidade_L1}
\norm{\rho_{t+h}-\rho_t}_{L^1} \le C |h| (\log\frac{1}{|h|}+1).
\end{equation}

\section{Good transversal families}

 It turns out that we can cut the parameter interval of a transversal family  $f_t$ in smaller intervals in such way that the family, when restricted to each one of those intervals satisfies  stronger assumptions. Here, we introduce the notation of partitions following Schnellmann in \cite{sch2}. Let us denote by $K(t)=[f^2_t(c),f_t(c)]$ the support of $\rho_t$.

Let $\p_j(t)$, $j \ge 1$ be the partition on the dynamical interval composed by the maximal open intervals of smooth monotonicity for the map $f_t^j : K(t) \to K(t)$, where $t$ is a fixed parameter value. Therefore, $\p_j(t)$ is the set of open intervals $\omega \subset K(t)$ such that $f_t^j : \omega \to K(t)$ is $C^4$ and $\omega$ is maximal.

We can also define analogous partitions on the parameter interval $[a,b]$. Let 
$$
\begin{array}{cccl}
x_0 \; : & \! [a,b] & \! \longrightarrow & \! [0,1]\\
& \! t & \! \longmapsto & \! f_t(c)
\end{array}
$$
be a $C^2$ map from the parameter interval into the dynamical interval.
We will denote by
$$
x_j(t):=f_t^j(x_0(t)),
$$
$ j\ge0$, the orbit of the point $x_0(t)$ under the map $f_t$.

Consider a interval $J \subset [a,b]$. Let us denote by $\p_j|J$, $j \ge 1$, the partition on the parameter interval composed by all open intervals $\omega$ in $J$ such that $x_i(t) \neq c$, for all $i$ satisfying $0 \le i < j$, that is 
$$f_t^i(x_0(t)) = f_t^{i+1}(c) \neq c,$$ 
for all $t \in \omega$, and such that $\omega$ is maximal, that is, if $s \in \partial \omega$, then there exists $0 \le i < j$ such that $x_i(s) =c$.

The intervals $\omega \in \p_j$ are also called cylinders.

We quote almost verbatim the definition of the Banach spaces $V_\alpha$ given in \cite{sch2}. The spaces $V_\alpha$ were introduced by Keller \cite{kellerg}. Let $m$ be the Lesbegue measure on the interval $[0,1]$

\begin{definition}{\bf (Banach space $V_{\alpha}$)} \label{bvalpha}
For every  $\psi\colon [0,1]\rightarrow \mathbb{R}$ be a function in $L^1(m)$ and $\gamma > 0$, we can define
$$
\mbox{osc\;}(\psi, \gamma, x) = \mbox{ess\,sup\;} \psi|_{(x-\gamma,x+\gamma)}- \mbox{\;ess\;inf\;}\psi|_{(x-\gamma,x+\gamma)}.
$$
Given  $A>0$ and $0<\alpha \le 1$ denote
$$
|\psi|_{\alpha} = \sup_{0<\gamma \le A} \frac{1}{\gamma^{\alpha}} \int_0^1 \mbox{osc\;}(\psi, \gamma, x)dx.
$$
The Banach space  $V_{\alpha}$ is the set of all $\psi\in L^1(m)$ such that $|\psi|_{\alpha} < \infty$, endowed with the norm
$$
|\!|\psi|\!|_{\alpha}=|\psi|_{\alpha} +\norm{\psi}_{L^1}. 
$$
\end{definition}
We quote almost verbatim the definition of the  almost sure invariant principle  given in \cite{sch2}.
\begin{definition}
Given a sequence of functions  $\xi_i$ on a probability  space, we say that  it   satisfies the {\bf almost sure invariance principle (ASIP)}, with exponent $\kappa < 1/2$   if one can construct a new probability space that has a sequence of functions $\sigma_i$, $i \ge 1$ and a representation of the Weiner process $W$  satisfying 
\begin{itemize}
\item We have
$$
\left| W(n) - \sum_{i=1}^{n} \sigma_i \right| = O(n^{\kappa}),
$$
almost surely as $n \to \infty$.
\item The sequences $\{\sigma_i\}_{i\ge 1}$ and $\{\xi_i\}_{i\ge 1}$ have identical distributions.
\end{itemize}
\end{definition}

A piecewise expanding $C^r$ unimodal map $f$ is {\it good} if either $c$ is not a periodic point of $f$ or 
$$\liminf_{x\rightarrow c} |Df^{p}(x)|\ge2.$$ where $p\ge 2$ is the prime period of $c$ (see \cite{bs1} and \cite{bs2} for more details).

\begin{definition} \label{goodf} \index{good transversal family}
A $C^2$ transversal (see equation (\ref{eq_funcional})) family of good mixing $C^4$ piecewise expanding unimodal maps $f_t$, $t \in [c,d]$ is a  {\bf good transversal family} if we can extend this family to a $C^2$ transversal family of good mixing $C^4$ piecewise expanding unimodal maps $f_t$, $t \in [c-\delta , d+\delta ]$, for some $\delta > 0$, with the following properties  
\begin{itemize}
\item[(I)]  There exists $j_0 > 0$ with the following property.  For every $t \in [c,d]$ and for each $j \ge j_0$ there exists a neighborhood $V$ of $t$ such that   for all $t' \in V\backslash \{t\}$ and all $0 <  i < j$, we have $f_{t'}^i(c) \neq c$. In particular the 
one-sided limits 
$$ \lim_{t'\to t ^+}\frac{\partial_{t'}f_{t'}^j(c)}{Df_{t'}^{j-1}(f_{t'}(c))} \text{\; \;and\; }  \lim_{t'\to t ^-}\frac{\partial_{t'}f_{t'}^j(c)}{Df_{t'}^{j-1}(f_{t'}(c))}$$
exist for every $j \ge j_0$, and there is $C \ge 1$ so that
\begin{equation}\label{const_t}
\frac{1}{C}\le\left| \lim_{t'\to t ^+}\frac{\partial_{t'}f_{t'}^j(c)}{Df_{t'}^{j-1}(f_{t'}(c))}\right| \le C,
\end{equation}
and
\begin{equation} \label{const_t2}
\frac{1}{C}\le\left| \lim_{t'\to t ^-}\frac{\partial_{t'}f_{t'}^j(c)}{Df_{t'}^{j-1}(f_{t'}(c))}\right| \le C,
\end{equation}
for all $j\ge j_0$ and $t \in [c-\delta , d+\delta ]$.
\end{itemize}
 
\begin{itemize}
\item[(II)] 
The  map $f_t$ is mixing and there are constants $\delta>0$, $L\ge 1$ and $0<\tilde{\beta}< 1$ such that for all $\psi \in V_\alpha$
\begin{equation}\label{eq_LY}
\left|\!\left|\PF_t^n \psi\right|\!\right|_{\alpha} \le L\tilde{\beta}^n\norm{\psi}_{\alpha} + L\norm{\psi}_{L^1},
\end{equation}
for all $t \in [c-\delta , d+\delta ]$. Here  $\mathcal{L}_t$ is  the Ruelle-Perron-Frobenious operator of  $f_t$ given by 
$$(\mathcal{L}_t\psi)(x)= \sum_{f_t(y)=x} \frac{1}{|Df_t(y)|} \psi(y).$$
\end{itemize}

\begin{itemize}
\item[(III)] 
There is $\delta > 0$ such that for every $\zeta> 0$ there is a constant $\tilde{C}$ satisfying
$$
\sum_{\omega \in \p_{n}|[a-\delta , b+\delta ]} \frac{1}{\norm{x_n'|_{\omega}}_{\infty}} \le \tilde{C}e^{n^{\zeta}}
$$
for all $n\ge 1$.
\end{itemize}

\begin{itemize}
\item[(IV)] For all $\varphi \in V_{\alpha}$ such that $\sigma_t(\varphi)>0$ the functions $\xi_i:[c-\delta , d+\delta ]\to \R$ $i\ge 1$, defined by
$$
\xi_i(t) = \frac{1}{\sigma_t(\varphi)}\left(\varphi(f_t^{i+1}(c))- \int \varphi d\mu_t\right)
$$
satisfy the ASIP for every exponent $\gamma>2/5$. 
\end{itemize}

\begin{itemize}
\item[(V)]  There are positive constants $\tilde{C}_1, \tilde{C}_2, \tilde{C}_3, \tilde{C}_4, \tilde{C}_5, \tilde{C}_6 $ and $\beta \in (0,1)$ such that  for every $t \in[c-\delta , d+\delta ]$ and its respective  density $\rho_t$ of the unique absolutely continuous invariant  probability of $f_t$
\begin{itemize}
\item[A$_1$.] The Perron-Frobenious operator $\mathcal{L}_t$ satisfies the Lasota-Yorke inequality in the space of bounded variation functions
$$|\mathcal{L}^k_t\phi |_{BV} \leq \tilde{C}_6 \beta^k |\phi|_{BV} + \tilde{C}_5 |\phi|_{L^1(m)}.$$
\item[A$_2$.] We have $\rho_t \in BV$ and $|\rho_t|_{BV}< \tilde{C}_1$.
\item[A$_3$.]  We have  $\rho'_t \in BV$ and  $|\rho'_t|_{BV}< \tilde{C}_2$. Moreover
$$\rho_t(x) = \int_0^x \rho'_t(u) \ du + \sum_{k=1}^{M_t-1} s_k(t) H_{f^k_t(c)}(x), $$
where $H_a(x)=0$ if $x<a$ and $H_a(x)=-1$ if $x\ge a$,
\begin{equation}\label{defs1}
s_1(t) =\frac{\rho_t(c)}{|Df_t(c-)|}+ \frac{\rho_t(c)}{|Df_t(c+)|}
\end{equation}
and
$$s_k(t) =\frac{s_1(t)}{Df_t^{k-1}(f_t(c))}.$$
Note that $S_t=s_1(t)$. 
\item[A$_4$.]  We have $\rho''_t \in BV$ and  $|\rho''_t|_{BV}< \tilde{C}_3$. Moreover
$$\rho'_t(x) = \int_0^x \rho_t''(u) \ du + \sum_{k=1}^{M_t-1} s_k'(t) H_{f^k_t(c)}(x), $$
where 
$$|s_k'(t)|\leq \frac{\tilde{C}_4}{|Df_t^{k-1}(f_t(c))|}.$$
\end{itemize}
\end{itemize}

\begin{itemize}
\item[(VI)] Let $j_0>0$ be the constant given by condition (I). For all $i,j$ satisfying $0\le i,j \le j_0$ and $t \in [c,d]$, such that $t+h \in [c-\delta, d + \delta]$ we have $$c\notin I_{i,j}(t,h),$$
where $I_{i,j}(t,h)$ is the smallest interval that contains the set
$$\{f_{t+h}^{i+j+1}(c),f_{t}^{i+j+1}(c),f_{t+h}^{i} \circ f_{t}^{j+1}(c), f_{t+h}^{i+1} \circ f_{t}^{j}(c)\}.$$
\end{itemize}

\end{definition}

\begin{remark} Conditions (I), (II) and (III) are exactly those that appears in Schnellmann \cite{sch2}, with obvious cosmetic modifications.\end{remark}

\begin{remark} If  $f_t$ is a good transversal family then of course  Eq. (\ref{cltp}) converges to
$$\frac{1}{\sqrt{2\pi}} \int_{-\infty}^{y} e^{-\frac{s^2}{2}} ds$$
if and only if 
\begin{equation} \label{clt2}
\lim_{h\to 0} m \left\{t \in [a,b]\colon   \frac{1}{\Psi(t)\sqrt{-\log |h|}}\left(\frac{\mathcal{R_{\phi}}(t+h)-\mathcal{R_{\phi}}(t)}{h}\right) \le y\right\} 
\end{equation}
converges to it as well. 
\end{remark}

\begin{proposition} \label{cut_good}  Let $f_t$, $t\in [a,b]$, be a transversal $C^2$  family of mixing $C^4$ piecewise expanding unimodal maps. Then there is a countable family of intervals  $[c_i,d_i] \subset [a,b]$, $i \in \Delta \subset \mathbb{N}$,  with pairwise disjoint interior and  
$$m([a,b]\setminus \bigcup_{i \in \Delta} [c_i,d_i])=0,$$
such that $f_t$ is a good transversal family on each $[c_i,d_i]$, $i\in \Delta$.
\end{proposition}
\begin{proof} 
Since $f_t$ is transversal, there is just a countable subset  $Q$ of parameters where $f_t$ has a periodic critical point. It is easy to see that the  subset $Q' \subset Q$  of parameters $t$ such that $f_t$ is not good and it has  a periodic critical point  is finite,  so without loss of generality we can assume that all maps $f_t$ are good.     Consider $\Omega = [a,b]\setminus (Q\cup \{a,b\})$. It  follows from the analysis in the proof of \cite[Theorem 4.1]{bs3} and  \cite[Proposition 3.3]{baladi} that for every $t' \in \Omega$ there  exists $\epsilon_1=\epsilon_1(t')$ such that if $[c,d] \subset (t'-\epsilon_1, t'+ \epsilon_1)$ then the family $f_t$ restricted to  $[c,d]$ satisfies condition $(V)$.  By Schnellmann \cite{sch2} for every $t' \in \Omega$  there  exists $\epsilon_2=\epsilon_2(t')$ such that if $[c,d] \subset (t'-\epsilon_2, t'+ \epsilon_2)$ then the family $f_t$ restricted to  $[c,d]$ satisfies conditions $(I)$, $(II)$, $(III)$ and $(IV)$.  

\noindent We claim that for every $t'\in \Omega$ there is $\epsilon_3=\epsilon_3(t')$ such that if $[c, d] \subset (t'-\epsilon_3, t'+ \epsilon_3)$ and $\delta > 0$ is small enough  then  the family $f_t$, with $t \in [c,d]$, satisfies condition  $(VI)$. Indeed, since $c$ is not a periodic point of $f_{t'}$,  there is $\epsilon_3(t')>0$ such that 
\begin{equation}\label{eq_cont}
\eta := \min\ \{|f_t^{i+j+1}(c) - c| : 0\le j \le j_0 \mbox{\; and \;} 0< i \le j_0, t \in  (t'-\epsilon_3, t'+ \epsilon_3)\} >0,
\end{equation}
Since $t \in [t'-\epsilon_3/2, t'+ \epsilon_3/2] \mapsto f_t$ is a $C^2$ family the map
$$
(t,h) \mapsto f_{t+h}^i(f_t^j(c)) 
$$
is continuous for every $0<i\le j_0$ and every  $j$ satisfying $0\le j \le j_0$. Therefore  there is $\gamma_1:=\gamma_1(i,j) < \epsilon_3/2$ such that, if $|h|\le \gamma_1$ and $t \in [t'-\epsilon_3/2, t'+ \epsilon_3/2]$, then 
$$
|f_{t+h}^{i+1}(f_t^{j}(c))-f_{t}^{i+1}(f_t^{j}(c))| \le \eta,
$$
and
$$
|f_{t+h}^{i}(f_t^{j+1}(c))-f_{t}^{i}(f_t^{j+1}(c))| \le \eta,
$$
for all $0\le j \le j_0$ and $0< i \le j_0$. Let $\gamma := \min\{\gamma_1(i,j) : 0\le j \le j_0 \mbox{\; and \;} 0< i \le j_0\}$.  In particular  if $|h|\le \gamma_1$ and $t \in [t'-\epsilon_3/2, t'+ \epsilon_3/2]$ then 
 $c \notin I_{i,j}(t,h)$ for all $0\le j \le j_0$, $0< i \le j_0$.

\noindent Let $\epsilon_4(t')=\min \{\epsilon_1(t'),  \epsilon_2(t'),  \gamma\}.$ Consider the family $\mathcal{F}$ of intervals $[c,d]\subset [a,b]$ such that $[c,d] \subset (t'-\epsilon_4(t'), t'+ \epsilon_4(t'))$ for some $t'\in \Omega$. By the Vitali's covering theorem there exists a  countable family of intervals  $[c_i,d_i] \subset [a,b]$, $[c_i,d_i]\in \mathcal{F}$, $i \in \Delta \subset \mathbb{N}$,  with pairwise disjoint interior and  
$$m([a,b]\setminus \bigcup_{i \in \Delta} [c_i,d_i])=m(\Omega \setminus \bigcup_{i \in \Delta} [c_i,d_i])=0.$$
\end{proof}

We will also need

\begin{lemma}\label{inf}  Let 
$$t \in [a,b] \mapsto f_t$$ be a good transversal $C^2$ family of good and mixing  $C^4$ piecewise expanding unimodal maps
$$f_t \colon [0,1]\rightarrow [0,1].$$
 If  $\phi$ is a lipschitzian observable satisfying  $\sigma_t \neq 0$ for every $t\in [a,b]$ then
$$\underline{J}= \inf_{t\in[a,b]} |J(f_t,v_t)|,\;\; \underline{\sigma} = \inf_{t\in[a,b]} \sigma_t(\phi),\;\; \underline{s}=\inf_{t\in [a,b]} S_t, \ \underline{\ell} =  \inf_{t\in [a,b]} \ell_t,$$
are positive, where $S_t$ and $\ell_t$ are as defined in Eqs. (\ref{defSt}) and (\ref{defLt}) respectively. Moreover $J(f_t,v_t)$ does not changes signs for $t \in [a,b]$.  In particular the function 
$$t \rightarrow \Psi(t)=\sigma_t S_t J_t \ell_t$$
does not change signs for $t \in [a,b]$ and satisfies 
$$\inf_t |\Psi(t)| > 0.$$
\end{lemma} 
\begin{proof} The function
$$t \mapsto J(f_t,v_t)$$
is {\it not} continuous in a transversal family (see \cite{bs2}). Indeed, its points of discontinuity lie on the parameters $t$ where the critical point $c$ is periodic for $f_t$, where this function have one-sided limits.  However, in \cite{bs2},  Baladi and Smania showed that if $v_n$ converges to $v$ and $f_n$ converges to $f$, then if  $J(f_n,v_n)\to0$ when $n\to \infty$ we have  $J(f,v)=0$ and  if $J(f,v)\neq 0$ then $J(f_n,v_n)$ has the same sign that $J(f,v)$ for $n$ large. From this it follows that $\underline{J} > 0$ and that $J(f_t,v_t)$ does not changes signs for $t \in [a,b]$. In \cite{sch2}, Schnellmann proved that $t \mapsto \sigma_t$ is H\"older continuous. Therefore, $\underline{\sigma} > 0$.  Note that $S_t=s_1(t) > 0$ everywhere, where $s_1$ is as defined in Eq. (\ref{defs1}). Suppose that $\lim_n s_1(t_n)=0$. Remember that (see \cite{bs1} and \cite{baladi}), 
\begin{equation}\label{comp1} \rho_{t_n} = \rho_{abs, t_n} + \rho_{sal,t_n} = \rho_{abs,t_n} + \sum_{k=1}^{M_{t_n}-1} \frac{s_1(t_n)}{Df^{k-1}_{t_n}(f_{t_n}(c))} H_{f^k_{t_n}(c)}\end{equation}
where $\rho_{abs, t_n}$ is absolutely  continuous, $(\rho_{abs,t_n})'$ has bounded variation and 
\begin{equation}\label{comp2} |(\rho_{abs,t_n})'|_{BV}\leq C.\end{equation} 
Taking a subsequence, if necessary, we can assume that $\lim_n t_n=t$  and that $\rho_{t_n}$  converges in $L^1(m)$  to $\rho_{ t}$. But if $\lim_n s_1(t_n)=0$ then by Eqs. (\ref{comp1}) and (\ref{comp2}) we conclude that $\rho_{t}$ is a continuous function. But  this is absurd since $s_1(t)\neq 0$ for every $t$. 
\end{proof}

\begin{remark}
As an example, we have the family of tent maps defined by
 \begin{equation*}
  f_t(x)= \left\{
    \begin{array}{ll}
      tx, & \text{if } x < 1/2,\\
      t-tx, & \text{if } x \ge 1/2,\\
     \end{array} \right.
\end{equation*}
$t \in (1,2)$. Tsujii \cite{Tsujii} show that the family of tent maps satisfies $$J(f_{t_0}, \partial_t f_t |_{t=t_0})\neq 0$$ at every  parameter $t_0$ where $f_{t_0}$ has  a periodic turning point. So the restriction of this family to a small neighborhood of such parameter $t_0$ is a  transversal family. We can observe that, since $f_t$ is a piecewise linear map for all $t$, the density $\rho_t$ is purely a saltus function.
\end{remark}


\section{Decomposition of the Newton quotient  for  good families}

In this section we will assume that $f_t$ is a good family.  In order to prove Theorem \ref{main} we will decompose the quotient
$$
\frac{\mathcal{R_{\phi}}(t+h) - \mathcal{R_{\phi}}(t)}{h}
$$
in two parts which  will be called the {\it Wild part}  and the {\it Tame part} of the decomposition.

\begin{definition}\label{def_projection}
Let $g: [0,1] \to \R$ be a function of bounded variation and $t\in [a,b]$. We define the projection
$$
\begin{array}{cccl}
\Pi_t \; : & \! BV & \! \longrightarrow & \! BV\\
& \! g & \! \longmapsto & \! g-\rho_t\int g dm.
\end{array}
$$
\end{definition}

Indeed  $\Pi_t$ is also a well defined operator in $L^1(m)$ and $$\sup_ t|\Pi_t|_{BV} < \infty   \ and \ \sup_ t|\Pi_t|_{L^1(m)}< \infty. $$ 
A function $g \in L^1(m)$ belongs to $\Pi_t(BV)$ if and only if $\int g \ dm =0$. In particular  the operator $(I-\PF_{t})^{-1}$ is well defined on $\Pi_t(BV)$.We are going to use the following observation quite often. If  $\int g \ dm =0$, and 
$$g=\sum_{i=0}^{\infty} g_i,$$
with $g_i  \in BV$ and the convergence of the series is in the BV norm, then
$$
(I-\PF_{t})^{-1}g = \sum_{i=0}^\infty (I-\PF_{t})^{-1}\Pi_t(g_i).
$$
Note also that
$$\Pi_t\circ \PF_{t}= \PF_{t}\circ \Pi_t.$$

\begin{proposition}\label{parteA}
Assume that $f_t$ is a family of piecewise expanding unimodal maps as defined in section \ref{def_fam} and let $\PF_{t}$ be the Perron-Frobenius operator. Then
$$
\frac{\rho_{t+h}-\rho_{t}}{h}= (I-\PF_{t+h})^{-1}\left(\frac{\PF_{t+h}(\rho_t)-\PF_{t}(\rho_t)}{h}\right).
$$
\end{proposition}
\proof
\noindent Note that $(I-\PF_{t})^{-1}$ is well defined in $\Pi_t(BV)$ and is given by
$$
(I-\PF_{t})^{-1}(\rho) = \sum_{i=0}^{\infty}\PF_{t}^i (\rho),
$$  
for every $\rho \in \Pi_t(BV)$.
Therefore, the result follows as an immediate consequence of the identity
$$
(I-\PF_{t+h})(\rho_{t+h}-\rho_t)=(I-\PF_{t+h})(\rho_t) - (I-\PF_{t})(\rho_t).
$$

\endproof

\begin{proposition}\label{parteB} 
Let $f_t$ be a $C^2$ family of good mixing  $C^4$ piecewise expanding unimodal maps that satisfies property (V) in Definition \ref{goodf}. There exists $C > 0$ with the following property. For every $t \in [a,b]$ such that the critical point of $f_t$ is not periodic, we can decompose
$$
\frac{\PF_{t+h}(\rho_t)-\PF_{t}(\rho_t)}{h} = \Phi_h + r_h
$$
where 
\begin{align*}
\Phi_h =& \frac{1}{h}\sum_{k=0}^{\infty}s_{k+1}(t)\Pi_{t+h}\left(H_{f_{t+h}(f_t^k(c))}- H_{f_{t}(f_t^k(c))}\right)
\end{align*}
and $r_h$ satifies
$$\int r_h dm =0 \mbox{\;\;\; and\;\;\;} \sup_{h\neq 0}\norm{r_h}_{BV}< C.$$
\end{proposition}

We will prove Proposition \ref{parteB} in Section \ref{sec_tame}. We will call $\mathcal{W}(t,h)=(I-\PF_{t+h})^{-1}\Phi_h$ the {\it Wild part}  and $(I-\PF_{t+h})^{-1} r_h$ will be called the {\it Tame  part } of the decomposition.  Note that 

$$
\frac{\mathcal{R}_{\phi}(t+h)-\mathcal{R}_{\phi}(t)}{h}= \int \phi \mathcal{W}(t,h)\ dm + \int \phi (I-\mathcal{L}_{t+h})^{-1}r_h \ dm.
$$

\begin{definition}Given  $h\neq 0$ and $t \in [0,1]$, let $N:=N(t,h)$ be the unique integer such that
\begin{equation}\label{defNth}
\frac{1}{|Df_t^{N+1}(f_t(c))|}\le |h| <  \frac{1}{|Df_t^{N}(f_t(c))|}.
\end{equation}
There is some ambiguity in the definition of $N(t,h)$  when $f_t^k(c)=c$ for some $k >0$. But since the family is transversal, there exists just a countable number of such parameters (see \cite{bs2}).
\end{definition} 

Given $a \in \mathbb{R}$ define 
$$ \lfloor  a \rfloor= \max \{ k \in \mathbb{Z}\colon \ a \geq k\}.$$ 
The following proposition gives us a control on the orbit of the critical point.

\begin{proposition}\label{main_prop}  
For large $\mathcal{K} > 0$ and  every $\gamma > 0$ there exists $\delta > 0$ such that for every  small  $h_0$ there are sets  $\Gamma^\delta_{h',h_0}, \Gamma^\delta_{h_0} \subset I=[a,b]$, with  $\Gamma^\delta_{h',h_0}\subset \Gamma^\delta_{h_0}$, for every  $h'$ satisfying $0< h'< h_0$, with the following properties
\begin{itemize}
\item[A.]  $\lim_{h' \rightarrow 0} m(\Gamma^\delta_{h',h_0})=  m(\Gamma^\delta_{h_0}) > 1-\gamma$.
\item[B.] If $t \in  \Gamma^\delta_{h',h_0}$ and $|h|\leq h'$ then there exists $N_3(t,h)$ such that 
\begin{equation}\label{est_N3} \lfloor  \frac{\mathcal{K}}{2} \log N(t,h) \rfloor  \leq N(t,h)-N_3(t,h)\leq C_5 \mathcal{K} \log N(t,h)\end{equation} 
and 
\begin{equation}\label{semc1} c \notin I_{i,j}
\end{equation}
for all $0\le j<  N_3(t,h)$ and $0\le i <  N_3(t,h)-j$, where $I_{i,j}$ is the smallest interval that contains the set
$$\{  f_{t+h}^{i+j+1}(c), f_{t}^{i+j+1}(c), f_{t+h}^i\circ f_t^{j+1}(c), f_{t+h}^{i+1}\circ f_{t}^{j}(c)\}$$
\item[C.] For every $t \in \Gamma^\delta_{h',h_0}$ the critical point of $f_t$ is not periodic. 
\item[D.] If $0< \hat{h} < h' \leq h_0$ then $\Gamma^\delta_{h',h_0}\subset \Gamma^\delta_{\hat{h},h_0}$, 
\end{itemize}
where $m$ is the normalized Lebesgue measure on $I=[a,b]$.
\end{proposition}
We will prove Proposition \ref{main_prop} in Section \ref{sec_controle}. The following proposition is one of the most important results in this work.  It relates the Birkhoff sum of  the observable $\phi$ with  the Wild part. This fact will allow us to use  the almost sure invariance principle obtained by  Schnellmann \cite{sch2}.

\begin{proposition}\label{wild_part} 
Let $f_t$ be a good transversal family.  Let  $\phi: [0,1] \to \R$ be a lipschitzian observable. If $t \in \Gamma_{h,h_0}^\delta$, where $\Gamma_{h,h_0}^\delta$ is the set given by Proposition \ref{main_prop}, then
$$
\int \phi \, \mathcal{W}(t,h) dm = s_1(t)J(f_t,v_t)\sum_{j=0}^{N_3(t,h)}\left(\phi(f_t^{j}(c))-\int \phi d\mu_t\right) + O\left(\log \log \frac{1}{|h|}\right).
$$
\end{proposition}
We will prove Proposition \ref{wild_part} in Section \ref{sec_wild}.

\begin{proposition}\label{prop_eqs}
Let $f_t$ be a good transversal family.  Let  $\phi: [0,1] \to \R$ be a lipschitzian observable. If $t \in \Gamma_{h,h_0}^\delta$, where $\Gamma_{h,h_0}^\delta$ is the set given by Proposition \ref{main_prop}, then
 \begin{align*}
\frac{\mathcal{R_{\phi}}(t+h_n)-\mathcal{R_{\phi}}(t)}{s_1(t)J(f_t,v_t) h_n}=&\sum_{j=0}^{N_3(t,h_n)}\left(\phi(f_t^{j}(c))-\int \phi d\mu_t\right) +O\left(\log\log \frac{1}{|h_n|}\right)\\
& + \frac{1}{s_1(t)J(f_t,v_t)}\int \phi (I-\PF_{t+h})^{-1}r_h \ dm.
\end{align*}
\end{proposition}
The proof follows directly from Propositions \ref{parteB} and \ref{wild_part}.


\section{Proof of the Central Limit Theorem for the modulus of continuity of $\mathcal{R_{\phi}}$}

To simplify the notation in this section, given a transversal family $t \mapsto f_t$ we will denote $S_t^f= s_1^f(t)$, $J_t^f=J(f_t,\partial_s f_s|_{s=t})$, $\sigma_t^f=\sigma_t^f(\phi)$. Moreover 
$$L_t^f = \int \log |Df_t| d \mu_t^f,$$
where $\mu_t^f$ is the unique absolutely continuous invariant probability of $f_t$, and 
$$\ell^f_t = \frac{1}{\sqrt{L_t^f}}.$$
When there are not confusion with respect  to which family we are dealing with, we will omit $f$ in the notation. 

\begin{lemma}[Functional Central Limit Theorem]\label{fclt} Let $f_t$ be a good transversal $C^2$ family of $C^4$ unimodal maps and $\sigma_t(\phi) \neq 0$ for every $t$. For each $t\in [a,b]$ let us consider the continuous function $\theta \mapsto X_N(\theta,t)$, where
 $$X_N(\theta,t)$$
$$= \frac{1}{\sigma_t  \sqrt{N} }\sum_{k=0}^{\lfloor N \theta  \rfloor-1} \big(  \phi(f_t^k(c)) -\int \phi \ d\mu_t\big)+ \frac{(N\theta-\lfloor N \theta  \rfloor)}{\sigma_t  \sqrt{N}} \big( \phi(f^{\lfloor N \theta  \rfloor}_t(c)) -\int \phi \ d\mu_t\big).$$
Considering the  normalized Lebesgue measure on $ t \in [a,b]$, for each $N$ the function $t\mapsto X_N(\cdot,t)$  induces a measure on the space of continuous functions   and such measures converges  in distribution to the Wiener measure.  We denote $X_N  \stackrel{D} \longrightarrow_N W.$
\end{lemma}
\begin{proof} 
By Schnellmann \cite{sch2} we know that the sequence of functions 
$$\xi_i(t) = \frac{1}{\sigma_t}\left(\phi(f_t^{i+1}(c))- \int_{0}^{1} \phi d\mu_t\right)$$
satisfies the ASIP for every exponent error larger than $2/5$.  By  \cite[Theorem E]{philipp}, the  ASIP implies the Functional Central Limit Theorem for $X_N(\theta,t)$. 
\end{proof}

We are going to need the following 

\begin{proposition}[\cite{bil}]\label{teo_princ}
If
\begin{equation}
\frac{\nu_n}{a_n} \stackrel{P} \longrightarrow_n L \label{hip},
\end{equation}
where $L$ is a positive constant and $(a_n)_n$ is a sequence such that $a_n\to \infty$ when $n \to \infty$, then
$$X_N \stackrel{D} \longrightarrow_N W$$
implies
$$Y_n \stackrel{D} \longrightarrow_n W,$$
where $Y_n$ is 
$$ \frac{1}{\sigma_t \sqrt{\nu_n(t)} }\sum_{k=0}^{\lfloor \nu_n \theta  \rfloor-1} \big(  \phi(f_t^k(c)) -\int \phi \ d\mu_t\big)+ \frac{(\nu_n\theta-\lfloor \nu_n \theta  \rfloor)}{\sigma_t\sqrt{\nu_n(t)}} \big( \phi(f^{\lfloor \nu_n \theta  \rfloor}_t(c)) -\int \phi \ d\mu_t\big),$$
\end{proposition}
\proof See \cite{bil}, page 152.

From now on we will denote
$$\mathcal{D}_{\mathcal{N}}(y)=  \frac{1}{\sqrt{2\pi}} \int_{-\infty}^{y} e^{-\frac{s^2}{2}} ds.$$
The following lemma will be used many times
\begin{lemma}[A variation of Slutsky's Theorem]\label{slutsky}  Let $A_n\colon [0,1] \rightarrow \mathbb{R}$ be   functions and $\Omega_n \subset [0,1]$ be such that 
$$\liminf_n m(\Omega_n) > 1-\gamma,$$
and for every  $y \in \mathbb{R}$  the sequence  $$ a_n(y)=m( t \in \Omega_n  \colon   A_n(t)\leq y )$$
eventually belongs to 
 $$O(y,\epsilon) =( \mathcal{D}_{\mathcal{N}}(y)- \epsilon, \mathcal{D}_{\mathcal{N}}(y)+ \epsilon),$$
 that  is, there is $n_0=n_0(y)$ such that $a_n(y) \in O(y,\epsilon)$ for every $n\geq n_0$. 
 Then
 \begin{itemize}
\item[A.] There exists $\delta > 0$ such that if  $B_n\colon [0,1] \rightarrow \mathbb{R} $ is a function such that 
$$ \liminf_n m(t\in [0,1]\colon |B_n(t)-1|< \delta) > 1-\gamma,$$
then the sequence
$$ b_n(y)=m( t \in [0,1]  \colon   A_n(t)B_n(t) \leq y )$$
eventually belong to $O(y,\epsilon+3\gamma)$.
\item[B.]   There exists $\delta > 0$ such that if  $B_n\colon [0,1] \rightarrow \mathbb{R} $ is a function such that 
$$ \liminf_n m(t\in [0,1]\colon |B_n(t)|< \delta) > 1-\gamma,$$
then the sequence
$$ b_n(y)=m( t \in [0,1]  \colon   A_n(t)+B_n(t) \leq y )$$
eventually belong to  $O(y,\epsilon+3\gamma)$.
 \end{itemize}
\end{lemma} 
\begin{proof}[Proof of A] Define
$$D_{A}^n(y) =\{ t \in \Omega_n \colon A_n(t) \leq y \}$$
 $$D_{B}^n= \{t\in [0,1]\colon |B_n(t)-1|< \delta\}$$
 $$D_{AB}^n(y)= \{ t \in [0,1] \colon A_n(t) B_n(t) \leq y \}$$
 Choose $\delta > 0$  such that 
 $$\sup_{y \in \mathbb{R}}  \sup_{|\delta'| < \delta } |\mathcal{D}_{\mathcal{N}}(y)- \mathcal{D}_{\mathcal{N}}(y (1-\delta'))| < \gamma,$$
 and 
 $$\sup_{y \in \mathbb{R}}  \sup_{|\delta'| < \delta } |\mathcal{D}_{\mathcal{N}}(y)- \mathcal{D}_{\mathcal{N}}(y (1-\delta')^{-1})| < \gamma.$$ 
If $y \geq  0$  
$$D_A^n((1-\delta) y)\cap D_{B}^n  \subset D^n_{AB}(y)  \text { and } D^n_{AB}(y) \cap D_{B}^n  \cap \Omega_n   \subset D^n_A((1-\delta)^{-1}y),  $$
Thus, if $n$ is large
\begin{eqnarray} 
m(D^n_{AB}(y)) &\geq& m(D_A^n((1-\delta) y)\cap D_{B}^n) \nonumber \\
&\geq&  m(D_A^n((1-\delta) y)) -\gamma   \geq   \mathcal{D}_{\mathcal{N}}((1-\delta) y)- \epsilon -\gamma  \nonumber \\
&\geq& \mathcal{D}_{\mathcal{N}}(y) -\epsilon- 2\gamma,
\end{eqnarray}
and 
\begin{eqnarray} 
m(D^n_{AB}(y))  &\leq& m(D^n_{AB}(y) \cap D_{B}^n) +\gamma  \nonumber \\
&\leq& m(D^n_{AB}(y) \cap D_{B}^n\cap \Omega_n) +2\gamma  \nonumber \\
&\leq&  m(D^n_A((1-\delta)^{-1}y))  +2\gamma \leq   \mathcal{D}_{\mathcal{N}}((1-\delta)^{-1} y) +\epsilon +2\gamma \nonumber \\
&\leq& \mathcal{D}_{\mathcal{N}}(y) +\epsilon+3\gamma,
\end{eqnarray}

and if $y < 0$ we have
$$D_A^n((1-\delta)^{-1} y)\cap D_{B}^n  \subset D^n_{AB}(y)  \text { and } D^n_{AB}(y) \cap D_{B}^n  \cap \Omega_n  \subset D^n_A((1-\delta)y),  $$
and an analogous analysis as above gives
$$m(D^n_{AB}(y)) \in O(y,\epsilon + 3 \gamma).$$
\end{proof}

\begin{proof}[Proof of B] Since the proof is  is quite similar  to the proof of  A, we will skip it. 

\end{proof}

\begin{lemma}\label{red}  Let $t \mapsto f_t$, $t \in [a,b]$ be a good transversal $C^2$ family of $C^4$ unimodal maps. Let $\psi\colon [c,d]\to [a,b]$ be an affine  map, $\psi(c)=a$ and $\psi(d)=b$ and  $g_\theta=f_{\psi(\theta)}$. For every small enough $h\neq 0$ we can define
$$\Omega_g(h,y)= \left\{\theta\in [c,d]\colon \frac{1}{\sigma_\theta^g \ell_\theta^g S_\theta^g  J_\theta^g \sqrt{-\log |h|}}\left(\frac{\mathcal{R_{\phi}}_g(\theta+h)-\mathcal{R_{\phi}}_g(\theta)}{h}\right) \le y\right\}$$
and
$$\Omega_f(w,y)= \left\{t \in [a,b]\colon \frac{1}{\sigma_t^f \ell_t^f S_t^f J_t^f\sqrt{-\log |w|}}\left(\frac{\mathcal{R_{\phi}}_f(t+w)-\mathcal{R_{\phi}}_f(t)}{w}\right) \le y\right\}.$$
 If 
  $$\frac{m(\Omega_g(h, y))}{m([c,d])}$$
 eventually belong to $O(y,\gamma)$ when $h$ converges to $0$ 
then  
$$\frac{m(\Omega_f(rh,y))}{m([a,b])}$$
 eventually belong to $O(y,\gamma')$ when $h$ converges to $0$,  for every $\gamma' > \gamma$.  Here  $r=\psi'$.
\end{lemma}
\begin{proof} It follows easily from  Lemma \ref{slutsky}.A. \end{proof}
\begin{remark} Lemma \ref{red} implies that it is enough to show our main theorem for families parametrized by $[0,1]$.
\end{remark}

\begin{proposition}\label{main_step} For every $\gamma > 0$ there exists $Q_1$ with the following property. Let $f_t$ be a good transversal $C^2$ family of $C^4$ piecewise expanding  unimodal maps with $\sigma_t(\phi) \neq 0$ for every $t$ and 
$$Q=\sup_{t,t'\in [c,d]} \left|1-\frac{L_{t'}}{L_t}\right| < Q_1.$$
Then for every $h$ small enough we have
$$\frac{1}{m([c,d])}m \left\{t \in [c,d]\colon \frac{1}{\sigma_t \ell_t S_t J_t\sqrt{-\log |h|}}\left(\frac{\mathcal{R_{\phi}}(t+h)-\mathcal{R_{\phi}}(t)}{h}\right) \le y\right\}$$
belongs to $O(y,13\gamma)$. 

\end{proposition} 
\begin{proof} Without loss of generality we assume that $[c,d]=[0,1]$. It is enough to prove the following claim: For every sequence 
$$h_n\rightarrow_n 0$$
and every $\gamma > 0$,   the sequence
$$s_n= m \left\{t \in [0,1]:\frac{1}{\sigma_t \ell_t S_t J_t\sqrt{-\log |h_n|}}\left(\frac{\mathcal{R_{\phi}}(t+h_n)-\mathcal{R_{\phi}}(t)}{h_n}\right) \le y\right\}   $$
eventually belong to the interval $O(y,12\gamma)$.

\noindent Fix a large  $\mathcal{K} > 0$.  By Proposition \ref{main_prop}, for every $\gamma > 0$ there exist $\delta > 0$, $h_0 > 0$ and  sets $\Gamma^\delta_{h,h_0}, \Gamma^\delta_{h_0} \subset I$, with  $\Gamma^\delta_{h,h_0}\subset \Gamma^\delta_{h_0}$,  for every $h\neq 0$  satisfying $|h|< h_0$,  such that 
 \begin{itemize}
\item[A.]  $\lim_{h \rightarrow 0} m(\Gamma^\delta_{h,h_0})=  m(\Gamma^\delta_{h_0}) > 1-\gamma$.
\item[B.] If $t \in  \Gamma^\delta_{h,h_0}$ then there exists $N_3(t,h)$ such that 
$$
\lfloor  \frac{\epsilon}{2} \log N(t,h) \rfloor  \leq N(t,h)-N_3(t,h)\leq C_5 \mathcal{K} \log N(t,h)
$$
and 
$$
c \notin [f_{t+h}^i\circ f_t^{j+1}(c), f_{t+h}^{i+1}\circ f_{t}^{j}(c)]
$$
for all $1\le j<  N_3(t,h)$ and $0\le i <  N_3(t,h)-j$.
\end{itemize} 
For all $h \neq 0$ and $t \in [0,1]$, define $N_4(t,h)=N_3(t,h)$ if $t \in \Gamma^\delta_{h,h_0}$ and $|h| < h_0$, and $N_4(t,h)=N(t,h)$, otherwise. Therefore, for $B$ we have
\begin{equation}\label{refer_B} 
N(t,h)-N_4(t,h)\leq C_5 \mathcal{K} \log N(t,h)
\end{equation}
for every $(t,h)$. Since
$$\frac{1}{|Df_t^{N(t,h)+1}(f_t(c))|}\le |h| <  \frac{1}{|Df_t^{N(t,h)}(f_t(c))|},$$
we have 
$$ \frac{1}{N(t,h)} \sum_{k=1}^{N(t,h)}\log |Df_t(f_t^k(c))| <  \frac{-\log |h| }{N(t,h)}\leq   \frac{1}{N(t,h)}   \sum_{k=1}^{N(t,h)+1} \log|Df_t(f_t^k(c))|.$$
By Schnellmann\cite{sch1}, we have for almost every $t$
$$\lim_{N\to +\infty} \frac{1}{N} \sum_{k=1}^{N} \log|Df_t(f_t^k(c))|  = L_t= \int \log |Df_t| \ d\mu_t,$$
which implies that for almost every $t$
$$\lim_{h\rightarrow 0} \frac{-\log |h| }{N(t,h)} =  \int \log|Df_t| \ d\mu_t.$$
And by Eq. (\ref{refer_B})
$$\frac{-\log |h| }{N(t,h)}  \leq \frac{-\log |h| }{N_4(t,h)} \leq  \frac{-\log |h| }{N(t,h)- C_5 \mathcal{K} \log N(t,h) },$$ 
we also have  
\begin{equation}\label{conv_P2} \lim_{h\rightarrow 0} \frac{ L_t N_4(t,h)} {-\log |h| }= 1.\end{equation}
for almost every $t \in [0,1]$.  Fix $t_0 \in [0,1]$ such that $L_{t_0}= min_{t \in [0,1]} L_t$. Then 
\begin{equation}\label{conv_P}
\frac{ L_t /L_{t_0}N_4(t,h_n)} {-\log |h_n| } \stackrel{P} \longrightarrow \frac{1}{L_{t_0}}.
\end{equation} 
By Lemma \ref{fclt} and Propostion \ref{teo_princ},
\begin{equation}\label{yn} Y_n(\theta,t) \stackrel{D} \longrightarrow_n W,\end{equation} 
where $Y_n$ is given in Propostion \ref{teo_princ} and $W$ is the Wiener measure, with 
$$\nu_n(t) =  N_4(t,h_n) \frac{L_t}{L_{t_0}}.$$
Hence, taking $\theta=1$ we conclude that 
\begin{equation}\label{Yn_converge_normal}
Y_n(1,t) \stackrel{D} \longrightarrow_n \mathcal{N}(0,1),
\end{equation} 
where $\mathcal{N}(0,1)$ denotes the Normal distribution with average zero and  variance one.  
Let $$Q=\sup_{t\in [0,1]} \left|1-\frac{L_{t_0}}{L_t}\right|.$$
 Fix $\alpha \in (0,1/2)$.  The L\'evy's modulus of continuity  theorem (see for instance Karatzas and Shreve \cite{w}) implies that for almost every function $f$ with respect to  the Wiener measure there exists $C_f$ such that 
 $$|f(\theta')-f(\theta)|\leq C_f |\theta'-\theta|^\alpha$$
 for all $\theta', \theta \in [0,1]$. In particular there exist $H=H(\gamma)$ and a set $\Omega_\gamma$ of $\alpha$-H\"older continuous  functions in $C([0,1],\mathbb{R})$, whose measure with respect to the Wiener measure is larger than $1-\gamma$, such that 
$$|f(\theta')-f(\theta)|\leq H  |\theta'-\theta|^\alpha.$$
In particular for $f \in \Omega_\gamma$ we have
\begin{equation}\label{wiener}  \max_{\theta \in [1-Q,1]} |f(1)-f(\theta)|\leq HQ^\alpha,\end{equation} 
Due to Eq. (\ref{yn}), 
$$\liminf_{n} m\{ t \in [0,1]\colon \  \max_{\theta \in [1-Q,1]} |Y_n(1,t)-Y_n(\theta,t)|\leq HQ^\alpha  \} > 1-\gamma.$$
In particular if 
$$D_n = \{ t \in [0,1]\colon \   |Y_n(1,t)-Y_n(\frac{L_{t_0}}{L_t},t)|\leq 2HQ^\alpha  \}$$
then $\liminf_n m(D_n) > 1-\gamma.$ Let us apply Lemma \ref{slutsky}.B with $\Omega_n=D_n$, $A_n(t)= Y_n(1,t)$ and $B_n(t)= Y_n(\frac{L_{t_0}}{L_t},t)-Y_n(1,t)$. Observe that by Eq.(\ref{Yn_converge_normal}) the sequence $a_n(y)$ defined in Lemma \ref{slutsky} eventually belongs to $O(y,\epsilon)$ for all $\epsilon>0$. Hence, taking $\epsilon=\gamma$, there exists $\delta_1=\delta_1(\gamma) > 0$ such that if $2HQ^\alpha < \delta$  we have
\begin{equation}
m(t\in [0,1]\colon  Y_n(\frac{L_{t_0}}{L_t},t) \leq y )  
\end{equation} 
eventually belongs to $O(y,4\gamma)$. Choose $Q_0=Q_0{\gamma} > 0$ such that if $Q  < Q_0$ then $2HQ^\alpha < \delta_1$. Note that 
\begin{equation} 
Y_n(\frac{L_{t_0}}{L_t},t)=\sqrt{ \frac{L_{t_0}}{L_t} } \frac{1}{\sigma_t\sqrt{N_4(t,h_n)} }\sum_{k=0}^{\lfloor N_4(t,h_n)  \rfloor-1} \big(\phi(f_t^k(c)) -\int \phi \ d\mu_t\big).
\end{equation} 
By Eq. (\ref{conv_P2}) and Lemma \ref{slutsky}.A, the sequence
$$ m(t \in [0,1]\colon  \frac{\sqrt{L_{t_0}}}{\sigma_t \sqrt{-\log |h_n|} }\sum_{k=0}^{\lfloor N_4(t,h_n)  \rfloor-1} \big(  \phi(f_t^k(c)) -\int \phi \ d\mu_t \big)\leq  y )$$
eventually belongs to  $O(y,7\gamma)$. Applying again Lemma \ref{slutsky}.A, with 
$$A_n(t)=  \frac{\sqrt{L_{t_0}}}{\sigma_t \sqrt{-\log |h_n|} }\sum_{k=0}^{\lfloor N_4(t,h_n)  \rfloor-1} \big(  \phi(f_t^k(c)) -\int \phi \ d\mu_t \big),$$
$\Omega_n= [0,1]$ and 
$$B_n(t)= \sqrt{\frac{L_t}{L_{t_0}}},$$
there exists $\delta_2=\delta_2(\gamma)>0$ such that if 
\begin{equation}
\label{sq}  \left|\sqrt{\frac{L_t}{L_{t_0}}}-1\right| < \delta_2
\end{equation} 
for every $t$ then 
$$ m(t \in [0,1]\colon  \frac{\sqrt{L_{t}}}{\sigma_t \sqrt{-\log |h_n|} }\sum_{k=0}^{\lfloor N_4(t,h_n)  \rfloor-1} \big(  \phi(f_t^k(c)) -\int \phi \ d\mu_t \big)\leq  y )$$
eventually belong to $O(y,10\gamma)$. Choose $Q_1 = \min\{Q_0,\delta_2\}$ such that $Q< Q_1$ implies Eq. (\ref{sq}). Finally by Propositions \ref{wild_part} and \ref{prop_eqs}, if $0< |h_n|\leq h_0$ and $t\in \Gamma^\delta_{h_n,h_0}$  we have 
 \begin{align*}
\frac{\mathcal{R_{\phi}}(t+h_n)-\mathcal{R_{\phi}}(t)}{ S_t J_t  h_n}=&\sum_{j=0}^{N_3(t,h_n)}\left(\phi(f_t^{j}(c))-\int \phi d\mu_t\right) +O\left(\log\log \frac{1}{|h_n|}\right)\\
& + \frac{1}{S_t J_t}\int \phi (I-\PF_{t+h})^{-1}r_h \ dm.
\end{align*}
Since 
$$
\frac{\log\log \frac{1}{|h_n|}}{\sqrt{\log\frac{1}{|h_n|}}}\rightarrow_n 0
$$
and
$$
\sup |(I-\PF_{t+h})^{-1}r_h|_{L^1} < \infty,
$$
we have 
$$\frac{\mathcal{R_{\phi}}(t+h_n)-\mathcal{R_{\phi}}(t)}{ S_t  \sigma_t J_t  h_n \sqrt{-\log |h_n| }}= \frac{1}{\sigma_t  \sqrt{-\log |h_n| }} \sum_{j=1}^{N_3(t,h_n)}\left(\phi(f_t^{j}(c))-\int \phi d\mu_t\right) +r(t,h_n),$$
where 
$$\lim_n \sup_{t\in \Gamma^\delta_{h_n,h_0}} |r(t,h_n)|=0.$$
Hence, it is easy to conclude that 
\begin{align} 
& \label{pp} \frac{\mathcal{R_{\phi}}(t+h_n)-\mathcal{R_{\phi}}(t)}{ S_t \sigma_t  \ell_tJ_t  h_n \sqrt{-\log |h_n| }}\nonumber \\ 
&= \frac{1}{\ell_t  \sigma_t  \sqrt{-\log |h_n| }} \sum_{j=1}^{N_3(t,h_n)}\left(\phi(f_t^{j}(c))-\int \phi d\mu_t\right) +r'(t,h_n),
\end{align} 
for every $t\in \Gamma^\delta_{h_n,h_0}$, where $$\ell_t= \frac{1}{\sqrt{L_t}}$$
and
$$\lim_n \sup_{t\in \Gamma^\delta_{h_n,h_0}} |r'(t,h_n)|=0.$$
Since $m(\Gamma^\delta_{h,h_0})> 1-\gamma$, we can apply Lemma \ref{slutsky} 
  (remember that $N_4(t,h)=N_3(t,h)$ for $t \in \Gamma^\delta_{h,h_0}$)  to conclude that the sequence 
$$m(t \in [0,1] \colon  \frac{\mathcal{R_{\phi}}(t+h_n)-\mathcal{R_{\phi}}(t)}{ S_t \sigma_t  \ell_tJ_t h_n \sqrt{-\log |h_n| }} \leq y)$$
eventually belong to the interval  $O(y,13\gamma)$. 
\end{proof}

\begin{lemma}\label{dec}  
Let $[c_i,d_i] \subset [a,b]$, $i \in \Delta \subset \mathbb{N}$, be intervals with pairwise disjoint interior and such that 
$$m([a,b]\setminus \bigcup_{i \in \Delta} [c_i,d_i])=0.$$
If $t\mapsto f_t$, with $t \in [c_i,d_i]$, are  good  transversal families such that for all $i\in \Delta$ and $y \in \mathbb{R}$ we have
$$\frac{1}{m([c_i,d_i])}m \left\{t \in [c_i,d_i]\colon \frac{1}{\sigma_t \ell_t S_t J_t\sqrt{-\log |h|}}\left(\frac{\mathcal{R_{\phi}}(t+h)-\mathcal{R_{\phi}}(t)}{h}\right) \le y\right\}$$
eventually belongs to $O(y,\gamma),$ then 
$$\frac{1}{m([a,b])}m \left\{t \in [a,b]\colon t + h \in [a,b] \text { and } \frac{1}{\sigma_t \ell_t S_t J_t\sqrt{-\log |h|}}\left(\frac{\mathcal{R_{\phi}}(t+h)-\mathcal{R_{\phi}}(t)}{h}\right) \le y\right\}$$
eventually belongs to $O(y,\gamma +\epsilon),$ for every $\epsilon > 0$.
\end{lemma}
\begin{proof} Define
$$\Omega(h,y)= \left\{t \in [a,b]\colon t+h \in [a,b] \text{ and } \frac{1}{\sigma_t \ell_t S_t J_t\sqrt{-\log |h|}}\left(\frac{\mathcal{R_{\phi}}(t+h)-\mathcal{R_{\phi}}(t)}{h}\right) \le y\right\}$$
and
$$\Omega_i(h,y)= \left\{t \in [c_i,d_i]\colon t+h \in [a,b] \text{ and } \frac{1}{\sigma_t \ell_t S_t J_t\sqrt{-\log |h|}}\left(\frac{\mathcal{R_{\phi}}(t+h)-\mathcal{R_{\phi}}(t)}{h}\right) \le y\right\}.$$
Of course $\Omega_i(h,y)$ are pairwise disjoint up to a countable set, $\Omega_i(h,y)\subset \Omega(h,y)$ and 
$$m(\Omega(h,y)\setminus \cup_i \Omega_i(h,y))=0.$$
Then $$ m(\Omega(h,y))= \sum_{i \in \Delta} m(\Omega_i(h,y)).$$
Given $\epsilon \in (0,1)$, choose $i_0$ such that 
$$m(\cup_{i >  i_0} [c_i,d_i]) < \epsilon \, m([a,b]) .$$
For every $i \leq  i_0$ there exists $h_i > 0$ such that for every $|h| < h_i$ we have
$$\frac{m(\Omega_i(h,y))}{m([c_i,d_i])}$$
belongs to $O(y, \gamma+\epsilon)$. Let $\hat{h} =\min_{i \leq i_0} h_i$. Let
$$U_{i_0}(h,y)= \cup_{i\leq i_0} \Omega_i(h,y),$$
and
$$W_{i_0}(h,y)=\cup_{i\leq i_0} [c_i,d_i].$$
Then for $|h| < \hat{h}$ we have
$$\frac{m(U_{i_0}(h,y))}{m(W_{i_0}(h,y))} = \sum_{i \leq i_0} \frac{m([c_i,d_i])}{m(W_{i_0}(h,y))}  \frac{m(\Omega_i(h,y))}{m([c_i,d_i])} $$
is a convex combination of elements of $O(y, \gamma+\epsilon)$, then it belongs to $O(y, \gamma+\epsilon)$. We conclude that
\begin{align} 
& (\mathcal{D}_{\mathcal{N}}(y)- \gamma -2\epsilon) m([a,b]) \nonumber \\  
&\le (\mathcal{D}_{\mathcal{N}}(y)- \gamma -\epsilon) ( m([a,b]) -  \epsilon\, m([a,b]))  \nonumber \\ 
&\le (\mathcal{D}_{\mathcal{N}}(y)- \gamma -\epsilon) m(W_{i_0}(h,y))   \nonumber \\ 
&\le  m(U_{i_0}(h,y)) \nonumber \\ 
&\leq     m(\Omega(h,y)) \nonumber \\
&\le m(U_{i_0}(h,y))+ \epsilon \,m([a,b]) \nonumber \\ 
&\le (\mathcal{D}_{\mathcal{N}}(y)+ \gamma+\epsilon) m(W_{i_0}(h,y)) +  \epsilon \,m([a,b]) \nonumber \\
&\le (\mathcal{D}_{\mathcal{N}}(y)+ \gamma+2\epsilon ) m([a,b]).
\end{align} 
\end{proof}

\begin{proof}[Proof of Theorem \ref{main}] Remember that 
$$t\mapsto L_t$$
is a continuous and positive function on $[a,b]$. Given $\gamma > 0$, let $Q_1 > 0$ be as in Proposition \ref{main_step}.  Then there are  $k > 0$ and  intervals $[c_i,d_i]$, $i \leq  k=k(\gamma)$,  which forms a partition $\mathcal{F}$ of $[a,b]$ and 
$$\sup_{t, t' \in [c_i,d_i]} \big| 1- \frac{L_{t'} }{L_t}    \big| < Q_1$$
for every $i \leq k$. Then the restrictions of the family $f_t$ to each one of the intervals $[c_i,d_i]$ satisfy the assumptions of Proposition \ref{main_step}.  Now it remains to apply  Lemma \ref{dec} to the full family and the partition  $\mathcal{F}$. Since $\gamma > 0$ is arbitrary we completed the proof of  Theorem \ref{main}. 
\end{proof}


\section{Controlling how the orbit of the critical point moves}\label{sec_controle}

The aim of this section is to prove Proposition \ref{main_prop}. Let us denote by $I=[0,1]$ the interval of parameters.

\begin{remark}\label{const2}
In Schnellmann \cite[Lemma 4.4]{sch2}  it is proven  that there is $C_1>0$ such that if $N\ge1$, $|t_1-t_2|<1/N$ and if $\omega_1 \in \p_N(t_1)$ and $\omega_2 \in \p_N(t_2)$ have the same combinatorics up to the $(N-1)$-th iteration then
$$
\left|\frac{Df_{t_1}^N(x_1)}{Df_{t_2}^N(x_2)}\right|\le C_1,
$$
for all $x_1 \in \omega_1$ and $x_2 \in \omega_2$.

We also observe that if $x,y \in \omega \in  \p_{N}(t)$, then by the bounded distortion lemma, there is $C_2 > 0$ such that
$$
\left|\frac{Df_{t}^j(x)}{Df_{t}^j(y)}\right|\le C_2,
$$
for every $j\leq N$. Let 
$$
\tilde{M} = \sup_{0\le j \le j_0}\sup_{t\in [a,b]} |\partial_tf_t^{j}(c)|,
$$ 
and let us define
\begin{equation}\label{def_c3}
C_3 = \max\{C, \tilde{M}\},
\end{equation} 
where $C$ is the constant given by the transversality condition (see Eqs  (\ref{const_t}) and (\ref{const_t2})) and 
$$C_4=\sup_{t\in [0,1]} \sup_{x \in [0,1]} |\partial_t f_t(x)|.$$

\end{remark}

To prove Proposition \ref{main_prop}  we will need

\begin{lemma}\label{lemmaB}  Let $N_3 \in \mathbb{N}$ and  $\omega  \in \p_{N_3}$ be  such that 
$$|\omega|\leq \frac{1}{{N_3}}.$$
If $t \in \omega$ and  
 \begin{equation} \label{deep} dist(t, \partial \omega)> (M+1)|h|,\end{equation} 
 where 
\begin{equation}\label{def_M}
M>  \max\{ C_1C_3C_4, C_1^2C_2 C_3^2 \}
\end{equation}
 Then  
 \begin{equation}\label{semc12} c \notin I_{i,j}(t,h) \end{equation}
for all $0\le j < {N_3}$ and $0\le i < {N_3}-j$, where $I_{i,j}(t,h)$ is the smallest interval that contains the set
$$
\{ f_{t+h}^{i+j+1}(c), f_{t}^{i+j+1}(c),  f_{t+h}^i\circ f_t^{j+1}(c), f_{t+h}^{i+1}\circ f_{t}^{j}(c)\}.
$$
\end{lemma}
\begin{proof} Let $j_0$ be as defined in condition $(I)$ (see Definition \ref{goodf}). If $j>j_0$ define $i_1=0$ and if $0\le j < j_0$ define $i_1 = j_0$.
First of all, we observe that if $0\le j \le j_0$ and $0\leq i \le j_0$ then Eq. (\ref{semc12})  follows from condition $(VI)$. In particular  
\begin{equation}\label{ooo}  c \notin I_{i,j}(t,h) \text{  for every } i < i_1. \end{equation}  
Hence, it is left to consider the cases when 
$$
j_0 < j < N_3 \mbox{\; and \;}  0 <i \le N_3-j
$$
and
$$
0 < j < j_0 \mbox{\; and \;}  j_0 <i \le N_3-j.
$$
 We claim that 
\begin{equation}\label{contido} 
c \notin  I_{i_1,j}.
\end{equation}
Indeed, if $0\le j\le j_0$, it follows from condition $(VI)$. Now, if $j_0 < j < N_3$, due to condition ($I$), Eqs. (\ref{const_t}) and (\ref{const_t2}) the maps
$$
\theta \in \omega \to f_\theta^k(c) \in [0,1]
$$
are diffeomorphisms on their images for every $j_0< k\leq {N_3}$ and they  do not contain the critical point in its image, for all $j_0<k<{N_3}$, $\theta \in \omega$.  In particular if $\omega=(s_1,s_2)$ then 
\begin{equation}\label{semc} 
c\notin \{f_\theta^k(c)\colon \theta \in \omega\} = (f_{s_1}^k(c),f_{s_2}^k(c))
\end{equation}
for every $j_0< k < {N_3}$. Therefore,
$$
c\notin [f_{t}^k(c),f_{t+h}^k(c)].
$$
By the Mean Value Theorem and Remark \ref{const2}, for every $j < {N_3}$ 
$$
|f_{t}^{j+1}(c) - f_{t+h}^{j+1}(c)| = |\partial_\theta f_{\theta}^{j+1}(c)|_{\theta=\theta_1}||h|\le C_3 |Df_{\theta_1}^j(f_{\theta_1}(c))||h| \leq  C_3C_1 |Df_t^j(f_t(c))||h|.
$$
Moreover,
\begin{equation}\label{length1}
|f_{t+h}(f_{t}^{j}(c)) - f_{t}(f_{t}^{j}(c))|\le  |\partial_\theta f_{\theta}(f_t^{j}(c))|_{\theta=\theta_2}||h|\leq  C_4|h|.
\end{equation}
By assumption, $d([t,t+h], \partial \omega)>M|h|$. Thus,
\begin{equation}\label{length2}
|\omega| \ge (2M +1)|h|.
\end{equation}
If $\partial \omega = \{s_1,s_2\}$ and $s\in [t,t+h]$ then 
\begin{eqnarray}\label{length3}
|f_{s_i}^{k+1}(c) -f_s^{k+1}(c)| &=& |\partial_\theta  f_{\theta}^{k+1}(c)|_{\theta=\theta_3}||s_i-s|\nonumber \\
&\ge& \frac{1}{C_3}|Df_{\theta_3}^k(f_{\theta_3}(c))|M|h|  \nonumber \\
&\geq& \frac{1}{C_1 C_3}|Df_{t}^k(f_{t}(c))|M|h|
\end{eqnarray}
for every $k < {N_3}$. Taking $k=j$ we obtain 
\begin{eqnarray} 
|f_{t+h}(f_{t}^{j}(c)) - f_t(f_{t}^{j}(c))|&\leq& C_4 |h|\leq \frac{M}{C_1C_3} |h| \nonumber \\
&\leq& \frac{M}{C_1C_3}  |Df_{t}^j(f_{t}(c))||h|\leq |f_{s_i}^{j+1}(c) -f_t^{j+1}(c)|.  
\end{eqnarray}
Hence,
\begin{equation}
[f_{t+h}(f_{t}^{j}(c)),f_t(f_{t}^{j}(c))] \subset  (f_{s_1}^{j+1}(c) ,f_{s_2}^{j+1}(c)).
\end{equation}
In particular
$$c \notin I_{0,j}(t,h)=I_{i_1,j}(t,h).$$
 We concluded the proof of our claim. Now fix $0\le j < {N_3}$. We are going to prove by induction on $i$ that, for every $i_1\leq  i <  {N_3}-j$,
\begin{equation}\label{inducao} 
c \notin I_{i,j}(t,h),
\end{equation} 
The case $i=i_1$ follows from Eq. (\ref{contido}). Now suppose that Eq. (\ref{inducao}) holds up to $i$.  Provided that $i\ge i_1$, we have $i+j+1> j_0$. Therefore, by Eq. (\ref{semc}), with $k=i+j+2$, we obtain
$$f_{t+h}^{(i+1)+j}(f_{t+h}(c)) \in (f_{s_1}^{(i+1)+j+1}(c) ,f_{s_2}^{(i+1)+j+1}(c)).$$
And as in Eq. (\ref{length3}) 
\begin{equation}\label{e1}
| f_{s_i}^{(i+1)+j}(f_{s_i}(c))-f_{t+h}^{(i+1)+j}(f_{t+h}(c))| \geq  \frac{1}{C_3C_1}|Df_{t+h}^{(i+1)+j}(f_{t+h}(c))| M|h|
\end{equation}
Moreover by induction assumption and Eq. (\ref{ooo}) , we have for every $0 \leq  k\leq i$ 
$$c \not\in I_{k,j}(t,h)$$
Thus  the points 
$$ f_{t+h}^{j+1}(c) \text { and }  f_{t}^{j+1}(c) $$
have the same combinatorics up to $i$ iterations of the map $f_{t+h}$. Then by Remark \ref{const2}
\begin{align}\label{e2}
|f_{t+h}^{i+1}(f_{t}^{j+1}(c))-f_{t+h}^{i+1}(f_{t+h}^{j+1}(c))| &\leq C_2   |Df_{t+h}^{i+1}(f_{t+h}^{j+1}(c))||f_{t}^{j+1}(c)-f_{t+h}^{j+1}(c)| \nonumber \\
&\le C_2   |Df_{t+h}^{i+1}(f_{t+h}^{j+1}(c))| |\partial_\theta f_{\theta}^{j+1}(c)|_{\theta=\theta_4}||h| \nonumber \\
&\le C_3  C_2   |Df_{t+h}^{i+1}(f_{t+h}^{j+1}(c))| |Df_{\theta_4}^{j}(f_{\theta_4}(c))||h| \nonumber \\
&\le C_1 C_2 C_3 | Df_{t+h}^{i+1}(f_{t+h}^{j+1}(c))| | Df_{t+h}^{j}(f_{t+h}(c))||h| \nonumber\\
&\le C_1 C_2 C_3 |Df_{t+h}^{(i+1)+j}(f_{t+h}(c))||h|
\end{align}
and
$$f_{t}^{j}(c) \text{ and }  f_{t+h}^{j}(c)$$
have the same combinatorics up to $i+1$ iterations of the map $f_{t+h}$. Then by Remark \ref{const2}
\begin{align}\label{e3}
|f_{t+h}^{(i+1)+1}(f_{t}^{j}(c))-f_{t+h}^{(i+1)+1}(f_{t+h}^j(c))|& \leq C_2  |Df_{t+h}^{(i+1)+1}(f_{t+h}^{j}(c))||f_{t}^{j}(c)-f_{t+h}^{j}(c)| \nonumber \\
&\le C_2  |Df_{t+h}^{(i+1)+1}(f_{t+h}^{j}(c))| |\partial_\theta f_{\theta}^{j}(c)|_{\theta=\theta_5}|  |h| \nonumber \\
&\le  C_2 C_3   |Df_{t+h}^{(i+1)+1}(f_{t+h}^{j}(c))| |Df_{\theta_5}^{j-1}(f_{\theta_5}(c))|||h| \nonumber \\
&\le  C_1 C_2 C_3    |Df_{t+h}^{(i+1)+1}(f_{t+h}^{j}(c))|  |Df_{t+h}^{j-1}(f_{t+h}(c))||h| \nonumber \\
&\le  C_1 C_2 C_3    |Df_{t+h}^{(i+1)+j}(f_{t+h}(c))||h|.
\end{align}
Since 
$$C_1 C_2 C_3 < \frac{M}{C_1 C_3},$$
Eqs. (\ref{e1}), (\ref{e2}) and (\ref{e3}) imply that 
$$\{ f_{t+h}^{(i+1)+1}(f_{t}^{j}(c)), f_{t+h}^{i+1}(f_{t}^{j+1}(c)) \} \subset (f_{s_1}^{(i+1)+j+1}(c),f_{s_2}^{(i+1)+j+1}(c)).$$
In particular, $c\notin I_{i,j}(t,h)$ for all $0\le j< N_3$ and $i_1< i <N_3-j$.
\end{proof}

To prove Proposition \ref{main_prop} we need to show that,  for each given $h\neq 0$,  for most of the parameters $ t \in [0,1]$  we can find a cylinder $\omega \in \p_{N_3(t,h)}$ where $[t,t+h]$ is deep inside $\omega$ (see Eq. (\ref{deep}) ) and moreover $N_3(t,h)$ satisfies  Eq. (\ref{est_N3}). To this end, for most $t$  we will find $\omega$, with $t \in \omega$,  in such way that $|\omega|$ is quite large with respect to $|h|$ and $N_3(t,h)$ satisfies  Eq. (\ref{est_N3}), but not necessarily the whole interval $[t,t+h]$ is deep inside  $\omega$.  Then we will use a simple argument to conclude that for most of the parameters $t$ this indeed occurs.  

Let $\p_j$ be the partition of level $j>j_0$. Observe that for each cylinder $\omega \in \p_j$
$$
|\omega| \le C_3 \left(\frac{1}{\lambda}\right)^j,
$$
where $C_3$ is the constant given by Eq. (\ref{def_c3}).

Let $N>1$ and  define $j=j(N)$ as
$$
j = \left\lfloor \frac{\log(C_3N)}{\log \lambda}\right\rfloor +1.
$$
Note that the cylinders of $\p_{j}$ divide the interval of parameters $I$ in subintervals of length shorter than $1/N$. Let $J$ be one of these intervals in $\p_j$. And we will denote by $t_R$ the right boundary point of $J$.

Observe that, by defintion, there is an integer $i$, $0\le i< j$ such that $$x_i(t_R)=f_{t_R}^{i+1}(c)=c.$$

Fix an integer $\tau$ such that $2^{1/\tau}\le \sqrt{\lambda}$.

\begin{definition}[The sets $E_{N,J}$]  
Let $J \in \p_j$, $j=j(N)$. Let  $E_{N,J}$ be the family of all intervals $\omega \in \p_N$ such that  for every $k$ satisfying 
$$
0 \le k \le \left\lfloor \frac{\mathcal{K} \log{N}}{\tau} \right\rfloor  
$$
and for $\tilde{\omega}$ satisfying 
$$
\tilde{\omega} = (a,b) \in \p_{\Nivel + q},  \text { with } \omega \subset \tilde{\omega}\subset J,
$$
where
$$
q=\min \{ (k+1)\tau, \lfloor \mathcal{K} \log N \rfloor \},
$$
one of the following statements  holds: 
\begin{itemize} 
\item[A.]  For every $i$ satisfying 
$$
\Nivel +k\tau\le i < \Nivel +q
$$
we have
$$
 x_{i}(a) \neq c.
$$
\item[B.]
For every $i$ satisfying 
$$
\Nivel +k\tau\le i < \Nivel +q
$$
we have
$$
 x_i(b) \neq c.
$$
\end{itemize} 
\end{definition}
\noindent Define  
$$
E_{N}=\bigcup_{J \in \p_j}  E_{N,J}.
$$

Let us denote by $|E_N|$ the sum of the lengths of the intervals in this family.

Given $n\in \N$ and $\tilde{\omega} \in \p_n$  define
$$\delta_{t} :=   \min \{  |f_t^i(c) - f_t^j(c)|\colon \   f_t^i(c) \neq  f_t^j(c) \ i,j \leq \tau. \}$$
$$\delta_{\omega} :=  \frac{ \min_{t\in \overline{\omega}} \delta_t}{2}.$$
Notice that if  $\tilde{\omega} \supset \omega$ then $\delta_{\tilde{\omega}}\leq \delta_{\omega} $.

Let $C_L$ be such that 
$$|f_t^i(c) - f_s^i(c)|\leq C_L |t-s|$$
for all $i\leq \tau$, $s,t \in [0,1]$.

\begin{lemma}\label{compare} 
There is $C > 0$ such that the following holds.  If  $\tilde{\omega} \in \p_i$, $i>j_0$, with $|\tilde{\omega}| < 1/i$ and $t \in \tilde{\omega}$ then
 \begin{equation}\label{const33}  \frac{1}{C}\  \frac{|x_i(\tilde{\omega})|}{|Df_t^{i}(f_t(c))|}\leq  |\tilde{\omega}| \leq C \frac{|x_i(\tilde{\omega})|}{|Df_t^{i}(f_t(c))|}.\end{equation}
Moreover, if  $\omega \in \p_N\setminus E_{N}$  then there exists $i$ satisfying    
$$ \Nivel\leq i \leq N$$
such that $\omega \subset \tilde{\omega} \in \p_i$ and if
$$C_L |\tilde{\omega}| <      \delta_{\tilde{\omega}}$$ then 
$$|x_i(\tilde{\omega})|\geq \delta_{\tilde{\omega}}$$
 and
 \begin{equation}\label{const}  \frac{1}{C}\  \frac{\delta_{\tilde{\omega}}}{|Df_t^{i}(f_t(c))|}\leq  |\tilde{\omega}| \leq C \frac{1}{|Df_t^{i}(f_t(c))|} \end{equation}
for every $t \in \tilde{\omega}$.
\end{lemma} 
\begin{proof}  If $t\in \tilde{\omega} \in \p_k$ then by the Mean Value Theorem for some $\theta_1 \in \tilde{\omega}$ 
$$|x_k(\tilde{\omega})|= |\partial_\theta f_\theta^{k+1}(c)|_{\theta=\theta_1}| |\tilde{\omega}|,$$
then
$$\frac{ |Df_{t}^{k}(f_{t}(c))| |\tilde{\omega}|}{C_1C_3}  \leq \frac{ |Df_{\theta_1}^{k}(f_{\theta_1}(c))| |\tilde{\omega}|}{C_3}  \leq  |x_k(\tilde{\omega})|$$
and
$$ |x_k(\tilde{\omega})| \leq  C_3 |Df_{\theta_1}^{k}(f_{\theta_1}(c))| |\tilde{\omega}| \leq C_1 C_3  |Df_{t}^{k}(f_{t}(c))| |\tilde{\omega}|,$$
therefore, Eq. (\ref{const33}) holds. 
Now assume $\omega \in \p_N\setminus E_{N}$. Then  there exists  
 $k$ satisfying 
$$0 \le k \le \left\lfloor\frac{\mathcal{K} \log{N}}{\tau}\right\rfloor $$
and  $$\tilde{\omega} = (a,b) \in \p_{\Nivel + q},$$
where $q=\min \{ (k+1)\tau, \lfloor \mathcal{K} \log N \rfloor \},$ such that  $x_{i_a}(a)=c=x_{i_b}(b)$, where 
$$\Nivel +k\tau    \leq i_a, i_b <  \Nivel +q, $$
in particular
$$x_{\Nivel +q}(\tilde{\omega})= (f_a^{n_a}(c), f_b^{n_b}(c)),$$
where 
$$0\leq n_a,n_b < \tau, \text { with }  n_a \neq n_b.$$
Thus,
\begin{eqnarray} |x_{\Nivel +q}(\tilde{\omega})|&=& |f_a^{n_a}(c)- f_b^{n_b}(c)| \nonumber \\ 
&\geq&  |f_a^{n_a}(c)- f_a^{n_b}(c)| -  | f_a^{n_b}(c)- f_b^{n_b}(c)| \nonumber \\
&\geq& 2 \delta_{\tilde{\omega}} - C_L |a-b| \geq \delta_{\tilde{\omega}}.
\end{eqnarray}
\end{proof}

Since $\delta_{\tilde{\omega}}\geq 0$ depends only on a fixed finite number of  iterations of the family $f_t$, it will be easy to give positive lower bounds to it  that hold for most of the intervals $\tilde{\omega}$. Indeed define
$$\Lambda_{N_0}^\delta=\{ t \in [0,1] \colon \text{  for every $N\geq N_0$ if   }   t \in \omega \in \p_{N-2\lfloor \mathcal{K} \log N \rfloor} \text{ then } \delta_\omega > \delta\}. $$
Note that $\Lambda_{N_0}^\delta \subset \Lambda_{N_0+1}^\delta$. Moreover $\delta' < \delta$ implies $\Lambda_{N_0}^{\delta'}\supset \Lambda_{N_0}^\delta$. 

\begin{lemma} 
Given  $\gamma > 0$ there exists $\delta >0$ such that  
$$\lim_{N_0\rightarrow \infty} |\Lambda_{N_0}^{\delta}|  \geq 1 -\gamma.$$
 \end{lemma} 
\begin{proof} Since $f_t$ is a transversal family, the set of parameters $t$ such that $f_t^i(c)~=f_t^j(c)$ for some $i\neq j$, with $i, j \leq \tau+1$ is finite. Let $t_1, \dots, t_m$ be those parameters. The function $t \rightarrow \delta_t$ is positive and continuous on
$$O=[0,1]\setminus \{ t_1, \dots, t_m\}.$$
Choose $N_0$ large enough such that 
$$\# \{ \omega \in \p_{N_0-2\lfloor \mathcal{K} \log N_0 \rfloor}\colon \overline{\omega} \cap \{ t_1, \dots, t_m\} \neq \emptyset \}\leq 2m.$$
Thus,
$$|\{ \omega \in \p_{N_0-2\lfloor \mathcal{K} \log N_0 \rfloor}\colon \overline{\omega} \subset O \}|\geq 1 - \frac{2C m}{\lambda^{N_0-2\lfloor \mathcal{K} \log N_0 \rfloor}} > 1 - \gamma,$$
provided $N_0$ is large enough. Let
$$\delta:=\frac{1}{2} \min \{ \delta_\omega\colon \ \omega \in \p_{N-2\lfloor \mathcal{K} \log N \rfloor}, \ \overline{ \omega}  \subset O \} .$$
Note that $\delta > 0$ and
$$\Lambda_{N}^\delta \supset \bigcup \{ \omega \in \p_{N-2\lfloor \mathcal{K} \log N \rfloor}\colon \overline{\omega} \subset O \}$$
for every $N \geq N_0$, provided that $N_0$ is large. 
\end{proof}

\begin{proposition}\label{pro_E} 
There exist $\hat{C}_1, \hat{C}_2 > 0$, that do not depend on $\mathcal{K}$,  such that  for every $\mathcal{K}' < \mathcal{K}$  there exists   $K=K(\mathcal{K}') >0$ such that
\begin{equation}
|E_{N}| \le KN^{\hat{C}_2-\hat{C}_1\mathcal{K}' }.
\end{equation}
\end{proposition}

The proof of this proposition follows from 

\begin{lemma}\label{lema_E}  There exists $\hat{C}_1 > 0$, that does not depend on $\mathcal{K}$, such that for every $\mathcal{K}' < \mathcal{K}$ there exists $K=K(\mathcal{K}') >0$ such that
if  $J \in \p_j$, j =j(N),  and $E_{N,J}$ is as defined before, then
\begin{equation}
|E_{N,J}| \le KN^{-\hat{C}_1\mathcal{K}' }. \end{equation}
\end{lemma}

We will prove Lemma \ref{lema_E} later in this section. 

\begin{proof}[Proof of Proposition  \ref{pro_E}] We have 
$$E_{N}=\bigcup_{J \in \p_j}  E_{N,J}.$$
Since there are at most $2^j$ cylinders of level $j$, we have by Lemma \ref{lema_E} that there exist $\hat{C}_1 > 0$ and $K=K(\mathcal{K}')$ such that 
\begin{equation}\label{medida_EN}
|E_N| \le  2^{\left(\frac{\log(C_3N)}{\log \lambda}\right)} KN^{-\hat{C}_1\mathcal{K}' } = KC_3^{\frac{\log 2}{\log \lambda}} N^{\frac{\log 2}{\log \lambda}-\hat{C}_1\mathcal{K}' }.
\end{equation}
\end{proof}
Define
\begin{equation} \label{defOmega}
\Omega_{N_0} = \left\{ t \in [0,1] \colon \forall N \ge N_0 \, \exists \, \omega  \in \p_{\Nivel}  \backslash E_{\Nivel}\mbox{\, and \, } t \in \omega \right\}. 
\end{equation} 
Note that $\Omega_{N_0}\subset \Omega_{N_0+1}$.
\begin{cor}  
If $\hat{C}_2-\hat{C}_1\mathcal{K} < -1$ we have
\begin{equation} 
\lim_{N_0 \rightarrow \infty} |\Omega_{N_0}| =1.\end{equation}
\end{cor} 
\begin{proof} Notice that
$$\Omega_{N_0} = \bigcap_{N\geq N_0}  \bigcup_{ \omega \in \p_{\Nivel}\setminus  E_{\Nivel}   } \omega .  $$
If we  choose $\mathcal{K}' < \mathcal{K}$ such that $\hat{C}_2-\hat{C}_1\mathcal{K}' <  - 1$ we have 
$$|\Omega_{N_0}^c|= \big|\bigcup_{N\geq N_0}  \bigcup_{ \omega  \in  E_{\Nivel}   } \omega \big|\leq \sum_{N\geq N_0} K(\Nivel)^{\hat{C}_2-\hat{C}_1\mathcal{K}' }\stackrel{N_0\to \infty}\longrightarrow 0.$$
\end{proof}

From now on we choose and fix $\mathcal{K} > 0$ satisfying  $\hat{C}_2-\hat{C}_1\mathcal{K} < -1$.

\begin{cor}\label{grande} For every $\gamma > 0$ there exists $\delta >0$ such that 
$$\lim_{N_0\rightarrow \infty} m(\Lambda_{N_0}^\delta \cap \Omega_{N_0}) >  1 -\gamma.$$
\end{cor} 

\begin{definition} \label{defg} Given $\delta > 0$ and $h_0 > 0$, define
$$\Gamma_{h_0}^\delta$$
as the set of all parameters $t \in [0,1]$  such that for every $h$, $0< |h|\leq h_0$, there exists $k$ satisfying 
$$N(t,h)-2 \lfloor \epsilon \log N(t,h) \rfloor   \leq k \leq N(t,h)-\lfloor \epsilon \log N(t,h) \rfloor $$
such that if $t \in \hat{\omega }\in \p_{k}$ then  $|x_{k}(\hat{\omega })| > \delta$. 

Given $t \in \Gamma_{h_0}^\delta$ and $h \neq 0$, let $N_2(t,h)$ be the largest $k$ with this property.   \end{definition}

\begin{definition}Given  $h$ and $t \in [0,1]$, define 
\begin{equation}\label{defN1}
N_1(t,h):= N(t,h) - \lfloor \mathcal{K} \log N(t,h) \rfloor,
\end{equation}
and for $h_0 > 0$ define 
$$
\hat{N}_1(h_0):= \min_{t\in I, |h|\leq h_0} N_1(t,h).
$$
Since 
$$\lim_{N\rightarrow \infty}  \max_{t \in [0,1]} \frac{1}{|Df_t^{N}(f_t(c))|} =0,$$
we have
$$\lim_{h_0\rightarrow 0} \hat{N}_1(h_0) =+\infty.$$
\end{definition} 

\begin{lemma}\label{bc-3} For every $\gamma > 0$ there exists $\delta > 0$ such that 
$$\lim_{h_0\rightarrow 0} m(\Gamma_{ h_0}^\delta) > 1-\gamma.$$
\end{lemma}
\begin{proof} By Corollary \ref{grande} there exist $\delta > 0$ and $N_0$ such that 
$$m(\Lambda_{N_0}^\delta \cap \Omega_{N_0}) > 1-\gamma.$$
Choose $h_0$ such 
$$\hat{N}_1(h_0) > N_0.$$
Let  $|h|\leq h_0$.  Then 
$$N(t,h)-\lfloor \mathcal{K} \log N(t,h) \rfloor \geq N_0.$$
If $t \in\Lambda_{N_0}^\delta \cap \Omega_{N_0}$,   choosing $\tilde{\omega}$ such that  $t \in \tilde{\omega} \in \p_{N(t,h)-\lfloor \mathcal{K} \log N(t,h) \rfloor}$ then $$\tilde{\omega}  \not\in  E_{\nivelh }.$$
Hence, by Lemma \ref{compare} there exists $k$ satisfying (here $N=N(t,h)$)
$$ N-\lfloor \mathcal{K} \log N \rfloor - \lfloor \mathcal{K} \log (N-\lfloor \mathcal{K} \log N \rfloor) \rfloor \leq k \leq N-\lfloor \mathcal{K} \log N\rfloor$$
such that if $t \in \tilde{\omega}\subset \hat{\omega} \in \p_{k}$ then 
$$|x_{k}(\hat{\omega})|\geq \delta_{\hat{\omega}} > \delta$$
since $t \in \Lambda_{N_0}^\delta $, so that $C_L|\hat{\omega}|< \delta < \delta_{\tilde{\omega}}$.  Therefore, $\Gamma_{ h_0}^\delta \supset  \Lambda_{N_0}^\delta \cap \Omega_{N_0}$.
\end{proof}

\begin{definition}  Given $h_0 > 0$ and  $\delta  > 0 $, for every $h$ such that $|h|\leq h_0$ let $\mathcal{A}_{h,h_0}^\delta$ be a covering of $\Gamma^\delta_{h_0}$ by intervals $\omega$ with the following properties
\begin{itemize}
\item[P$_1$.] There exists $t \in \Gamma^\delta_{h_0}$ such that $t \in \omega \in \p_{N_2(t,h)}$.
\item[P$_2$.] If $t' \in \Gamma^\delta_{h_0}$ and $t'\in \omega$ then $\omega' \subset \omega$, where  $t'\in \omega' \in \p_{N_2(t',h)}$.
\item[P$_3$.]  There does not exist  $t'' \in \Gamma^\delta_{h_0}$ such that $t''\in \omega'' \in \p_{N_2(t'',h)}$ and $\omega \subsetneqq \omega''$.
\end{itemize}
One can check that one such collection   $\mathcal{A}_{h,h_0}^\delta$ does exist.  Indeed, consider  the covering of $\Gamma^\delta_{h_0}$ given by 
$$\{ \omega \colon  \text{ there exists } t \in \Gamma^\delta_{h_0} \text{ such that } t \in \omega \in \p_{N_2(t,h)} \}.$$
Of course this covering satisfies property P$_1$. Remove from this covering all intervals  $\omega$ that do not satisfy property P$_3$. Then the remaining  collection  is  a covering of $\Gamma^\delta_{h_0}$ satisfying properties P$_1$, P$_2$ and P$_3$. Note also that the distinct  intervals in $\mathcal{A}_{h,h_0}^\delta$ are pairwise disjoint. Indeed, if $\omega, \omega' \in \mathcal{A}_{h,h_0}^\delta$, with $\omega\neq \omega'$ and $\omega\cap\omega' \neq \emptyset$ then either $\omega \subsetneqq \omega'$ or $\omega' \subsetneqq \omega$, which is in contradiction with property P$_3$. 

We note that $|\mathcal{A}_{h,h_0}^\delta|\geq m(\Gamma^\delta_{h_0})$, since $\mathcal{A}_{h,h_0}^\delta$ covers $\Gamma^\delta_{h_0}$. Here $|\mathcal{A}_{h,h_0}^\delta|$ denotes the Lebesgue measure of the union of the intervals in the family $\mathcal{A}_{h,h_0}^\delta$. 
\end{definition}

\begin{lemma} \label{est_N} 
If $h_0$ is small enough there are $C_5 > 0$ and $C_6 > 0$ such that the following holds. Given $t' \in \Gamma^\delta_{h_0}$, let $\omega$ be the unique interval in $\mathcal{A}_{h,h_0}^\delta$ such that $t'\in \omega$. Let  $t \in \Gamma^\delta_{h_0}$ be such that $t \in \omega \in \p_{N_2(t,h)}$. Then 
\begin{equation}\label{l1}  
\lfloor\frac{\epsilon}{2}  \log N(t',h) \rfloor  \leq N(t',h)-N_2(t,h)\leq C_5 \mathcal{K} \log N(t',h)
\end{equation} 
and 
\begin{equation}\label{l2}  |\omega|\geq C_6 \delta N(t',h)^{\mathcal{K} \frac{\log \lambda}{2}}|h|.\end{equation} 
\end{lemma}
\begin{proof}Consider $\omega'$ such that 
$$t' \in \omega' \in \p_{N_2(t',h)}.$$
Then by property P$_2$ we have $\omega' \subset \omega$. Since
\begin{align*}
\delta &\le |x_{N_2(t',h)}(\omega')| = |\partial_{\theta}f^{N_2(t',h)}(c)||\omega'|\le C_1C_3|Df_{t'}^{N_2(t',h)}(f_{t'}(c))|,
\end{align*}
it follows that
\begin{equation}\label{compr_omega}
\frac{\delta}{C_1C_3}  \frac{1}{|Df_{t'}^{N_2(t',h)}(f_{t'}(c))|}  \leq |\omega'| \leq |\omega|\leq \frac{C_1C_3}{|Df_t^{N_2(t,h)}(f_t(c))|}.
\end{equation}
Since $t, t'\in \omega$, there is $C_1> 1$ such that 
$$  \frac{1}{C_1} \frac{1}{|Df_{t'}^{i}(f_{t'}(c))|} \leq    \frac{1}{|Df_t^{i}(f_t(c))|} \leq  C_1 \frac{1}{|Df_{t'}^{i}(f_{t'}(c))|}.$$
for every $i\leq N_2(t,h)$.  Choose $\bar{C}$ such that

\begin{equation}\label{esc2} 
\frac{\delta}{C_3^2 C_1^3}>  \frac{1}{\lambda^{\bar{C}}}.
\end{equation} 
Then 
$$N_2(t',h)\geq N_2(t,h) - \bar{C},$$
otherwise 
\begin{align} 
\frac{\delta}{C_1C_3}  \frac{1}{|Df_{t'}^{N_2(t',h)}(f_{t'}(c))|}  &\le \frac{C_1C_3}{|Df_t^{N_2(t,h)}(f_t(c))|} \nonumber \\
&\le \frac{C_1C_3}{|Df_t^{N_2(t,h)-N_2(t',h)}(f_t^{N_2(t',h)+1}(c))|} \frac{1}{|Df_t^{N_2(t',h)}(f_t(c))|} \nonumber \\
&\le \frac{C_1C_3}{\lambda^{ \bar{C}}} \frac{C_1}{|Df_{t'}^{N_2(t',h)}(f_{t'}(c))|}, \nonumber 
\end{align} 
which contradicts  Eq. (\ref{esc2}). In particular 
\begin{align*} 
N(t',h) - N_2(t,h)&\ge  N(t',h) - N_2(t',h)- \bar{C} \\
&\ge \lfloor \epsilon  \log N(t',h) \rfloor -\bar{C} \\
&\ge \lfloor\frac{\epsilon}{2}  \log N(t',h) \rfloor .
 \end{align*}
Note that the lower bound holds  if $h_0$ is small enough. Thus, $$N(t',h)  > N_2(t,h).$$ 
Moreover, 
\begin{align*}
   |h|&\le\frac{1}{|Df_{t'}^{N(t',h)}(f_{t'}(c))|} \\
   &\le \frac{1}{|Df_{t'}^{N(t',h)-N_2(t,h)}(f_{t'}^{N_2(t,h)+1}(c))|} \frac{1}{|Df_{t'}^{N_2(t,h)}(f_{t'}(c))|} \\
   &\le\frac{1}{|Df_{t'}^{N(t',h)-N_2(t,h)}(f_{t'}^{N_2(t,h)+1}(c))|} \frac{C_1}{|Df_{t}^{N_2(t,h)}(f_{t}(c))|}. 
\end{align*}
On the other hand, 
\begin{align}     
|h|&\ge  \frac{1}{|Df_{t}^{N(t,h)+1}(f_t(c))|} \nonumber \\ 
   &\ge \frac{1}{|Df_{t}^{N(t,h)-N_2(t,h)}(f_{t}^{N_2(t,h)+1}(c))|} \frac{1}{|Df_{t}^{N_2(t,h)}(f_{t}(c))|}\frac{1}{\Lambda}.
\end{align}
Then
\begin{align*} 
&\log |Df_{t'}^{N(t',h)-N_2(t,h)}(f_{t'}^{N_2(t,h)+1}(c))| -\log C_1\\
& \leq \log |Df_{t}^{N(t,h)-N_2(t,h)}(f_{t}^{N_2(t,h)+1}(c))| + \log \Lambda 
\end{align*}
and consequently 
$$N(t',h)-N_2(t,h) \leq \hat{C}_3 (N(t,h)-N_2(t,h)) + \hat{C}_4.$$
In a similar way, we can obtain  
$$N(t,h)-N_2(t,h) \leq \hat{C}_3 (N(t',h)-N_2(t,h)) + \hat{C}_4, $$
where 
$$\hat{C}_3= \frac{\log \Lambda}{ \log \lambda}$$
and 
$$\hat{C}_4= \frac{\log C_1}{ \log \lambda}.$$
\begin{align} 
N(t,h)&=N(t,h)-N_2(t,h) + N_2(t,h) \nonumber \\
&\le 2\epsilon \log N(t,h) +N(t',h) \nonumber \\
&\le \frac{N(t,h)}{N(t,h)-2\epsilon \log N(t,h)} N(t',h) \nonumber \\
&\le  2  N(t',h),\nonumber 
\end{align} 
provided that $h_0$ is small. Consequently 
\begin{align} 
N(t',h)-N_2(t,h) &\le\hat{C}_3 (N(t,h)-N_2(t,h)) + \hat{C}_4\nonumber \\
 &\le \hat{C}_3 2\lfloor \mathcal{K} \log N(t,h) \rfloor + \hat{C}_4  \nonumber \\
 &\le \hat{C}_3 2\mathcal{K} \log [2 N(t',h)]+ \hat{C}_4  \nonumber \\
 &\le C_5 \mathcal{K} \log N(t',h).
 \end{align} 
Here the last inequality holds if  $h_0$ is small enough.   Moreover, by Eq. (\ref{compr_omega})

\begin{align}  
|\omega|&\ge \frac{1}{C_1C_3} \frac{\delta}{|Df_{t'}^{N_2(t',h)}(f_{t'}(c))|}  = \frac{\delta }{C_1C_3} \frac{|Df_{t'}^{N(t',h)-N_2(t',h)}(f_{t'}^{N_2(t',h)+1}(c))|}{|Df_{t'}^{N(t',h)}(f_{t'}(c))|} \nonumber \\
&\ge  \frac{\delta }{C_1C_3} \frac{\lambda^{N(t',h)-N_2(t',h)}}{|Df_{t'}^{N(t',h)}(f_{t'}(c))|}\geq  \frac{\delta }{C_1C_3} \lambda^{\frac{\mathcal{K} \log N(t',h)}{2}  -1 }|h| = \frac{\delta }{C_1C_3 \lambda} N(t',h)^{\mathcal{K} \frac{\log \lambda}{2}}|h|. 
\end{align}
Hence, we obtain  Eq. (\ref{l2}).\end{proof}

Choose $\epsilon > 0$  such that 
$$  \frac{1}{\sqrt{\lambda}}<    1-\epsilon.$$

\begin{lemma}\label{lema_M}Given  $M> 0$,  define 
$$
B^\delta_{h,h_0,M} = \big\{t : t \in \omega \in \mathcal{A}_{h,h_0}^\delta \mbox{ and }  dist(t, \partial\omega) > \frac{M+1}{1-\epsilon} |h| \big\}.
$$
Let $h_i =(1-\epsilon)^i h_0. $ Given $h$ satisfying $0< | h |\leq   h_0$, let 
$$i(h)= max  \{i\in \mathbb{N}\colon |h|< (1-\epsilon)^{i -1}h_0    \}.$$
For every $h > 0$  define
$$\hat{\Gamma}^\delta_{h,h_0} =  \Gamma^\delta_{h_0} \cap \big( \bigcap_{i\geq  i(h)}   B^\delta_{h_i,h_0,M}\big).  $$
Then
\begin{itemize}
\item[A.] If $0 < \hat{h} < h$ then $\hat{\Gamma}^\delta_{h,h_0}  \subset \hat{\Gamma}^\delta_{\hat{h},h_0}$,
\item[B.] We have 
$$\lim_{h\rightarrow 0} m(\hat{\Gamma}^\delta_{h,h_0})= m(\Gamma^\delta_{h_0}).$$
\end{itemize}
\end{lemma}
\begin{proof}  Note that
$$\min_{t \in [0,1]} N(t,h_i) \ge  -\frac{\log h_0}{\log \Lambda}-1-\frac{ i\log (1-\epsilon)}{\log \Lambda},$$
where
$$
-\frac{\log h_0}{\log \Lambda}>0 \mbox{\;\; and \;\;} -\frac{ i\log (1-\epsilon)}{\log \Lambda}>0.
$$
Therefore, if $h_0$ is samall enough, there are $K_1,K_2>0$, such that
$$
\min_{t \in [0,1]} N(t,h_i) \ge K_1 + iK_2.
$$
Define
$$A_h= \bigcup_{\omega \in \mathcal{A}_{h,h_0}^\delta }\omega.$$
If $\omega \in \mathcal{A}_{h,h_0}^\delta$ then there is $t \in \Gamma^\delta_{h_0}$ such that $t \in  \omega \in \p_{N_2(t,h)}$. By Lemma \ref{est_N} 
\begin{align}
m(\omega\cap (B^\delta_{h,h_0,M})^c)&= m \{ t' \in \omega \colon  dist(t', \partial\omega) \leq  \frac{M+1}{1-\epsilon}|h|\} \nonumber \\ 
&\le  2\frac{M+1}{1-\epsilon}|h| \nonumber \\
&\le\frac{2(M+1)|h|}{(1-\epsilon)|\omega|}  |\omega| \nonumber \\
&\le   \frac{2C_6(M+1)}{\delta (1-\epsilon) N(t,h)^{\mathcal{K} \frac{\log \lambda}{2}} }  |\omega|.   
\end{align} 
Choose $\mathcal{K} $ large enough such that  $\mathcal{K} \log \lambda > 2$. Then
\begin{align}\label{sum1}
\sum_{i=0}^{\infty}  m(A_{h_i}\cap (B^\delta_{h_i,h_0,M})^c)  &\le   \sum_{i=0}^{\infty} \frac{2C_6(M+1)\sqrt{\lambda}}{\delta (K_1 +i K_2)^{\mathcal{K} \frac{\log \lambda}{2}}  }  <  \infty.
\end{align}
In particular 
\begin{align} \label{sum2} 
m\Big(\Gamma^\delta_{h_0} \cap \big( \bigcap_{i\geq  i(h)}   B^\delta_{h_i,h_0,M}\big)\Big) &= m(\Gamma^\delta_{h_0}) - m(\Gamma^\delta_{h_0}  \cap \big( \bigcap_{i\geq  i(h)}   B^\delta_{h_i,h_0,M}\big)^c)  \nonumber \\
&\ge m(\Gamma^\delta_{h_0})-  \sum_{i\geq  i(h)} m(\Gamma^\delta_{h_0}  \cap(B^\delta_{h_i,h_0,M})^c)  \nonumber \\
&\ge m(\Gamma^\delta_{h_0})-  \sum_{i\geq  i(h)} m(A_{h_i}  \cap(B^\delta_{h_i,h_0,M})^c).
\end{align} 
Eq. (\ref{sum1}) implies that 
$$\lim_{h\rightarrow 0} \sum_{i\geq  i(h)} m(A_{h_i}  \cap(B^\delta_{h_i,h_0,M})^c) = 0.$$
\end{proof}

\begin{proof}[Proof of Proposition \ref{main_prop}] 
By Lemma \ref{bc-3} for every $\gamma > 0$ there exists $\delta > 0$ such that for every small $h_0$ we have
$$m(\Gamma^\delta_{h_0}) > 1-\gamma.$$
Choose $M$ satisfying Eq. (\ref{def_M}). Define 
$$
\Gamma^\delta_{h,h_0}=\hat{\Gamma}^\delta_{h,h_0} \backslash Q,
$$
where $\hat{\Gamma}^\delta_{h,h_0} $ is the set defined in Lemma \ref{lema_M} and $Q$ is the countable set of parameters where $f_t$ has a periodic critical point.  By Lemma \ref{lema_M} Property A. holds. 
Let $t' \in \Gamma^\delta_{h,h_0}$, with $|h| < h_0$. There exists $i\geq i(h)$ such that 
$$ h_{i+1} \leq  | h| \leq h_{i},$$
where $h_i=(1-\epsilon)^i h_0$.  Thus, $N(t',h)=N(t',h_j)$, for some  $j \in \{i,i+1\}$, and consequently  $N_2(t',h)=N_2(t',h_j)$. Then there exists a unique  $\omega \in \mathcal{A}^\delta_{h_j,h_0}$ and $t \in \Gamma^\delta_{h_0}$ such that $t,t' \in \omega \in\p_{N_2(t,h)}$. Moreover, since $t' \in B^\delta_{h_j,h_0,M}$ we have 
$$dist(t',\partial \omega)\geq \frac{M+1}{1-\epsilon} h_j \geq (M+1)|h|.$$

Define $N_3(t',h)=N_2(t,h)$. By Lemma \ref{est_N} we have Eq. (\ref{est_N3}) holds.  By Lemma \ref{lemmaB}, Eq. (\ref{semc1}) holds. 

\end{proof}

\subsection{Proof of Lemma \ref{lema_E}}

Let $J$ be the interval as in the statement of Lemma \ref{lema_E}. The sets $E_{N, J}$ `live' in the parameter space. To estimate its measures we will compare them, following \cite{sch1}, with the measures of similarly defined sets in the phase space of the map $f_{t_R}$.

\begin{definition}[The sets $\hat{E}_{N,t_R}$] Let $J=[t_L,t_R]$. Denote by  $\hat{E}_{N,t_R}$ the  set of all $$\eta \in \p_N(t_R)$$ such that for all $k$ satisfying
$$0 \le k \le \left\lfloor\frac{\mathcal{K} \log{N}}{\tau} \right\rfloor$$
there is not  $$\tilde{\eta} \in \p_{\Nivel+j}(t_R), \;\; \eta \subset \tilde{\eta},$$
where 
$$j = \min \{   (k+1) \tau, \lfloor\mathcal{K} \log{N} \rfloor \},$$
such that   $$f_{t_R}^{\Nivel + k\tau}(\tilde{\eta}) \in \p_{j-k\tau}(t_R).$$
\end{definition}

Using a strategy similar to the one applied in  \cite{sch1},  we  estimate the measure  $|E_{N,J}|$ in terms of the measure  $|\hat{E}_{N,t_R}|$. To this end we need to define the map $\mathcal{U}_J$.
Recall that if $\mathcal{F}$ is a family of disjoint intervals then $|\mathcal{F}|$ denotes the sum of the measures of the intervals.
\begin{definition}[The map $\mathcal{U}_J$] 
Let $J= (t_L,t_R)$. Consider  the map $\mathcal{U}_J$ 
$$\mathcal{U}_J: \p_N|_J \to \p_N(t_R)
$$defined by Schnellmann  \cite[proof of Lemma 3.2]{sch1} in the following way.
Let $\omega \in \p_N|_J$ and choose $t \in \omega$. Since $\omega$ is a cylinder, it follows that $x_j(t) \neq c$ for all $0\le j < N$.
Therefore, there is a cylinder $\omega(x_0(t))$ in the partition $\p_N(t)$ such that $x_0(t)\in \omega(x_0(t))$.

Let $$\mathcal{U}_J(\omega) = \mathcal{U}_{t, t_R, N}(\omega(x_0(t))),$$ where $\mathcal{U}_{t, t_R, N}: \p_N(t) \to \p_N(t_R)$ is such that for all $\eta \in \p_N(t)$, the elements $\eta$ and $\mathcal{U}_J(\eta)$ have the same combinatorics.
$$
\mbox{symb}_{t}(f_{t}^i(\eta)) = \mbox{symb}_{t_R}(f_{t_R}^i(\mathcal{U}_{t, t_R, N}(\eta)),
$$
for $0\le i < N$. Schnellmann  \cite{sch1} proved  that $\mathcal{U}_{t, t_R, N}$ is well defined when $f_t$ is a family of piecewise expanding unimodal maps satisfying our assumptions. In particular, if $t < t'$ and a certain symbolic dynamic appears in the dynamics of $f_t$, then it also appears in the dynamics of $f_t'$.

Therefore, the cylinder $\omega'=\mathcal{U}_J(\omega) =\mathcal{U}_{t, t_R, N}(\omega(x_0(t))) $ has the same combinatorics as $\omega$, that is,
$$
\mbox{symb}(x_{j}(\omega)) = \mbox{symb}_{t_R}(f_{t_R}^j(\omega')),
$$
when $0\le j < N$. Since there are not two cylinders in $\p_N(t_R)$ with the same combinatorics, the element $\omega'$ does not depend on the choice of $t \in \omega$. Therefore, $\mathcal{U}_J$ is well defined.
\end{definition} 

\begin{lemma}\label{phase_par} 
If $\omega \in E_{N,J}$, then $\mathcal{U}_J(\omega) \in \hat{E}_{N,t_R}$. Moreover, there exists $C'\ge 1$ such that 
\begin{equation}
\label{sch_est} |\omega| \le C'|\mathcal{U}_{J}(\omega)|.
\end{equation} 
In particular
\begin{equation}\label{medidaENJ}
|E_{N,J}| \le C'|\hat E_{N,t_R}|.
\end{equation}
\end{lemma}
\begin{proof} Note  that $\mathcal{U}_J(\omega) \in \hat{E}_{N,t_R}$ follows from the fact  that $\omega$ and $\mathcal{U}_J(\omega)$ have the same combinatorics \cite{sch1}.  By \cite[Lemma 3.2]{sch1}, there exists a constant $C'\ge 1$ such that 
$$
|\omega| \le C'|\mathcal{U}_{J}(\omega)|.
$$
Thus,
\begin{equation}
|E_{N,J}| \le \sum_{\omega \in E_{N,J}}|\omega| \le \sum_{\omega \in E_{N,J}} C'|\mathcal{U}_{J}(\omega)| \le C'|\hat{E}_{N,t_R}|.
\end{equation}
\end{proof}

\begin{definition} For each $\eta' \in \p_{\Nivel}(t_R)$,  define the set
\begin{align*}
\hat{E}_{N,t_R, \eta'} = \Bigg\{\eta \in \p_N(t_R)&:  \eta \in  \hat{E}_{N,t_R} \ and \  \eta \subset \eta'  \Bigg\}.
\end{align*}
\end{definition}

\begin{lemma}\label{lemma_card_ENJ}
Let $\eta' \in \p_{\Nivel}(t_R)$.  Then
\begin{equation} \label{numero2}
\# \hat{E}_{N,t_{R},\eta'}  \leq 2^{ \left\lfloor\frac{\mathcal{K} \log{N}}{\tau} \right\rfloor +1}.
\end{equation}
\end{lemma}
\begin{proof}  Define
 $$k_0= \left \lfloor\frac{\mathcal{K} \log{N}}{\tau} \right\rfloor.$$
Notice that
 $$N\geq  \Nivel +  k_0  \tau > N -\tau.$$
 If $N=  \Nivel +  k_0  \tau$ define $k_1=k_0$. Otherwise define $k_1=k_0+1$.
 For every $k$ satisfying 
$$0\leq k \leq k_1,$$
define  families of intervals $\mathcal{F}_k$ in the following way. If $k\leq k_0$ define
\begin{equation} \mathcal{F}_k = \{ \hat{\eta} \subset \eta'\colon \ \hat{\eta} \in  \p_{\Nivel + k \tau}(t_R) \text { and there is }    \eta \in \hat{E}_{N,t_{R},\eta'} \text {with } \eta \subset \hat{\eta} \}\end{equation} 
otherwise $k=k_1=k_0+1$ and 
\begin{equation} 
\mathcal{F}_{k_1} =  \hat{E}_{N,t_{R},\eta'}.
\end{equation} 
Note that if $k_1=k_0$ then we also have $\mathcal{F}_{k_1} =  \hat{E}_{N,t_{R},\eta'}$.
We claim that 
\begin{equation}\label{numero} 
 \# \mathcal{F}_k  \leq 2^{k}.
\end{equation} 
We observe that, taking $k=k_1$ in Eq. (\ref{numero}) we obtain Eq. (\ref{numero2}).  Note that either $\mathcal{F}_0$ is the empty set or $\mathcal{F}_0=\{ \eta' \}$.  Then $\#\mathcal{F}_0\leq 1$. Moreover, it is easy to see  that if $\hat{\eta}_{k+1} \in \mathcal{F}_k$, with $k < k_1$, then there exists a unique $\hat{\eta}_{k} \in \mathcal{F}_{k}$ such that $\hat{\eta}_{k+1} \subset \hat{\eta}_{k}$. Therefore, it is  enough to show that  for each $\hat{\eta}_{k} \in \mathcal{F}_{k}$, with $k < k_1$, there are at most two intervals $\hat{\eta}_{k+1} \in \mathcal{F}_{k+1}$ such that $\hat{\eta}_{k+1} \subset \hat{\eta}_{k}$. Indeed, given $k < k_1$, for every $\hat{\eta}_{k} \in \mathcal{F}_{k}$ we have  $\hat{\eta}_{k} \in \p_{\Nivel+k\tau}(t_R)$. Moreover, there is $j$ such that  for every  $\hat{\eta}_{k+1} \in \mathcal{F}_{k+1}$ we have  $\hat{\eta}_{k+1} \in \p_{\Nivel+j}(t_R)$, with $k\tau < j \leq \lfloor \mathcal{K} \log N \rfloor$, and $j \leq k\tau+\tau$.  Note that if the closure of $\hat{\eta}_{k+1}=(a,b)$ is contained in the interior of $\hat{\eta}_{k}$, then for every  $x \in \overline{\hat{\eta}_{k+1}}$ we have $f_{t_R}^p(x)\neq c$, for every $p < \Nivel+k\tau$. Furthermore, there are $n_a, n_b$ such that 
$$ f_{t_R}^{n_a}(a)=c=f_{t_R}^{n_b}(b),$$
where 
$$\Nivel+k\tau \leq  n_a, \ n_b < \Nivel+j.$$
We conclude that 
$$f_{t_R}^{\Nivel+k\tau}(\hat{\eta}_{k+1}) \in \p_{j-k\tau}(t_R).$$
where $j - k\tau  \leq \tau$. Therefore, if $\eta \subset \hat{\eta}_{k+1}$, with $\eta \in \p_{N}(t_R)$, then $\eta \not\in \hat{E}_{N,t_{R},\eta'}$ and consequently $\hat{\eta}_{k+1} \not\in \mathcal{F}_{k+1}$. Since there are at most two intervals $ \p_{\Nivel+j}(t_R)$ whose closure is not contained in the interior of $\hat{\eta}_{k}$, we conclude that there are at most two intervals in $\mathcal{F}_{k+1}$ that are contained in $\hat{\eta}_{k}$.

\end{proof}

\begin{lemma}\label{lemma_imagens}
Let $\eta', \eta'' \in \p_{\Nivel}(t_R)$ such that 
$$
f_{t_R}^{\Nivel}(\eta')=f_{t_R}^{\Nivel}(\eta'').
$$
Then
$$
f_{t_R}^{\Nivel}(\hat{E}_{N,{t_R},\eta'})=f_{t_R}^{\Nivel}(\hat{E}_{N,{t_R},\eta''}).
$$
\end{lemma}
\begin{proof}
Let $\omega'=(y_1',y_2') \in \p_{N}(t_R)$, with $\omega' \subset \eta'$,  be a cylinder in $\hat{E}_{N,{t_R},\eta'}$. Then
\begin{equation}\label{eq_imagem_omega}
f_{t_R}^{\Nivel}(\omega') \subset f_{t_R}^{\Nivel}(\eta') = f_{t_R}^{\Nivel}(\eta'').
\end{equation}
Remember that since $\omega' \in \p_N(t_R)$, it follows that for all $x \in \omega'$
\begin{equation}\label{cilindro_omega}
f_{t_R}^i(x) \neq c \mbox{\;\; for all \;\;} 0 \le i < N,
\end{equation}
 and if $y \in \partial \omega'$, then there exists $j$, $0\le j < N$ such that $f_{t_R}^j(y)=c$. Define
 $$a_i = f_{t_R}^{\Nivel}(y'_i).$$
Then $f_{t_R}^{\Nivel}(\omega')=(a_1,a_2)$ is an open interval and, by Eq. (\ref{eq_imagem_omega}), we have $(a_1,a_2) \subset f_{t_R}^{\Nivel}(\eta'')$. Therefore, there is an open interval $\omega''= (y_1'',y_2'') \subset \eta''$ such that $f_{t_R}^{\Nivel}(\omega'')=(a_1,a_2)$ with 
$$a_i = f_{t_R}^{\Nivel}(y''_i).$$
We claim that  $\omega''$ is also a cylinder. Indeed, let $x\in \omega''$. Then, since $\omega'' \subset \eta''$ and $\eta''$ is a cylinder of level $\Nivel$, it follows that 
$$
f_{t_R}^i(x) \neq c,
$$
for all $1 \le i < \Nivel$. On the other hand, 
$$
f_{t_R}^{\Nivel}(\omega'')=f_{t_R}^{\Nivel}(\omega'),
$$
and by Eq. (\ref{cilindro_omega}), we can conclude that $f_{t_R}^i(x) \neq c$ for all $i$ satisfying $\Nivel \le i < N$. Therefore, for all $x\in \omega''$, we have $f_{t_R}^i(x) \neq c$ for all $0 \le i < N$.
Now, let $y_i'' \in \partial \omega_2$. Since $\omega'' \subset \eta''$, we have two cases. \\

\noindent {\bf Case $1$:} $y_i'' \in \partial \eta''$. In this case, there is an integer $j$, $0\le j < \Nivel$, such that  $f_{t_R}^j(y_i'')=c$.\\

\noindent {\bf Case $2$:} $y_i''\notin \partial \eta''$. In this case, $f_{t_R}^j(y_i'')\neq c$ for all $0\le j < \Nivel$. Then 
$f_{t_R}^{\Nivel}(y_i'')=a_i=f_{t_R}^{\Nivel}(y_i')$ belongs to the interior of $  f_{t_R}^{\Nivel}(\eta'')=f_{t_R}^{\Nivel}(\eta')$. Thus, $y_i'$ belongs to the interior of $\eta'$, which implies that there exists $j$ such that  $\Nivel\leq j< N$ such that $f_{t_R}^{j}(y_i')=f_{t_R}^{j}(y_i'')=c$.\\

\noindent Therefore $\omega'' \in \p_{N}(t_R)$.

By assumption, $\omega' \in \hat{E}_{N,{t_R},\eta'}$. Then for all $0 \le k \le \lfloor \frac{\mathcal{K} \log N}{\tau} \rfloor $, if 
$$\tilde{\omega}_k \in \p_{\Nivel + j(k)}(t_R),$$ 
where $\omega' \subset \tilde{\omega}_k \subset \eta' $ and 
$$j(k) = \min \{   (k+1) \tau, \lfloor\mathcal{K} \log{N} \rfloor \},$$
then there is  $z_k' \in \partial \tilde{\omega}$ satisfying 
\begin{equation}\label{fronteira_omegatil23}
f_{t_R}^{q_k'}(z_k')= c, \text{ for some $q_k'$, } 0\le q_k' <\Nivel +k \tau.
\end{equation}
In the same manner as for $\omega'$, there exists a unique cylinder $\hat{\omega}_k \in \p_{\Nivel+j(k)}$, $\hat{\omega}_k \subset \eta''$, such  $f_{t_R}^{\Nivel}(\tilde{\omega}_k)=f_{t_R}^{\Nivel}(\hat{\omega}_k)$. Note that    $\omega''\subset \hat{\omega}_k$. Let $z_k''\in \partial \hat{\omega}_k$ such that $$f_{t_R}^{\Nivel}(z_k')=f_{t_R}^{\Nivel}(z_k'').$$ 
If $z_k''\in \partial \eta''$ then there exists $i < \Nivel$ such that $f_{t_R}^{i}(z_k'')= c$. Define $q_k''=i$. \\
If $z_k'' \not \in \partial \eta''$ then $z_k' \not\in \partial \eta'$. Thus, $f_{t_R}^{q}(z_k')\neq c$ for every $q < \Nivel$, which implies that 
$$\Nivel \leq q_k' < \Nivel+k \tau.$$ Then  $f_{t_R}^{q_k'}(z_k'')=f_{t_R}^{q_k'}(z_k'')=c$. Define $q_k''=q_k'$. \\
In both cases we have $0\le q_k'' <\Nivel +k \tau$, then $\omega''\in \hat{E}_{N,{t_R},\eta''}$.
$$
f_{t_R}^{\Nivel}(\hat{E}_{N,{t_R},\eta''})\subset f_{t_R}^{\Nivel}(\hat{E}_{N,{t_R},\eta'}).
$$
A similar argument  shows that
$$
f_{t_R}^{\Nivel}(\hat{E}_{N,{t_R},\eta'})\subset f_{t_R}^{\Nivel}(\hat{E}_{N,{t_R},\eta''}).
$$
\end{proof}

\begin{proof}[Proof of Lemma \ref{lema_E}] Due to Lemma \ref{phase_par} it is enough to show that for every $\mathcal{K}' < \mathcal{K}$  there exists $C > 0$ and  $K=K(\mathcal{K}') >0$ such that
if  $J \in \p_j$, $j =j(N)$ then 
\begin{equation} |\hat{E}_{N,t_R}| \leq KN^{-C\mathcal{K}' }.\end{equation}
By Lemma \ref{lemma_card_ENJ} we have
$$
\# \hat{E}_{N,t_R, \eta'}  \leq 2^{\lfloor \frac{\mathcal{K}\log{N}}{\tau} \rfloor+1}.
$$
Let us define the set
$$
\Omega=\bigcup_{\eta' \in \p_{\Nivel}(t_R)} f_{t_R}^{\Nivel}(\hat{E}_{N,t_R,\eta'}).
$$
Note that
$$\hat{E}_{N,t_R}\subset f_{t_R}^{-(\Nivel)}(\Omega).$$
Therefore, if $\mu_{t_{R}}$ is the acip for $f_{t_R}$ we have 
\begin{equation}\label{upper}  \mu_{t_{R}}(\hat{E}_{N,t_R}) \leq \mu_{t_{R}}(\Omega).\end{equation}
In \cite[Section 6.2]{sch1} it is shown that there is $C'_1\ge 1$ such that for every density $\rho_t$ of the unique acip of $f_t$ 
$$
\frac{1}{C'_1}\le \rho_{t}(x)\le C'_1,
$$
for $\mu_{t}$-almost every $x\in [0,1]$, then
$$|\hat{E}_{N,t_R}|\leq  {C'_1}^2  |\Omega|.$$
Since $J \in \mathcal{P}_j$, $j=j(N)$, there exists an integer $p$, $0\le p< j$ such that $x_p(t_R)=f_{t_R}^p(f_{t_R}(c))=c$. In particular
$$\# \{ f^i_{t_R}(c)    \}_{i \geq 0} = p+1.$$
Thus,
$$\# \{ f_{t_R}^{\Nivel}(\eta'), \ \eta' \in    \p_{\Nivel}(t_R)   \}\leq (p+1)^2.$$
Therefore, by Lemma \ref{lemma_imagens},
\begin{align*}
|\hat{E}_{N,t_R}|&\le{C'_1}^2 |\Omega|  ={C'_1}^2| \cup_{\eta' \in \p_{\Nivel}(t_R)} f_{t_R}^{\Nivel}(\hat{E}_{\eta'}))| \\
&\le {C'_1}^2 (p+1)^2 \max_{\eta' \in    \p_{\Nivel}(t_R) }| {f_{t_R}^{\Nivel}(\hat{E}_{\eta'})}|\\
&\le {C'_1}^2(p+1)^2\left(\frac{1}{\lambda}\right)^{\lfloor \mathcal{K}\log{N} \rfloor } \#\left\{\eta \in \p_N(t_R)|_{\hat{E}_{\eta'}}\right\}\\
&\le {C'_1}^2 (p+1)^2\left(\frac{1}{\lambda}\right)^{\lfloor \mathcal{K}\log{N} \rfloor} 2^{\lfloor \frac{\mathcal{K}\log{N}}{\tau} \rfloor+1}\le {C'_1}^2 (p+1)^2 \left(\frac{1}{\lambda}\right)^{\frac{\lfloor\mathcal{K}\log{N}\rfloor}{2}}\\
&\le {C'_1}^2  j^2 \left(\frac{1}{\lambda}\right)^{\frac{\lfloor\mathcal{K}\log{N}\rfloor}{2}}\le {C'_1}^2 \left(\left\lfloor\frac{\log(C_3N)}{\log \lambda}\right\rfloor\right)^2 \left(\frac{1}{\lambda}\right)^{\frac{\lfloor\mathcal{K}\log{N}\rfloor}{2}}\\
&\leq K  N^{-\frac{\log \lambda}{2}\epsilon'}
\end{align*}
where $K=K(\epsilon')$.
\end{proof}


\section{Estimates for the Wild part}\label{sec_wild}
We start this section with a technical lemma.
\begin{lemma}\label{norma_L1} 
Given a good transversal family  $f_t$ there are constants $L_1$ and $L_2$ such that the following holds.  
Let $\varphi:[0,1]\to \R$, $|\varphi|_{L^1(m)} >  0$,  be a function of bounded variation such that
$$
\int \varphi dm =0.
$$
Then
$$
\norm{(I-\PF_t)^{-1}(\varphi)}_{L^1} \le \left(L_1\log\frac{\norm{\varphi}_{BV}}{\norm{\varphi}_{L^1}}+L_2 \right)\norm{\varphi}_{L^1}.
$$
\end{lemma}
\proof
\noindent Let $\tilde{\jmath}>0$ such that
$$
L\tilde{\beta}^{\tilde{\jmath}}\norm{\varphi}_{BV}  = \norm{\varphi}_{L^1}.
$$
And let $j_0$ the smallest integer such that $j_0-1\le \tilde{\jmath} \le j_0$. Hence, we have
$$
(I-\PF_t)^{-1}(\varphi)= \sum_{i=0}^{j_0} \PF_{t}^i(\varphi) + \sum_{l=1}^{\infty} \PF_{t}^l(\PF_t^{j_0}(\varphi)).
$$
Observing Assumption $(V)A_1$, the fact that $|\PF_t^l\varphi|_{BV} \le \tilde{L}\theta^l|\varphi|_{BV}$, when $\int \varphi \, dm = 0$ with constants $\tilde{L}$ and $\theta$ uniform in $t$, as well as the elementary facts that $|\PF_t|_{L^1} = 1$ and $|\cdot|_{L^1}\le |\cdot|_{BV}$, we see that
\begin{align*}
|(I-\PF_t)^{-1}(\varphi)|_{L^1} &\le (j_0+1)|\varphi|_{L^1}+ \frac{\tilde{L}}{1-\theta}|\PF_t^{j_0}\varphi|_{BV}\\
&\le (j_0+1)|\varphi|_{L^1}+ \frac{\tilde{L}}{1-\theta}\left(\tilde{C}_6\beta^{j_0}|\varphi|_{BV}+\tilde{C}_5|\varphi|_{L^1}\right)\\
&\le \left(c_1\beta^{j_0}\frac{|\varphi|_{BV}}{|\varphi|_{L^1}}+(j_0+c_2)\right)|\varphi|_{L^1}.
\end{align*}
By the choice of $j_0$, we have the desired estimate.

\endproof

The following proposition  will be quite important to study the Wild part of the decomposition. Denote
$$\mbox{supp}(\psi)= \overline{\{ x \in [0,1]\colon \ \psi(x)\neq 0\}}.$$

\begin{proposition}\label{var_LN3b}
 There exist $K, K_1', K_2' >0$ such that the following holds. For all $i,k\geq 0$, $t\in [0,1]$ and  $h\neq 0$, let 
$$\varphi_{k,i,h}=\frac{1}{h}  \PF^{i}_{t+h}\left(H_{f_{t+h}(f_t^k(c))} - H_{f_{t}(f_t^k(c))}\right).$$
Then
\begin{equation} \label{wild34}
\big| \varphi_{k,i,h} \big|_{L^1}  \le K,
\end{equation} 
and
\begin{equation} \label{wild345}
\big| \varphi_{k,i,h} \big|_{BV}  \le \frac{K}{|h|}.
\end{equation} 
Furthermore,
\begin{equation}\label{logbv}
\norm{(I-\PF_{t+h})^{-1}\Pi_{t+h} (\varphi_{k,i,h}) }_{L^1} \leq K_1' \max \{0, \log |\varphi_{k,i,h}|_{BV}\}+ K_2'.\\
\end{equation}
\end{proposition}
\begin{proof}
Note that 
 \begin{align}\label{est1}
&\big| \PF_{t+h}^{i}\left(H_{f_{t+h}(f_t^k(c))} - H_{f_{t}(f_t^k(c))}\right)\big|_{L^1} \nonumber \\ 
&\leq \big| H_{f_{t+h}(f_t^k(c))} - H_{f_{t}(f_t^k(c))} \big|_{L^1} \nonumber \\
&\leq  (\sup_t  |v_t|) |h|,
\end{align}
and, by Assumption (V) in Definition \ref{goodf}
 \begin{align}\label{est11}
&\big| \PF_{t+h}^{i}\left(H_{f_{t+h}(f_t^k(c))} - H_{f_{t}(f_t^k(c))}\right)\big|_{BV} \nonumber \\ 
&\leq 2\tilde{C}_6 \beta^i + \tilde{C}_5 (\sup_t  |v_t|) |h| \leq  \hat{C}.
\end{align}
Thus, we have Eqs. (\ref{wild34}) and (\ref{wild345}). In particular 
$$|\Pi_{t+h}(\varphi_{k,i,h})|_{L^1(m)}\leq 2|\varphi_{k,i,h}|_{L^1(m)}\leq 2K,$$
and if $h$ is small 
$$|\Pi_{t+h}(\varphi_{k,i,h})|_{BV}\leq  |\varphi_{k,i,h}|_{BV}   +|\varphi_{k,i,h}|_{BV} \sup_{t \in [0,1]} |\rho_t|_{BV} \leq C |\varphi_{k,i,h}|_{BV} ,$$
where $C \geq 1$. 

\noindent Now we can easily obtain Eq. (\ref{logbv}) applying Lemma \ref{norma_L1}.

\end{proof}

\begin{proposition}\label{var_LN3} Let $\phi$ be a Lipchitz function. 
There exists $K>0$ such that the following holds. Let $t\in \Gamma_{h',h_0}^\delta$ and $0< |h|\leq h'$. Then 
\begin{equation} \label{wild33}
\emph{var}\big(  \frac{1}{h} \PF^{i}_{t+h}\left(H_{f_{t+h}(f_t^k(c))} - H_{f_{t}(f_t^k(c))}\right) \big)  \le \frac{K}{|h| |Df_{t}^{i}(f_{t}^{k+1}(c))|},
\end{equation} 
and 
 \begin{align} \label{dirac2}
 &\int \phi(x) \PF_{t+h}^i \left(\frac{H_{f_{t+h}(f_t^k(c))} - H_{f_{t}(f_t^k(c))}}{h} \right) (x) \ dx \nonumber  \\
 &=\phi(f_{t}^{i+k+1}(c))v_t(f^k_t(c)) + O(|Df_{t}^i(f_{t}^{k+1}(c))| |h|),
 \end{align}
where $0\le k \le N_3(t,h)$ and $i < N_3(t,h)- k$.
\end{proposition}
\begin{proof} 
By Eq. (\ref{semc1}), the points $f_{t+h}^{k+1}(c), f_{t+h}(f_t^{k}(c)), f_t(f_{t}^{k}(c))$ belong to the same interval  of monotonicity  of $f_{t+h}^{i}$.  Let 
$$\zeta\colon Dom(\zeta) \rightarrow Im(\zeta)$$
be an inverse branch associated to such interval  of monotonicity, that is, $\zeta$ is a diffeomorphism such that $f^{i}_{t+h}(\zeta(y))=y$ for every $ y\in Dom(\zeta)$ and 
$$\{ f_{t+h}^{k+1}(c), f_{t+h}(f_t^{k}(c)), f_t(f_{t}^{k}(c))\}\subset Im(\zeta).$$
Hence,
\begin{align}\label{1ramo}
&\PF_{t+h}^{i}\left(H_{f_{t+h}(f_t^k(c))} - H_{f_{t}(f_t^k(c))}\right) (x) \nonumber\\
&=\frac{1}{Df^{i}_{t+h}(\zeta(x))}\fc_{Dom(\zeta)}(x)\left(H_{f_{t+h}(f_t^k(c))}(\zeta(x))-H_{f_{t}(f_t^k(c))}(\zeta(x))\right).
\end{align}
There is a constant $K\geq 1$ such that for all $t \in [0,1]$, $h$, and $i$, and every interval of monotonicity $Q$ of $f_{t+h}^i$ we have
$$
\frac{1}{K}\leq \left|\frac{Df_{t+h}^i(y_1)}{Df_{t+h}^i(y_2)}\right|\leq K 
$$ for all $y_1,y_2 \in Q$.  Now we can estimate  the variation of the function in Eq. (\ref{wild33}) using familiar properties of the variation of  functions (see Chapter $3$ from Viana \cite{viana}, for instance).
\begin{align*}
&\mbox{var}_{[0,1]}\left(\PF_{t+h}^{i}\left(H_{f_{t+h}(f_t^k(c))} - H_{f_{t}(f_t^k(c))}\right)\right) \\
&=\mbox{var}_{[0,1]}\left(\frac{1}{Df^{i}_{t+h}(\zeta(x))}\fc_{Dom(\zeta)}(x)\left(H_{f_{t+h}(f_t^k(c))}(\zeta(x))-H_{f_{t}(f_t^k(c))}(\zeta(x))\right)\right)\\
&=\mbox{var}_{Dom(\zeta)}\left(\frac{1}{Df^{i}_{t+h}(\zeta(x))}\right)\sup_{[0,1]}\left(H_{f_{t+h}(f_t^k(c))} - H_{f_{t}(f_t^k(c))}\right)\\
&+2\sup_{[0,1]}\left(\frac{1}{Df^{i}_{t+h}(\zeta(x))}\fc_{Dom(\zeta)}(x)\right)\sup_{[0,1]}\left(H_{f_{t+h}(f_t^k(c))} - H_{f_{t}(f_t^k(c))}\right)\\
&+\sup_{[0,1]}\left(\frac{1}{Df^{i}_{t+h}(\zeta(x))}\fc_{Dom(\zeta)}(x)\right)\mbox{var}_{[0,1]}\left(H_{f_{t+h}(f_t^k(c))} - H_{f_{t}(f_t^k(c))}\right)\\
&\leq 2 \mbox{var}_{Dom(\zeta)}\left(\frac{1}{Df^{i}_{t+h}(\zeta(x))}\right)+ \frac{6K}{|Df^{i}_{t+h}(f_{t+h}^{k+1}(c))|}.
\end{align*}
Now, note that since $\zeta$  is a diffeomorphism, it follows that 
\begin{align*}
&\mbox{var}_{Dom(\zeta)}\left(\frac{1}{Df^{i}_{t+h}(\zeta(x))}\right)=\mbox{var}_{Im(\zeta)}\left(\frac{1}{Df^{i}_{t+h}(y)}\right)\\
&=\int_{Im(\zeta)}\left|D\left(\frac{1}{Df^{i}_{t+h}(y)}\right)\right|dy \\
&= \int_{Im(\zeta)}\left|\sum_{j=1}^{i}-\frac{D^2f_{t+h}(f_{t+h}^{j-1}(y))}{Df_{t+h}^{i-j}(f_{t+h}^j(y))\left(Df_{t+h}(f_{t+h}^{j-1}(y))\right)^2}\right|dy\\
&\le K_1|Im(\zeta)|\le K_1\frac{|Dom(\zeta)|}{|Df_{t+h}^{i}(f_{t+h}^{k+1}(c))|}\\
&\le \frac{C K_2}{|Df_{t+h}^{i}(f_{t+h}^{k+1}(c))|}.
\end{align*}
Here we used that
\begin{eqnarray} & \ &\left| \sum_{j=1}^{i}-\frac{D^2f_{t+h}(f_{t+h}^{j-1}(y))}{Df_{t+h}^{i-j}(f_{t+h}^j(y))\left(Df_{t+h}(f_{t+h}^{j-1}(y))\right)^2}\right| \nonumber \\
&\leq& \sum_{j=1}^{i}   \frac{C}{\lambda^{i-j}}\nonumber \\
&\leq& K_1,\end{eqnarray}
and that 
$$|Im(\zeta)|\leq K \frac{|Dom(\zeta)|}{|Df_{t+h}^{i}(f_{t+h}^{k+1}(c))|}\leq K \frac{1}{|Df_{t+h}^{i}(f_{t+h}^{k+1}(c))|}.$$

\noindent Therefore,
\begin{align}\label{estbv}
&\mbox{var}_{[0,1]}\left(\PF_{t+h}^{i}\left(H_{f_{t+h}(f_t^k(c))} - H_{f_{t}(f_t^k(c))}\right)\right)\le\frac{K_3}{|Df_{t+h}^{i}(f_{t+h}^{k+1}(c))|}.
\end{align}
Finally, by Eq. (\ref{semc1}) note that the combinatorics up to $i$ iterations of $f_{t+h}^{k+1}(c)$ by the map $f_{t+h}$ is the same as the combinatorics up to $i$ iterations of $f_t^{k+1}(c)$ by the map $f_{t}$. By Remark \ref{const2} we obtain
\begin{equation}\label{distrr}
\frac{1}{|Df_{t+h}^{i}(f_{t+h}^{k+1}(c))|}\leq C_1 \frac{1}{|Df_{t}^{i}(f_{t}^{k+1}(c))|}.
\end{equation} 
Eqs. (\ref{distrr}) and (\ref{estbv}) give us Eq. (\ref{wild33}). Since
 $$\mbox{supp} \frac{1}{h}\left(H_{f_{t+h}(f_t^k(c))} - H_{f_{t}(f_t^k(c))}\right) = [f_{t+h}(f_t^k(c)), f_{t}(f_t^k(c))],$$
 by Eq. (\ref{1ramo}) we conclude that 
 $$ Z_{i,k}=\mbox{supp} \frac{1}{h}\PF_{t+h}^i \left(H_{f_{t+h}(f_t^k(c))} - H_{f_{t}(f_t^k(c))}\right)=[f^{i+1}_{t+h}(f_t^k(c)), f^{i}_{t+h}(f_t^{k+1}(c))].$$
 By Eq. (\ref{semc1}), the points $f_{t+h}^{k+1}(c), f_{t+h}(f_t^{k}(c)), f_t(f_{t}^{k}(c))$ belong to the same interval  of monotonicity  of $f_{t+h}^{i}$. Hence,
 \begin{align}\label{supp}
 &diam \  \mbox{supp} \frac{1}{h}\PF_{t+h}^i \left(H_{f_{t+h}(f_t^k(c))} - H_{f_{t}(f_t^k(c))}\right)  \nonumber \\
&= diam  \ [f^{i+1}_{t+h}(f_t^k(c)), f^{i}_{t+h}(f_t^{k+1}(c))] \nonumber \\ 
 &= |f^{i+1}_{t+h}(f_t^k(c))- f^{i}_{t+h}(f_{t}^{k+1}(c))| \nonumber \\
 &\le  K|Df_{t+h}^i(f_{t}^{k+1}(c))||f_{t+h}(f_t^k(c))- f_{t}(f_t^k(c))| \nonumber \\
  &\le K|Df_{t+h}^i(f_{t}^{k+1}(c))||\sup_t v_t| |h| \nonumber \\
   &\le  C_1 K|Df_{t}^i(f_{t}^{k+1}(c))||\sup_t v_t| |h|.
 \end{align}
 Therefore,
 \begin{align}
 &\int \phi(x) \PF_{t+h}^i \left(\frac{H_{f_{t+h}(f_t^k(c))} - H_{f_{t}(f_t^k(c))}}{h} \right) (x) \ dx \nonumber  \\
 &= \phi(f_{t}^{i+k+1}(c)) \int \PF_{t+h}^i \left(\frac{H_{f_{t+h}(f_t^k(c))} - H_{f_{t}(f_t^k(c))}}{h} \right) (x) \ dx \nonumber \\
 &+    \int \left(\phi(x)-\phi(f_{t}^{i+k+1}(c))\right) \PF_{t+h}^i \left(\frac{H_{f_{t+h}(f_t^k(c))} - H_{f_{t}(f_t^k(c))}}{h} \right) (x) \ dx.
 \end{align}
 Note that 
\begin{align}
 &\int \PF_{t+h}^i \left(\frac{H_{f_{t+h}(f_t^k(c))} - H_{f_{t}(f_t^k(c))}}{h} \right) (x) \ dx \nonumber \\
&= \int \frac{H_{f_{t+h}(f_t^k(c))} - H_{f_{t}(f_t^k(c))}}{h} (x)dx = v_t(f_t^k(c))+ O(|h|).
\end{align}
Due to Eq (\ref{supp}) and the fact that $\phi$ is a lipschitzian function with Lipschitz constant $L$, and that $f_{t}^{i+k+1}(c) \in Z_{i,k}$
 \begin{align}
 &\big| \int \left(\phi(x)-\phi(f_{t}^{i+k+1}(c))\right) \PF_{t+h}^i \left(\frac{H_{f_{t+h}(f_t^k(c))} - H_{f_{t}(f_t^k(c))}}{h} \right) (x) \ dx \big| \nonumber \\
 &\le\int_{Z_{i,k}} |\phi(x)-\phi(f_{t}^{i+k+1}(c))| \big| \PF_{t+h}^i \left(\frac{H_{f_{t+h}(f_t^k(c))} - H_{f_{t}(f_t^k(c))}}{h} \right) (x)\big| \ dx \nonumber  \\
 &\le L  C_1 K|Df_{t}^i(f_{t}^{k+1}(c))||\sup_t v_t| |h| \big| \PF_{t+h}^i \left(\frac{H_{f_{t+h}(f_t^k(c))} - H_{f_{t}(f_t^k(c))}}{h} \right)\big|_{L^1} \nonumber \\
 &\le L  C_1 K|Df_{t}^i(f_{t}^{k+1}(c))||\sup_t v_t|^2 |h|.  
 \end{align} 
 
\end{proof}

\begin{proof}[Proof of Proposition \ref{wild_part}] 
Let $\Phi_h$ be as in  Proposition \ref{parteB}, that is 
$$
\Phi_h = \frac{1}{h}\sum_{k=0}^{\infty}s_{k+1}(t)\Pi_{t+h}\left(H_{f_{t+h}(f_t^k(c))}- H_{f_{t}(f_t^k(c))}\right).
$$
Given $t \in \Gamma^\delta_{h,h_0}$.  Let  $\Ncerto(t,h)$ be as in  Proposition \ref{main_prop}. Since $t$ and $h$ are fixed throughout this proof, we will write $\Ncerto$ instead of $\Ncerto(t,h)$ and $N$ instead of $N(t,h)$. Let us divide $\Phi_h$   as follows
$$
\Phi_h = S_{1} + S_{2}.
$$
Where
\begin{align*}
S_{1}=\frac{1}{h}\sum_{k=0}^{\Ncerto}s_{k+1}(t)\Pi_{t+h}\left(H_{f_{t+h}(f_t^k(c))}- H_{f_{t}(f_t^k(c))}\right)
\end{align*}
and
\begin{align*}
S_{2}=\frac{1}{h}\sum_{k=\Ncerto+1}^{\infty}s_{k+1}(t)\Pi_{t+h}\left(H_{f_{t+h}(f_t^k(c))}- H_{f_{t}(f_t^k(c))}\right).
\end{align*}
Let us first estimate $S_{2}$.
\begin{align*}
(I-\PF_{t+h})^{-1}S_{2}=\frac{1}{h}&\sum_{k=\Ncerto+1}^{\infty}s_{k+1}(t)(I-\PF_{t+h})^{-1}\Pi_{t+h}\left(H_{f_{t+h}(f_t^k(c))} - H_{f_{t}(f_t^k(c))}\right).
\end{align*}
Thus,
\begin{align*}
&\norm{(I-\PF_{t+h})^{-1}S_{2}}_{L^1}\\
&\le\sum_{k=\Ncerto+1}^{\infty}|s_{k+1}(t)|
\norm{\frac{1}{h}(I-\PF_{t+h})^{-1}\Pi_{t+h}\left(H_{f_{t+h}(f_t^k(c))}- H_{f_{t}(f_t^k(c))}\right)}_{L^1}.
\end{align*}
\noindent By Proposition \ref{var_LN3b} and Lemma \ref{norma_L1}, taking 
$$
\varphi = \frac{1}{h} \Pi_{t+h}\left(H_{f_{t+h}(f_t^k(c))} - H_{f_{t}(f_t^k(c))}\right),
$$
we have,
\begin{align*}
&\norm{(I-\PF_{t+h})^{-1}\frac{1}{h}\Pi_{t+h}\left(H_{f_{t+h}(f_t^k(c))}- H_{f_{t}(f_t^k(c))}\right)}_{L^1}\\
&\le K_1 \log\frac{1}{|h|} +K_2\le K_1 \log\Lambda^{N+1} +K_2\\
&\le K_1(N+1)\log \Lambda +K_2 \le K_3N + K_4.
\end{align*}
Therefore, 
\begin{align*}
\norm{(I-\PF_{t+h})^{-1}S_{2}}_{L^1}&\le\sum_{k=\Ncerto+1}^{\infty}\frac{1}{\lambda^{k}}(K_3N +K_4)\le \frac{K_5N}{\lambda^{\Ncerto}} +K_6 \\
&\le  \frac{K_5N}{\lambda^{N-C_5\mathcal{K}\log N}} +K_6 \le K_7h^{K_8\log \lambda}\left(\log\frac{1}{|h|}\right)^{1+C_5\mathcal{K}\log\lambda} + K_6.
\end{align*}
It is left to analyze $S_{1}$. Applying the operator $(I-\PF_{t+h})^{-1}$,
\begin{align*}
(I-\PF_{t+h})^{-1}\left(S_{1}\right)=\frac{1}{h}\sum_{i=0}^{\infty}\PF_{t+h}^i \sum_{k=0}^{\Ncerto}s_{k+1}(t)
\Pi_{t+h}\left(H_{f_{t+h}(f_t^k(c))}- H_{f_{t}(f_t^k(c))}\right).
\end{align*}
Then
\begin{align*}
(I-\PF_{t+h})^{-1}\left(S_{1}\right)&= \frac{1}{h}\sum_{k=0}^{\Ncerto}s_{k+1}(t)\sum_{i=0}^{\infty}\PF_{t+h}^i\Pi_{t+h}\left(H_{f_{t+h}(f_t^k(c))}- H_{f_{t}(f_t^k(c))}\right)\\
=&S_{11}+S_{12}.
\end{align*}
\noindent Where
\begin{align*}
S_{11}=  \sum_{k=0}^{\Ncerto}s_{k+1}(t)\sum_{i=0}^{\Ncerto-k}\frac{1}{h}\PF_{t+h}^i\Pi_{t+h}\left(H_{f_{t+h}(f_t^k(c))}- H_{f_{t}(f_t^k(c))}\right)
\end{align*}
and
\begin{align*}
S_{12}&= \sum_{k=0}^{\Ncerto}s_{k+1}(t)\sum_{i=\Ncerto-k+1}^{\infty}\frac{1}{h}\PF_{t+h}^i\Pi_{t+h}\left(H_{f_{t+h}(f_t^k(c))}- H_{f_{t}(f_t^k(c))}\right)\\
&= \sum_{k=0}^{\Ncerto}s_{k+1}(t)
\frac{1}{h}  \PF_{t+h} \circ (I-\PF_{t+h})^{-1}\circ  \Pi_{t+h} \circ\PF_{t+h}^{\Ncerto-k}  \left(H_{f_{t+h}(f_t^k(c))}- H_{f_{t}(f_t^k(c))}\right).
\end{align*}
\noindent We observe that
\begin{align*}
\norm{S_{12}}_{L^1}\le C \sum_{k=0}^{\Ncerto}|s_{k+1}(t)
| |(I-\PF_{t+h})^{-1}\circ  \Pi_{t+h} \circ\PF_{t+h}^{\Ncerto-k} \frac{1}{h} \left(H_{f_{t+h}(f_t^k(c))}- H_{f_{t}(f_t^k(c))}\right) |_{L^1}.
\end{align*}
Let 
$$
\varphi_k = \frac{1}{h} \PF_{t+h}^{\Ncerto-k} \left(H_{f_{t+h}(f_t^k(c))}- H_{f_{t}(f_t^k(c))}\right),
$$
By Proposition \ref{var_LN3} it follows that 
\begin{align} 
|\varphi_k|_{BV} &= var(\varphi_k) + \norm{\varphi_k}_{L^1}\\
&\le \frac{C}{|h| |Df_{t}^{N_3(t,h)-k}(f^{k+1}_{t}(c))|} + K_1\\
& \leq C \frac{|Df_{t}^{N(t,h)+1}(f_{t}(c))|}{|Df_{t}^{N_3(t,h)-k}(f^{k+1}_{t}(c))|} + K_1\nonumber\\
&\leq C |Df_t^{N(t,h)+1 - N_3(t,h)}(f_t^{N_3(t,h)+1}(c))| | Df_t^{k}(f_t(c)))|+K_1\nonumber \\
&\leq C \Lambda^{N(t,h)+1 - N_3(t,h)+k}+K_1.
\end{align} 
By Lemma \ref{norma_L1} we have 
\begin{align*}
 & |(I-\PF_{t+h})^{-1}\circ  \Pi_{t+h}(\varphi_k) |_{L^1}\\
 &= |(I-\PF_{t+h})^{-1}\circ  \Pi_{t+h} \circ\PF_{t+h}^{\Ncerto-k} \frac{1}{h} \left(H_{f_{t+h}(f_t^k(c))}- H_{f_{t}(f_t^k(c))}\right) |_{L^1}  \\
&\le K'_1 \log (C \Lambda^{N(t,h)+1 - N_3(t,h)+k}+ K_1)  +K'_2\\
&\le K'_1 \log (K_2\Lambda^{N(t,h)+1 - N_3(t,h)+k})  +K'_2\\
&\le K_3(N-N_3+k+1).
\end{align*}

Therefore,
\begin{align*}
\norm{S_{12}}_{L^1}&\le K_3\sum_{k=0}^{\Ncerto}|s_{k+1}(t)|(N -\Ncerto+k+1)\\
&\le K_3(N-\Ncerto)\sum_{k=0}^{\Ncerto}\frac{1}{\lambda^k} + K_3\left(\sum_{k=0}^{\Ncerto}\frac{k}{\lambda^k} +\sum_{k=0}^{\Ncerto}\frac{1}{\lambda^k} \right)\\
&\le K_4\mathcal{K}\log N  +K_{5} \le K\left(\log\log\frac{1}{|h|} + 1\right).
\end{align*}
We proceed to examine $S_{11}$.
\begin{align*}
S_{11}&= \underbrace{\sum_{k=0}^{\Ncerto}s_{k+1}(t)\sum_{i=0}^{\Ncerto-k}\PF_{t+h}^i \left(\frac{H_{f_{t+h}(f_t^k(c))} - H_{f_{t}(f_t^k(c))}}{h} \right)}_{S_{111}}\\
&\underbrace{-\sum_{k=0}^{\Ncerto}s_{k+1}(t)\sum_{i=0}^{N_3-k}\PF_{t+h}^i \left(\rho_{t+h}\int\frac{ H_{f_{t+h}(f_t^k(c))} - H_{f_{t}(f_t^k(c))}}{h} \ dm \right)}_{S_{112}}.
\end{align*}
Note that
\begin{align*}
S_{112}&=-\sum_{k=0}^{\Ncerto}s_{k+1}(t)\sum_{i=0}^{\Ncerto-k}\rho_{t+h}\int\left(\frac{ H_{f_{t+h}(f_t^k(c))} - H_{f_{t}(f_t^k(c))}}{h}\right) dm \\
&=-\sum_{k=0}^{\Ncerto}s_{k+1}(t)\sum_{i=0}^{\Ncerto-k}\left(v_t(f_t^k(c)) + O(h)\right)\rho_{t+h}.
\end{align*}
Adding and subtracting the sum
$$
\sum_{k=0}^{\Ncerto}s_{k+1}(t)\sum_{i=0}^{\Ncerto-k}v_t(f_t^k(c))\rho_{t},
$$
we obtain
$$
S_{112}=S_{1121}+S_{1122},
$$
where
$$
S_{1121}=-\sum_{k=0}^{\Ncerto}s_{k+1}(t)\sum_{i=0}^{\Ncerto-k}v_t(f_t^k(c))\rho_{t}
$$
and
\begin{align*}
S_{1122}&=-(\rho_{t+h}-\rho_{t}) \sum_{k=0}^{\Ncerto}s_{k+1}(t)(\Ncerto-k)v_t(f_t^k(c))- O(h)\sum_{k=0}^{\Ncerto}s_{k+1}(t)(\Ncerto-k)\rho_{t+h}.
\end{align*}
By Eq. (\ref{densidade_L1})
\begin{align*}
\norm{S_{1122}}_{L^1}&\le K_1\sup_t |v_t||h| \log \big( \frac{1}{|h|}\big) \sum_{k=0}^{\Ncerto}|s_{k+1}(t)|(\Ncerto-k)\\
& + \norm{\rho_{t+h}}_{L^1}|O(h)|\sum_{k=0}^{\Ncerto}|s_{k+1}(t)|(\Ncerto-k)\\
&\le \left(K_2|h|\log \frac{1}{|h|}+K_3|O(h)|\right) \Ncerto\sum_{k=0}^{\Ncerto} \frac{1}{\lambda^k}\\
&\le K_4 N\left(|h|\log \frac{1}{|h|} +|O(h)|\right)\le K\log\frac{1}{|h|} \left(|h|\log \big( \frac{1}{|h|}\big)+|O(h)|\right).
\end{align*}
Therefore, taking $\phi:[0,1] \to \R$ a lipschitzian observable,
\begin{align*}
&\int \phi(x) \mathcal{W}(x) \ dx = \int \phi(x) (I-\PF_{t+h})^{-1}\Phi_h(x)\ dx\\
&=\int \phi(x) (S_{111}+S_{1121}) (x)\ dx + O\left(\log \log \frac{1}{|h|}\right)\\
&=\sum_{k=0}^{\Ncerto}s_{k+1}(t)\sum_{i=0}^{\Ncerto-k}\int \phi(x) \PF_{t+h}^i \left(\frac{H_{f_{t+h}(f_t^k(c))} - H_{f_{t}(f_t^k(c))}}{h} \right) (x)\ dx\\
&-\sum_{k=0}^{\Ncerto}s_{k+1}(t)\sum_{i=0}^{\Ncerto-k}v_t(f_t^k(c))\int \phi(x)\rho_{t}(x) \ dx+ O\left(\log \log \frac{1}{|h|}\right).
\end{align*}
By Eq. (\ref{dirac2}) we have 
\begin{align*}
&\int \phi(x) \PF_{t+h}^i \left(\frac{H_{f_{t+h}(f_t^k(c))} - H_{f_{t}(f_t^k(c))}}{h} \right) (x)\ dx \nonumber \\
&= \phi(f_{t}^{i+k+1}(c))v_{t}(f_t^k(c)) +  O( |Df_{t}^i(f_{t}^{k+1}(c))| |h|) \nonumber \\
&= \phi(f_{t}^{i+k+1}(c))v_{t}(f_t^k(c)) +  O( \frac{|Df_{t}^i(f_{t}^{k+1}(c))|}{|Df_{t}^N(f_{t}(c))|}).
\end{align*}
Since
\begin{align}
 & \big| \sum_{k=0}^{\Ncerto}s_{k+1}(t)\sum_{i=0}^{\Ncerto-k} O( \frac{|Df_{t}^i(f_{t}^{k+1}(c))|}{|Df_{t}^N(f_{t}(c))|}) \big| \nonumber \\
&\le K_1\sum_{k=0}^{\Ncerto}\left(\frac{1}{\lambda}\right)^{k}\sum_{i=0}^{N}\left(\frac{1}{\lambda}\right)^{N-i} < K,
\end{align}
it follows that
\begin{align*}
&\int \phi(x) \mathcal{W}(x) dx\\
&=\sum_{k=0}^{\Ncerto}s_{k+1}(t)v_{t}(f_t^k(c))\sum_{i=0}^{\Ncerto-k}\left(\phi(f_{t}^{i+k+1}(c)) -\int \phi \ d\mu_{t}\right) +  O\left(\log\log \frac{1}{|h|}\right) \\
&= \sum_{k=0}^{\Ncerto}s_{k+1}(t)v_t(f_t^k(c))\sum_{j=k+1}^{\Ncerto+1}\left(\phi(f_t^{j}(c)) - \int \phi \ d\mu_t\right) + O\left(\log\log \frac{1}{|h|}\right)\\
&=\sum_{j=1}^{\Ncerto+1}\left(\phi(f_t^{j}(c))-\int \phi \ d\mu_t\right) \sum_{k=0}^{j-1}s_{k+1}(t)v_t(f_t^k(c))+  O\left(\log\log \frac{1}{|h|}\right).
\end{align*}
Adding and subtracting the series
\begin{align*}
\sum_{j=1}^{\Ncerto+1}\left(\phi(f_t^{j}(c))-\int \phi \ d\mu_t\right) \sum_{k=j}^{\infty}s_{k+1}(t)v_t(f_t^k(c)),
\end{align*}
we obtain
\begin{align*}
\int \phi(x) \mathcal{W}(x) dx&
=\sum_{j=1}^{\Ncerto+1}\left(\phi(f_t^{j}(c))-\int \phi \ d\mu_t\right) \sum_{k=0}^{\infty}\frac{s_1(t)}{Df_t^k(f_t(c))}v_t(f_t^k(c))\\
&\underbrace{-\sum_{j=1}^{\Ncerto+1}\left(\phi(f_t^{j}(c))-\int \phi \ d\mu_t\right) \sum_{k=j}^{\infty}\frac{s_1(t)}{Df_t^k(f_t(c))}v_t(f_t^k(c))}_{I_1}\\ 
&+  O\left(\log \log \frac{1}{|h|}\right).
\end{align*}
Note that $|I_1|< \infty$. Indeed,
\begin{align*}
|I_1|&\le K_1\sum_{j=1}^{\Ncerto+1}\left|\phi(f_t^{j}(c))-\int \phi \ d\mu_t\right| \sum_{k=j}^{\infty}\left(\frac{1}{\lambda}\right)^k \\
&\le K_2\sum_{j=1}^{\Ncerto+1}\left(\frac{1}{\lambda}\right)^{j}\le K.
\end{align*}
Therefore,
\begin{align}
\int \phi(x) \mathcal{W}(x) dx &=  s_1(t)J(f_t,v_t)\sum_{j=1}^{\Ncerto+1}\left(\phi(f_t^{j}(c))-\int \phi \ d\mu_t\right) +O\left(\log\log \frac{1}{|h|}\right)\nonumber \\
&= s_1(t)J(f_t,v_t)\sum_{j=0}^{\Ncerto}\left(\phi(f_t^{j}(c))-\int \phi \ d\mu_t\right) +O\left(\log\log \frac{1}{|h|}\right).
\end{align}
\end{proof}


\section{Estimates  for the Tame part}\label{sec_tame}

Let $\nu$ be a signed, finite and borelian   measure  on $[0,1]$. Denote by  $|\nu|$ the  variation measure of $\nu$ and by  $||\nu||$ the total variation of $\nu$. Define the push-forward of  $\nu$ by $f_t$ as  the borelian measure 
$$(f_t^\star \nu)(A)= \nu(f^{-1}_t(A)).$$
Note that for every bounded  borelian function $g\colon[0,1]\rightarrow \mathbb{R}$
$$\int g   \ d(f_t^\star \nu)  = \int  g\circ f_t \ d\nu.$$
It is also easy to see that 
$$|f^\star_{t}\nu|\le f^\star_{t}|\nu|.$$
Suppose that   $\nu$ has the form
\begin{equation}\label{typem}  \nu =   \pi m  + \sum_{x \in \hat{\Delta}} q_x \delta_x,\end{equation} 
where $\pi \in L^\infty(m)$ with support on $[0,1]$, $m$ is the Lebesgue measure, $\hat{\Delta} \subset [0,1]$ is a countable subset, $q_x \in \mathbb{R}$, with 
$$\sum_{x \in \hat{\Delta}} |q_x| < \infty,$$
and $\delta_x$ is the Dirac measure supported on $\{x\}$. 
Then 
$$|\nu| =  |\pi| m  + \sum_{x \in \hat{\Delta}} |q_x| \delta_x,$$
$$||\nu||= |\pi|_{L^1(m)} + \sum_{x \in \hat{\Delta}} |q_x|.$$
Furthermore, $f_t^\star \nu$ has the form
$$f_t^\star \nu = \mathcal{L}_t(\pi) m + \sum_{x \in \hat{\Delta}} q_x \delta_{f_t(x)}. $$

\begin{proposition}\label{lemma_delta} Let $f_t$ be a $C^1$ family of $C^1$ piecewise expanding unimodal maps. Let $\nu$ be a signed, finite and borelian   measure.  Let $\psi_t: [0,1] \mapsto \mathbb{R}$, $t\in [0,1]$ be such that 
$\psi_t \in L^{\infty}(\nu)$ and $t \rightarrow \psi_t$ is a lipschitzian function with respect to  the $L^{\infty}(|\nu|)$ norm, that is, there exists $L$ such that for all $t, h$ we have 
$$|\psi_{t+h}-\psi_t|_{L^\infty(\nu)}\leq L  |h|.$$
Define
$$
\Delta_{t,h}(x) =\int_0^x  \ df_{t+h}^\star (\psi_{t+h} \nu)  -  \int_0^x \ df_{t}^\star (\psi_{t} \nu).
$$
Then there exist positive constants $K_1, K_2$ such that 
$$
|\Delta_{t,h}|_{L^1(m)} \leq (L+K_1K_2) ||\nu|| |h|
$$
for all $t \in [0,1]$, $h$, where 
$$K_1=\sup_{t} |\psi_t|_{L^\infty(\nu)} \text {and } K_2= \sup_{t,x} |\partial_t f_t(x)|.$$
\end{proposition}
\proof
Observe that 
\begin{align*}
\Delta_{t,h}(x)&=\int_0^x  \ df_{t+h}^\star (\psi_{t+h} \nu)  -  \int_0^x \ df_{t}^\star (\psi_{t} \nu)  \\
&= \underbrace{\int_0^x  \ df_{t+h}^\star (\psi_{t+h} \nu)  -  \int_0^x \ df_{t+h}^\star (\psi_{t} \nu) }_{\Delta_1}\\
&+\underbrace{\int_0^x  \ df_{t+h}^\star (\psi_{t} \nu)  -  \int_0^x \ df_{t}^\star (\psi_{t} \nu) }_{\Delta_2}.
\end{align*}
Therefore, 
$$
|\Delta_{t,h}(x)| \le |\Delta_1(x)| +|\Delta_2(x)|.
$$
We first estimate $\Delta_1$.
\begin{align*}
|\Delta_1(x)|& \le \int  \fc_{[0,x]}  \ d|f_{t+h}^\star (\psi_{t+h} \nu - \psi_{t} \nu )|\leq   \int  \fc_{[0,x]}  \ d(f_{t+h}^\star (|\psi_{t+h}  - \psi_{t}|| \nu| ))  \\
&\le  \int  \fc_{[0,x]}\circ f_{t+h}   |\psi_{t+h} - \psi_{t} | \ d|\nu|  \le | \psi_{t+h}- \psi_{t} |_{L^\infty(\nu)} ||\nu|| \leq L ||\nu||  |h|.
\end{align*}
In particular 
$$|\Delta_1|_{L^1(m)}\leq L ||\nu||  |h|.$$
We now estimate $\Delta_2$.
\begin{align*}
\Delta_2(x)& =  \int  \fc_{[0,x]} \ df_{t+h}^\star (\psi_{t} \nu)   - \int \fc_{[0,x]} \ df_{t}^\star (\psi_{t} \nu) \\
&= \int  \fc_{[0,x]}\circ f_{t+h}   \ d(\psi_{t} \nu) - \int  \fc_{[0,x]}\circ f_{t} d(\psi_{t} \nu) \\
&= \int  ( \fc_{f_{t+h}^{-1}([0,x])} -\fc_{f_{t}^{-1}([0,x])} ) \ d(\psi_{t} \nu).
\end{align*}
Therefore,
$$|\Delta_2(x)|\leq \int  | \fc_{f_{t+h}^{-1}([0,x])} -\fc_{f_{t}^{-1}([0,x])} | |\psi_{t} |\ d|\nu| \leq K_1   \int  | \fc_{f_{t+h}^{-1}([0,x])} -\fc_{f_{t}^{-1}([0,x])} | \ d|\nu| $$
where
$$K_1=\sup_{t} |\psi_t|_{L^\infty(\nu)}.$$
By the Fubini's Theorem
\begin{eqnarray}
 |\Delta_2|_{L^1(m)} &\leq& K_1 \int  \int  | \fc_{f_{t+h}^{-1}([0,x])}(y) -\fc_{f_{t}^{-1}([0,x])}(y) | \ d|\nu|(y) \ dm(x) \nonumber \\
                                 &\leq& K_1 \int  \int  | \fc_{f_{t+h}^{-1}([0,x])}(y) -\fc_{f_{t}^{-1}([0,x])}(y) |  \ dm(x) \ d|\nu|(y)
\end{eqnarray} 
Note that
$$ | \fc_{f_{t+h}^{-1}([0,x])}(y) -\fc_{f_{t}^{-1}([0,x])}(y) |= \fc_{U_y}(x),$$
where 
$$U_y = \{ x \in [0,1] \colon     f_{t+h}(y)<    x \leq  f_t(y)   \text { or } f_{t}(y)<   x \leq  f_{t+h}(y) \}.$$
Observe that
$$m(U_y)=| f_{t+h}(y)- f_t(y)|\leq K_2  |h|.$$
Thus,
\begin{eqnarray}  |\Delta_2|_{L^1(m)}&\leq& K_1 \int  \int  \fc_{U_y}(x)  \ dm(x) \ d|\nu|(y) \nonumber \\
&\leq&K_1K_2  ||\nu|| |h|.
 \end{eqnarray} 
\endproof

\begin{remark} To avoid a cumbersome notation, in the Proof of Proposition \ref{parteB} we will use the following notation. Whenever we take the supremum over all $t \in [0,1]$ we actually take the supremum over all $t \in [0,1]$ such that $f_t$ do not have a periodic critical point. And whenever  we take the supremum over all $h \neq 0$ we indeed mean taking the supremum over all $h\neq 0$ such that $0< |h| < \delta$, where $\delta > 0$ is given by Definition \ref{goodf}. 

\end{remark}

\begin{proof}[Proof of Proposition \ref{parteB}] 
We first examine
$$
\frac{1}{h}(\PF_{t+h} \rho_t -\PF_{t} \rho_t).
$$
As we have seen, the density $\rho_t$ can be decomposed as 
$$\rho_t = (\rho_t)_{abs} + (\rho_t)_{sal}.$$
We also have $\PF_{t+h}\rho_t \in BV$ and
$$
\PF_{t+h}\rho_t = (\PF_{t+h}\rho_t)_{abs}+(\PF_{t+h}\rho_t)_{sal}.
$$
Therefore,
$$
(\PF_{t+h} \rho_t -\PF_{t} \rho_t) = \left((\PF_{t+h} \rho_t)_{abs} -(\PF_{t} \rho_t)_{abs}\right)+ \left((\PF_{t+h} \rho_t)_{sal} -(\PF_{t} \rho_t)_{sal}\right).
$$
Let us examine the absolutely continuous term
$$
\frac{1}{h}((\PF_{t+h} \rho_t)_{abs} -(\PF_{t} \rho_t)_{abs}).
$$	
Observe that for every $t$
$$(\PF_{t} \rho_t)(x) = (\PF_{t} \rho_t)_{abs}(x) + (\PF_{t} \rho_t)_{sal}(x).$$
Differentiating with respect to $x$,
\begin{align*}
((\PF_{t} \rho_t)_{abs})'(x)&= (\PF_{t} \rho_t)'(x)  \\
&= ((\PF_{t} \rho_t)')_{abs}(x) + ((\PF_{t} \rho_t)')_{sal}(x).
\end{align*}
Then
\begin{align*}
&(\PF_{t} \rho_t)_{abs}(x) = \int_{0}^{x} (\PF_{t} \rho_t)' \ dm.
\end{align*}
Similarly
$$
(\PF_{t} \rho_{t+h})_{abs}(x) = \int_{0}^{x} (\PF_{t+h} \rho_t)'(y) \ dm.
$$
Therefore, 
\begin{align*}
(\PF_{t+h} \rho_t)_{abs}(x) -(\PF_{t} \rho_t)_{abs}(x) &= \int_{0}^{x} (\PF_{t+h} \rho_t)'  -(\PF_{t} \rho_t)' \ dm\\
&= \int_{0}^{x} ((\PF_{t+h} \rho_t)')_{abs}- ((\PF_{t} \rho_t)')_{abs}  \ dm\\
&+ \int_{0}^{x} ((\PF_{t+h} \rho_t)')_{sal} - ((\PF_{t} \rho_t)')_{sal} \ dm.
\end{align*}
We define
\begin{align}\label{eqA}
A_{t,h}(x) = \int_{0}^{x} ((\PF_{t+h} \rho_{t})')_{abs}- ((\PF_{t} \rho_t)')_{abs} \ dm,
\end{align}
and
\begin{align}\label{eqB}
B_{t,h}(x) = \int_{0}^{x} ((\PF_{t+h} \rho_t)')_{sal}  - ((\PF_{t} \rho_t)')_{sal}  \ dm.
\end{align}
Our goal is to prove that
$$
 \sup_{t \in [0,1]} \sup_{h\neq 0} \norm{\frac{A_{t,h}}{h}}_{BV}<\infty  \mbox{\;\;\; and \;\;\;}  \sup_{t \in [0,1]} \sup_{h\neq 0} \norm{\frac{B_{t,h}}{h}}_{BV}<\infty.
$$
Since  $A_{t,h}$ is absolutely continuous, it follows that
$$
\mbox{var}(A_{t,h})  = \int | A'_{t,h} | dm.
$$
Hence, to prove that 
$$
\sup_{t \in [0,1]} \sup_{h\neq 0} \norm{\frac{A_{t,h}}{h}}_{BV}<\infty,
$$  
it is enough to prove that 
\begin{equation} 
\label{esta} \sup_{t \in [0,1]} \sup_{h\neq 0}  \left|\frac{A'_{t,h}}{h}\right|_{L^1(m)} dm<\infty \mbox{\;\;\; and \;\;\;}\sup_{t \in [0,1]} \sup_{h\neq 0}   \norm{\frac{A_{t,h}}{h}}_{L^1(m)}<\infty.
\end{equation} 

\noindent According to Eq. (\ref{eqA}),
$$
A'_{t,h}(x) =(\PF_{t+h} \rho_t)'_{abs}(x) - (\PF_{t} \rho_t)'_{abs}(x).
$$
Differentiating $(\PF_{t+h}\rho_t)'$, we have, for every $h$, 
$$
(((\PF_{t+h} \rho_t)')_{abs})'(x) = (\PF_{t+h} \rho_t)''(x).
$$
for $m$-almost every $x$. In particular 
$$A''_{t,h}(y)= (\PF_{t+h} \rho_t)''(y) - (\PF_{t} \rho_t)''(y),$$
for $m$-almost every $y$. and 
\begin{align}\label{eqA/h}
A'_{t,h}(x) = \int_{0}^{x} (\PF_{t+h} \rho_t)''  - (\PF_{t} \rho_t)'' \ dm.
\end{align}
As we have seen the Ruelle-Perron-Frobenius operator for $f_{t+h}$ is given by
\begin{align}\label{eqPF1}
(\PF_{t+h} \rho_t)(x) = \sum_{f_{t+h}(y)=x}\frac{\rho_t(y)}{|Df_{t+h}(y)|}.
\end{align}
Differentiating the equation (\ref{eqPF1}) with respect to $x$  we obtain
\begin{align}\label{eqPF2}
(\PF_{t+h} \rho_t)'(x) = \sum_{f_{t+h}(y)=x}\frac{\rho_t'(y)}{Df_{t+h}(y)|Df_{t+h}(y)|} - \frac{\rho_t(y)D^2f_{t+h}(y)}{|Df_{t+h}(y)|^3}.
\end{align}
Now, differentiating the equation (\ref{eqPF2}) with respect to $x$  we obtain
\begin{align*}
(\PF_{t+h} \rho_t)''(x) = &\sum_{f_{t+h}(y)=x}\left(\frac{\rho_t''(y)}{|Df_{t+h}(y)||Df_{t+h}(y)|^2} - 3\frac{\rho_t'(y)D^2f_{t+h}(y)}{|Df_{t+h}(y)|Df_{t+h}(y)^3}\right) \nonumber \\
&+ \sum_{f_{t+h}(y)=x}\left(-\frac{\rho_t(y)D^3f_{t+h}(y)}{|Df_{t+h}(y)|Df_{t+h}(y)^3} + 3\frac{\rho_t(y)(D^2f_{t+h}(y))^2}{|Df_{t+h}(y)||Df_{t+h}(y)|^4}\right).
\end{align*}
Observe that we can rewrite $(\PF_{t+h} \rho_t)''$ as follows
\begin{align}\label{PFrewrite}
(\PF_{t+h} \rho_t)''&= \PF_{t+h}\left( \frac{\rho_t''}{|Df_{t+h}|^2}\right) -3 \PF_{t+h}\left(\frac{\rho_t' D^2f_{t+h}}{(Df_{t+h})^3} \right)  \nonumber\\
&-\PF_{t+h}\left( \frac{\rho_t D^3f_{t+h}}{(Df_{t+h})^3}\right) +3\PF_{t+h}\left( \frac{\rho_t (D^2f_{t+h})^2}{|Df_{t+h}|^4}\right).
\end{align}
We obtain a similar expression for $(\PF_{t} \rho_t)''$.

\noindent Substituting Eq. (\ref{PFrewrite}) into Eq. (\ref{eqA/h}) we obtain
\begin{align*}
A'_{t,h}(x)&= \underbrace{    \int_{0}^{x}  \ d f_{t+h}^\star \left( \frac{\rho_t''}{|Df_{t+h}|^2} m \right) -  \int_{0}^{x}   d f_{t}^\star \left( \frac{\rho_t''}{|Df_{t}|^2} m \right)    }_{A_1}\\
&+    \underbrace{\int_{0}^{x}d f_{t+h}^\star\left(\frac{-3\rho_t' D^2f_{t+h}}{(Df_{t+h})^3} m \right)-  \int_{0}^{x} d f_{t}^\star\left(\frac{-3\rho_t' D^2f_{t}}{(Df_{t})^3}m \right)}_{A_2} \\
&+\underbrace{  \int_{0}^{x}  d f_{t+h}^\star \left( \frac{-\rho_t D^3f_{t+h}}{(Df_{t+h})^3} m\right) -  \int_{0}^{x} d f_{t}^\star\left( \frac{-\rho_t D^3f_{t}}{(Df_{t})^3}m \right) }_{A_3}\\
&+\underbrace{ \int_{0}^{x} d f_{t+h}^\star \left( \frac{3\rho_t (D^2f_{t+h})^2}{|Df_{t+h}|^4} m \right) - \int_{0}^{x} d f_{t}^\star \left( \frac{3\rho_t (D^2f_{t})^2}{|Df_{t}|^4} m \right)}_{A_4}  . 
\end{align*}
Observe that $A_i$, $1\le i\le 4$, satisfy the assumptions of Proposition \ref{lemma_delta} and the total variation of each one of the measures that appears above has a upper bound that depends on the constants in Assumption (V) of Definition \ref{goodf}. Therefore,
$$
\sup_{t \in [0,1]}\sup_{h\neq 0}  \left|  \frac{A'_{t,h}}{h}  \right|_{L^1(m)}< \infty
$$
and, consequently
\begin{equation} \label{estd}
\sup_{t \in [0,1]}\sup_{h\neq 0} \mbox{var}\left(\frac{A_{t,h}}{h}\right)= \sup_{t \in [0,1]}\sup_{h\neq 0} \left|\frac{A'_{t,h}}{h}\right|_{L^1(m)} dm < \infty.
\end{equation} 
It remains to  verify that the second part of Eq. (\ref{esta}). Note that 
$$
\norm{\frac{A_{t,h}}{h}}_{L^1}= \int \left| \frac{A_{t,h}}{h}\right| dm = \int \left| \int_0^x \frac{A'_{t,h}(y)}{h} dy \right| dm \le \left| \frac{A'_{t,h}}{h} \right|_{L^1(m)}.
$$
Hence, by Eq. (\ref{estd}), Eq. (\ref{esta}) holds. Hence, we need to show that  
$$\sup_{t \in [0,1]} \sup_{h\neq 0}   \left| \frac{B_{t,h}}{h} \right|_{BV}<\infty.$$
By Eq. (\ref{eqPF2}) and Property (V) in Definition \ref{goodf}  we have
\begin{align*}
&(\PF_{t+h}\rho_t)'_{sal}(x)\\
&= \sum_{k=1}^{\infty}\left( \frac{s_k'(t)H_{{f_{t+h}}(f_t^k(c))}(x)}{Df_{t+h}(f_t^k(c))|Df_{t+h}(f_t^k(c))|} - \frac{s_k(t)H_{{f_{t+h}}(f_t^k(c))}(x)}{|Df_{t+h}(f_t^k(c))|^3}D^2f_{t+h}(f_t^k(c))\right)\\
&+ \left(\frac{\rho_t'(c)}{Df_{t+h}(c-) |Df_{t+h}(c-)|} + \frac{\rho_t'(c)}{Df_{t+h}(c+) |Df_{t+h}(c+)|}\right)H_{{f_{t+h}}(c)}(x)\\
& - \left(\frac{\rho_t(c)D^2f_{t+h}(c-)}{|Df_{t+h}(c-)|^3} + \frac{\rho_t(c)D^2f_{t+h}(c+)}{|Df_{t+h}(c+)|^3}\right)H_{{f_{t+h}}(c)}(x).
\end{align*}
Since for every $ a \in [0,1]$ we have 
$$H_a(x)= \int_{0}^x \ d(-\delta_a),$$
we can write 
\begin{align*}
B_{t,h}(x)= \int_{0}^x \sum_{i=1}^{4}B_i(y) \ dm(y),
\end{align*}
with functions $B_i$ given by 
\begin{align*}
B_1(x)&= \int_{0}^{x}   \ df_{t+h}^\star\left(     \frac{1}{Df_{t+h} |Df_{t+h}|}   \nu_1  \right) - \int_{0}^{x}  df_{t}^\star\left(     \frac{1}{Df_{t} |Df_{t}|}   \nu_1  \right) 
\end{align*}
where 
$$\nu_1= \sum_{k=1}^{\infty} s_k'(t) (-\delta_{f_t^k(c)}),$$
\begin{align*}
B_2(x)=-\int_{0}^{x}   \ df_{t+h}^\star\left(     \frac{D^2f_{t+h} }{|Df_{t+h}|^3}   \nu_2  \right) + \int_{0}^{x}  df_{t}^\star\left(      \frac{D^2f_{t} }{|Df_{t}|^3}   \nu_2  \right) ,
\end{align*}
where
$$\nu_2= \sum_{k=1}^{\infty} s_k(t) (-\delta_{f_t^k(c)}).$$
Let  $\hat{\psi}$ be the constant borelian function $\hat{\psi} \colon [0,1] \rightarrow \mathbb{R}$ given by 
$$\hat{\psi}_t(y)= \frac{1}{Df_{t}(c-) |Df_{t}(c-)|} + \frac{1}{Df_{t}(c+) |Df_{t}(c+)|}.$$
Then 
\begin{align*}
B_3(x)&=\int_{0}^{x} \ df_{t+h}^\star   ( \hat{\psi}_{t+h} \nu_3) - \int_{0}^{x} \  df_{t}^\star (\hat{\psi}_t \nu_3).
\end{align*}
where
$$\nu_3 = -\rho_t'(c)\delta_c.$$

Let $\tilde{\psi}$ be the constant borelian function  $\tilde{\psi}\colon [0,1] \rightarrow \mathbb{R}$ given by 
$$\tilde{\psi}_t(y)=\frac{D^2f_{t}(c-)}{|Df_{t}(c-)|^3}  +  \frac{D^2f_{t}(c+)}{|Df_{t}(c+)|^3}.$$
then 
\begin{align*}
B_4(x)&=-\int_{0}^{x} \ df_{t+h}^\star   ( \tilde{\psi}_{t+h} \nu_4) + \int_{0}^{x} \  df_{t}^\star (\tilde{\psi}_t \nu_4).
\end{align*}
Here
$$\nu_4 = -\rho_t(c)\delta_c.$$
We can apply Proposition \ref{lemma_delta} on  each one of the pairs $(B_i,\nu_i)$. Moreover, by  property  (V)  of Definition \ref{goodf} there is a upper bound for  the total variation   of the measures $\nu_i$, $i=1,2,3,4,$ that holds for every $t \in [0,1]$. Hence, 
$$\sup_{t \in [0,1]} \sup_{h\neq 0} \left|\frac{B_i}{h} \right|_{L^1(m)} <\infty, $$ 
and consequently 
$$\sup_{t \in [0,1]} \sup_{h\neq 0} \mbox{var} \left(\frac{B_{t,h}}{h} \right) <\infty. $$  
Since
$$
\norm{\frac{B_{t,h}}{h}}_{L^1}= \int \left| \frac{B_{t,h}}{h}\right| dm = \int \left| \int_0^x \sum_{i=1}^{4}\frac{B_i(y)}{h} \ dy \right| dm \le \sum_{i=1}^4 \big| \frac{B_i}{h}\big|_{L^1(m)} ,
$$
we obtain
$$\sup_{t \in [0,1]} \sup_{h\neq 0}  \left|\frac{B_{t,h}}{h}\right|_{BV} <\infty.$$  
Therefore,
\begin{align*}
\sup_{t\in[0,1]} \sup_{h\neq 0} \norm{\frac{(\PF_{t+h} \rho_t)_{abs} -(\PF_{t} \rho_t)_{abs}}{h} }_{BV}<   \infty.
\end{align*}

\noindent It remains to examine the saltus.
\begin{align*}
&\frac{(\PF_{t+h}\rho_t)_{sal} - (\PF_{t}\rho_t)_{sal}}{h}\\ 
&=\underbrace{\frac{1}{h}\sum_{k=1}^{\infty}\left( \frac{s_k(t)}{Df_{t+h}(f_t^k(c))}H_{f_{t+h}(f_t^k(c))}- \frac{s_k(t)}{Df_{t}(f_t^k(c))}H_{f_{t}(f_t^k(c))}\right)}_{\tilde{S}_1}\\
&+\frac{1}{h}\Bigg(\left( \frac{\rho_t(c)}{|Df_{t+h}(c-)|} + \frac{\rho_t(c)}{|Df_{t+h}(c+)|}\right)H_{f_{t+h}(c)} - \left( \frac{\rho_t(c)}{|Df_{t}(c-)|} + \frac{\rho_t(c)}{|Df_{t}(c+)|}\right)H_{f_{t}(c)}\Bigg).
\end{align*}
Let us analyze $\tilde{S}_1$. Notice that 
\begin{align*}
\tilde{S}_1&=\frac{1}{h}\sum_{k=1}^{\infty}\frac{s_k(t)}{Df_{t}(f_t^k(c))}\left(H_{f_{t+h}(f_t^k(c))}- H_{f_{t}(f_t^k(c))}\right)\\
&+\underbrace{\frac{1}{h}\sum_{k=1}^{\infty}\left(\frac{s_k(t)}{Df_{t+h}(f_t^k(c))} - \frac{s_k(t)}{Df_{t}(f_t^k(c))}\right)H_{f_{t+h}(f_t^k(c))}}_{\tilde{S}_{11}}.
\end{align*}
Note that 
\begin{align*}
\norm{\tilde{S}_{11}}_{BV}&\le \frac{1}{|h|}\sum_{k=1}^{\infty}|s_k(t)|\left|\frac{1}{Df_{t+h}(f_t^k(c))} - \frac{1}{Df_{t}(f_t^k(c))}\right|\norm{H_{f_{t+h}(f_t^k(c))}}_{BV}\\
&\le \frac{2}{|h|}\sum_{k=1}^{\infty}|s_k(t)|\frac{|Df_{t+h}(f_t^k(c)) - Df_{t}(f_t^k(c))|}{|Df_{t+h}(f_t^k(c))Df_{t}(f_t^k(c))|} \\
&\le \frac{K_1}{|h|}\sum_{k=1}^{\infty}\frac{1}{|Df_{t}^{k-1}(f_t(c))|}|(\partial_s Df_s(f^k_t(c)))|_{s=\theta_{t,h,k}}||h| \le K.
\end{align*}
Hence, $\sup_h\norm{\tilde{S}_{11}}_{BV}< \infty$. Therefore,
\begin{align*}
&\frac{(\PF_{t+h}\rho_t)_{sal} - (\PF_{t}\rho_t)_{sal}}{h} \\ 
&=\frac{1}{h}\sum_{k=1}^{\infty}\frac{s_k(t)}{Df_{t}(f_t^k(c))}\left(H_{f_{t+h}(f_t^k(c))}- H_{f_{t}(f_t^k(c))}\right)+\tilde{S}_{11}\\
&+\frac{1}{h}\left(\left( \frac{\rho_t(c)}{|Df_{t+h}(c-)|} + \frac{\rho_t(c)}{|Df_{t+h}(c+)|}\right)H_{f_{t+h}(c)} - \left( \frac{\rho_t(c)}{|Df_{t}(c-)|} + \frac{\rho_t(c)}{|Df_{t}(c+)|}\right)H_{f_{t}(c)}\right)\\
&=\underbrace{\frac{1}{h}\sum_{k=0}^{\infty}s_{k+1}(t)\left(H_{f_{t+h}(f_t^k(c))}- H_{f_{t}(f_t^k(c))}\right)}_{\tilde{S}}+\tilde{S}_{11}\\
&+\underbrace{\frac{1}{h}\left(\frac{\rho_t(c)}{|Df_{t+h}(c-)|} - \frac{\rho_t(c)}{|Df_{t}(c-)|}\right)H_{f_{t+h}(c)}}_{\tilde{S}_2}\\
&+ \underbrace{\frac{1}{h}\left(\frac{\rho_t(c)}{|Df_{t+h}(c+)|} - \frac{\rho_t(c)}{|Df_{t}(c+)|}\right)H_{f_{t+h}(c)}}_{\tilde{S}_3}.
\end{align*}
We will analize only $\tilde{S}_2$, the term $\tilde{S}_3$ is analogous.
\begin{align*}
\norm{\tilde{S}_2}_{BV}\le K_1\frac{1}{|h|}\left| \frac{1}{|Df_{t+h}(c-)|} - \frac{1}{|Df_{t}(c-)|}\right|\le \frac{K_2}{|h|}|h| \le K.
\end{align*}
Hence, 
$$\sup_{h\neq 0} \norm{\tilde{S}_2}_{BV}< \infty \text { and } \sup_{h\neq 0} \norm{\tilde{S}_3}_{BV}< \infty.$$
We can write 
\begin{align*}
\frac{\PF_{t+h}(\rho_t)-\PF_{t}(\rho_t)}{h} &= \Pi_{t+h} \left( \frac{\PF_{t+h}(\rho_t)-\PF_{t}(\rho_t)}{h} \right) \\
&= \underbrace{\Pi_{t+h}(\tilde{S})}_{\Phi_h} + \underbrace{\Pi_{t+h} \left( \frac{A}{h} + \frac{B}{h} + \tilde{S}_{11}+\tilde{S}_2+\tilde{S}_3\right)}_{r_h}.
\end{align*}
Therefore 
$$
\int r_h dm =0 \mbox{\;\;and\;\;} \sup_{t \in [0,1]} \sup_{h\neq 0} \norm{r_h}_{BV}<\infty.
$$
This finishes the proof.
\end{proof}

\section{The function $\mathcal{R_{\phi}}$ is not Lipschitz on any subset of positive measure}

We give two interesting and simple  consequences of our main result. They tell us that, under the assumptions of our main result, the function  $\mathcal{R_{\phi}}$ is {\it not} very regular in {\it any} subset of the parameter space with positive Lebesgue measure. This show that there is not way to make $\mathcal{R_{\phi}}$ more regular using  some "parameter exclusion" strategy. 

\begin{cor}\label{cor_main2}
 Under the same assumptions of our main result, for every set $\Omega~\subset~ [a,b]$, with $m(\Omega) > 0$, we have for almost every $t \in \Omega$ 
\begin{equation}\label{ls} 
 \limsup_{h\rightarrow 0+} \frac{\mathcal{R_{\phi}}(t+h)-\mathcal{R_{\phi}}(t)}{h \sqrt{-\log |h|} }\fc_\Omega(t+h)=+\infty
\end{equation}
and 
\begin{equation}\label{li} 
\liminf_{h\rightarrow 0+}    \frac{\mathcal{R_{\phi}}(t+h)-\mathcal{R_{\phi}}(t)}{h \sqrt{-\log |h|} }\fc_\Omega(t+h)=-\infty,
\end{equation}
where $\fc_\Omega$ denotes the indicator function of $\Omega$.
\end{cor}
\begin{proof}
Due Propostition \ref{cut_good}, it is enough to prove  Corollary \ref{cor_main2} for good transversal families. We are going to prove that Eq. (\ref{ls}) holds for almost every $t \in \Omega$. The proof  that  Eq. (\ref{li}) holds for almost every $t \in \Omega$ is similar.  

If Eq. (\ref{ls}) fails for $t$ in a subset of $\Omega$ with positive Lebesgue measure, then there exist $\hat{\Omega}\subset \Omega$, with $m(\hat{\Omega}) > 0$ and $K_1>0$ such that for every $t \in  \hat{\Omega}$ we have 

$$ \limsup_{h\rightarrow 0+} \frac{\mathcal{R_{\phi}}(t+h)-\mathcal{R_{\phi}}(t)}{h \sqrt{-\log |h|} }\fc_\Omega(t+h) \leq  K_1.$$
Since $f_t$ is a good transversal family, by Lemma \ref{inf} we have that  $\inf_t |\Psi(t)|>  0$  and that $\Psi(t)$ does not changes signs on $t \in [a,b]$, so without loss of generality we can assume $\Psi(t) > 0$ for every $t \in [a,b]$ and $\inf_t \Psi(t) >  0$ (otherwise replace the family $f_t$ by $f_{-t}$). So there exists $K_2>0$ such that 
$$
 \limsup_{h\rightarrow 0+}    \frac{\mathcal{R_{\phi}}(t+h)-\mathcal{R_{\phi}}(t)}{\Psi(t) h \sqrt{-\log |h|} }\fc_\Omega(t+h) \leq K_2
$$
for every $t \in \hat{\Omega}$.  Then there exists $h_0 > 0$ and a set  $S\subset \hat{\Omega}$ with $m(S) > 0$  such that for every $t \in S$ we have
$$ \frac{\mathcal{R_{\phi}}(t+h)-\mathcal{R_{\phi}}(t)}{\Psi(t)  h \sqrt{-\log |h|} }\fc_\Omega(t+h) \leq K_2+1$$
for every $h$ satisfying $0< h\leq h_0$.
Let $t_0\in (a,b)$ be a Lebesgue density point of $S$. Choose $\delta > 0$ such that 
$$ \mathcal{D}_{\mathcal{N}}(K_2+1)+\delta < 1.$$
Then for every $\epsilon> 0$ small enough,
$$\frac{m( S \cap I_\epsilon)}{m(I_\epsilon)}  > \mathcal{D}_{\mathcal{N}}(K_2+1)+\delta,$$
where $I_\epsilon=[t_0-\epsilon,t_0+\epsilon]$.  Let $S_\epsilon=S \cap I_\epsilon.$ It is a well-known fact that if $$S_\epsilon-h=\{ t-h\colon t \in S_\epsilon\}$$ then
$$\lim_{h\rightarrow 0} m(S_\epsilon\cap(S_\epsilon-h))=m(S_\epsilon) > 0.$$ 
Note that  for every $ t\in S_\epsilon\cap(S_\epsilon-h)$, we have $t, t+h \in S_\epsilon\subset S\subset \Omega$, then 
$$ \frac{\mathcal{R_{\phi}}(t+h)-\mathcal{R_{\phi}}(t)}{\Psi(t)  h \sqrt{-\log |h|} } \leq K_2+1$$
for every $0< h \leq h_0$. In particular 

\begin{align}  \label{cont}
\limsup_{h\to 0+}  \ &\displaystyle{\frac{1}{m(I_\epsilon)}} \ m\left( t \in I_\epsilon \colon \ \frac{1}{\Psi(t) h \sqrt{-\log |h|}}\frac{\mathcal{R_{\phi}}(t+h)-\mathcal{R_{\phi}}(t)}{h}  \leq  K_2+1 \right)\nonumber \\
&\ge \frac{m(S_\epsilon)}{m(I_\epsilon)} \ge  \mathcal{D}_{\mathcal{N}}(K_2+1)+\delta.
\end{align}
On the other hand the restriction of $f_t$ to the interval  $I_\epsilon$ is a transversal family, then by Theorem \ref{main} we obtain
\begin{align}  \label{cont2}
\lim_{h\to 0+}  \ &\displaystyle{\frac{1}{m(I_\epsilon)}}\ m\left( t \in I_\epsilon \colon \ \frac{1}{\Psi(t) h \sqrt{-\log |h|}}\frac{\mathcal{R_{\phi}}(t+h)-\mathcal{R_{\phi}}(t)}{h}  \leq  K_2+1 \right)\nonumber \\
&=  \mathcal{D}_{\mathcal{N}}(K_2+1),\nonumber
\end{align}
which contradicts Eq.(\ref{cont}).
\end{proof}

\begin{proof}[Proof of Corollary \ref{cor_main}]  It follows from Corollary \ref{cor_main2}.\end{proof}

\begin{remark} In Baladi and Smania \cite{bs1}\cite{bs1b} it is  proven that for almost every $t \in [a,b]$ there exists a sequence $h_n \rightarrow 0$ such that 
$$ \frac{\mathcal{R_{\phi}}(t+h_n)-\mathcal{R_{\phi}}(t)}{h_n}$$
is not bounded.  In particular $\mathcal{R_{\phi}}$ is not a lipschitzian  function on the whole interval $[a,b]$. Naturally Corollaries \ref{cor_main2} and \ref{cor_main} do not  follow from this when $\Omega$ is not an interval. 
\end{remark}

\begin{remark} Two weeks before  this work be completed, Fabi\'an Contreras sent us his Ph. D. Thesis \cite{contreras} where he proves  a result sharper than Corollary \ref{cor_main2} when $\Omega=[a,b]$ and $\phi$ is a $C^1$ generic observable.  He proves that for almost every $t \in [a,b]$  the limit 
\begin{equation}  \lim_{h\rightarrow 0^+}    \frac{\mathcal{R_{\phi}}(t+h)-\mathcal{R_{\phi}}(t)}{h \sqrt{|\log h \log \log |\log h||} }\end{equation}
exists and it is non zero. Note again that Corollaries \ref{cor_main2} and \ref{cor_main} do not seem to follow from his result when $\Omega$ is not an interval.  As in our case, the main difficult is  to reduce the problem to Schnellmann's main  result in \cite{sch2}. We are not completely familiar with his methods, but they seems to be quite different from our approach.
\end{remark}

\section*{Acknowledgments}

We would like to thank the anonymous referees for  their valuable comments and suggestions.



\begin{thebibliography}{10}

\bibitem{baladi}
V.~Baladi.
\newblock On the susceptibility function of piecewise expanding interval maps.
\newblock {\em Comm. Math. Phys.}, 275(3):839--859, 2007.

\bibitem{bs1}
V.~Baladi and D.~Smania.
\newblock Linear response formula for piecewise expanding unimodal maps.
\newblock {\em Nonlinearity}, 21(4):677--711, 2008.

\bibitem{bs2}
V.~Baladi and D.~Smania.
\newblock Smooth deformations of piecewise expanding unimodal maps.
\newblock {\em Discrete Contin. Dyn. Syst.}, 23(3):685--703, 2009.

\bibitem{bs3}
V.~Baladi and D.~Smania.
\newblock Alternative proofs of linear response for piecewise expanding
  unimodal maps.
\newblock {\em Ergodic Theory Dynam. Systems}, 30(1):1--20, 2010.

\bibitem{bs1b}
V.~Baladi and D.~Smania.
\newblock Corrigendum: {L}inear response formula for piecewise expanding
  unimodal maps.
\newblock {\em Nonlinearity}, 25(7):2203--2205, 2012.

\bibitem{bil}
P.~Billingsley.
\newblock {\em Convergence of probability measures}.
\newblock Wiley Series in Probability and Statistics: Probability and
  Statistics. John Wiley \& Sons, Inc., New York, second edition, 1999.
\newblock A Wiley-Interscience Publication.

\bibitem{contreras}
F.~Contreras.
\newblock {\em Regularity of absolutely continuous invariant measures for
  piecewise expanding unimodal maps}.
\newblock Ph.D's Thesis. University of Maryland, 2015.

\bibitem{w}
I.~Karatzas and S.~E. Shreve.
\newblock {\em Brownian motion and stochastic calculus}, volume 113 of {\em
  Graduate Texts in Mathematics}.
\newblock Springer-Verlag, New York, second edition, 1991.

\bibitem{KL3}
G.~Keller.
\newblock Stochastic stability in some chaotic dynamical systems.
\newblock {\em Monatsh. Math.}, 94(4):313--333, 1982.

\bibitem{kellerg}
G.~Keller.
\newblock Generalized bounded variation and applications to piecewise monotonic
  transformations.
\newblock {\em Z. Wahrsch. Verw. Gebiete}, 69(3):461--478, 1985.

\bibitem{kpk}
G.~Keller, P.~J. Howard, and R.~Klages.
\newblock Continuity properties of transport coefficients in simple maps.
\newblock {\em Nonlinearity}, 21(8):1719--1743, 2008.

\bibitem{lasotayorke}
A.~Lasota and J.~A. Yorke.
\newblock On the existence of invariant measures for piecewise monotonic
  transformations.
\newblock {\em Trans. Amer. Math. Soc.}, 186:481--488 (1974), 1973.

\bibitem{ls}
R.~Leplaideur and B.~Saussol.
\newblock Central limit theorem for dimension of {G}ibbs measures in hyperbolic
  dynamics.
\newblock {\em Stoch. Dyn.}, 12(2):1150019, 24, 2012.

\bibitem{mazzo}
M.~Mazzolena.
\newblock {\em Dinamiche espansive unidimensionali: dipendenza della misura
  invariante da un parametro}.
\newblock Master's Thesis Roma 2, Tor Vergata, 2007.

\bibitem{philipp}
W.~Philipp and W.~Stout.
\newblock Almost sure invariance principles for partial sums of weakly
  dependent random variables.
\newblock {\em Mem. Amer. Math. Soc. 2}, (issue 2, 161):iv+140, 1975.

\bibitem{ruelle1}
D.~Ruelle.
\newblock Differentiation of {SRB} states.
\newblock {\em Comm. Math. Phys.}, 187(1):227--241, 1997.

\bibitem{ruelle2}
D.~Ruelle.
\newblock Nonequilibrium statistical mechanics near equilibrium: computing
  higher-order terms.
\newblock {\em Nonlinearity}, 11(1):5--18, 1998.

\bibitem{sch1}
D.~Schnellmann.
\newblock Typical points for one-parameter families of piecewise expanding maps
  of the interval.
\newblock {\em Discrete Contin. Dyn. Syst.}, 31(3):877--911, 2011.

\bibitem{sch2}
D.~Schnellmann.
\newblock Law of iterated logarithm and invariance principle for one-parameter
  families of interval maps.
\newblock {\em Probability Theory and Related Fields}, pages 1--45, 2014.

\bibitem{Tsujii}
M.~Tsujii.
\newblock A simple proof for monotonicity of entropy in the quadratic family.
\newblock {\em Ergodic Theory Dynam. Systems}, 20(3):925--933, 2000.

\bibitem{viana}
M.~Viana.
\newblock {\em Sthochastic dynamics of deterministic systems}.
\newblock Lecture notes XXI Col{\'o}quio Brasileiro de Matem{\'a}tica, IMPA,
  Rio de Janeiro, 1997.

\end{thebibliography}
\end{document}